\documentclass[oneside,english]{amsart}
\usepackage[T1]{fontenc}
\usepackage[latin9]{inputenc}
\usepackage{verbatim}
\usepackage{amstext}
\usepackage{amsthm}
\usepackage{amssymb}
\usepackage{xcolor}
\usepackage[shortlabels]{enumitem}
\makeatletter
\numberwithin{equation}{section}
\numberwithin{figure}{section}
\theoremstyle{plain}
\newtheorem{thm}{\protect\theoremname}[section]
\newtheorem*{thm*}{\protect\theoremname}
\theoremstyle{definition}
\newtheorem{defn}[thm]{\protect\definitionname}
\theoremstyle{remark}
\newtheorem{rem}[thm]{\protect\remarkname}
\theoremstyle{plain}
\newtheorem{lem}[thm]{\protect\lemmaname}
\theoremstyle{definition}
\newtheorem{problem}[thm]{\protect\problemname}
\newtheorem{claim}[thm]{\protect\claimname}
\theoremstyle{definition}
\newtheorem{example}[thm]{\protect\examplename}
\theoremstyle{plain}
\newtheorem{fact}[thm]{\protect\factname}
\theoremstyle{plain}
\newtheorem{prop}[thm]{\protect\propositionname}
\newenvironment{lyxlist}[1]
	{\begin{list}{}
		{\settowidth{\labelwidth}{#1}
		 \setlength{\leftmargin}{\labelwidth}
		 \addtolength{\leftmargin}{\labelsep}
		 }}
	{\end{list}}
\theoremstyle{plain}
\newtheorem{cor}[thm]{\protect\corollaryname}
\theoremstyle{plain}
\newtheorem{conjecture}[thm]{\protect\conjecturename}

\makeatother

\usepackage{babel}
\providecommand{\conjecturename}{Conjecture}
\providecommand{\corollaryname}{Corollary}
\providecommand{\definitionname}{Definition}
\providecommand{\examplename}{Example}
\providecommand{\factname}{Fact}
\providecommand{\lemmaname}{Lemma}
\providecommand{\problemname}{Problem}
\providecommand{\claimname}{Claim}
\providecommand{\propositionname}{Proposition}
\providecommand{\remarkname}{Remark}
\providecommand{\theoremname}{Theorem}

\begin{document}
\def\Ind#1#2{#1\setbox0=\hbox{$#1x$}\kern\wd0\hbox to 0pt{\hss$#1\mid$\hss}
\lower.9\ht0\hbox to 0pt{\hss$#1\smile$\hss}\kern\wd0}
\def\Notind#1#2{#1\setbox0=\hbox{$#1x$}\kern\wd0\hbox to 0pt{\mathchardef
\nn="3236\hss$#1\nn$\kern1.4\wd0\hss}\hbox to 0pt{\hss$#1\mid$\hss}\lower.9\ht0
\hbox to 0pt{\hss$#1\smile$\hss}\kern\wd0}
\def\indi{\mathop{\mathpalette\Ind{}}}
\def\nindi{\mathop{\mathpalette\Notind{}}}
\def\bdd {bdd}

\global\long\def\acl{\operatorname{acl}}%

\global\long\def\Avg{\operatorname{Avg}}%

\global\long\def\Sk{\operatorname{Sk}}%

\global\long\def\inp{\operatorname{inp}}%

\global\long\def\ch{\operatorname{char}}%

\global\long\def\dprk{\operatorname{dprk}}%

\global\long\def\ring{\operatorname{ring}}%

\global\long\def\ind{\operatorname{\indi}}%

\global\long\def\nind{\operatorname{\nindi}}%

\global\long\def\ist{\operatorname{ist}}%

\global\long\def\card{\operatorname{Card}^{\ast}}%

\global\long\def\Aut{\operatorname{Aut}}%
\global\long\def\Div{\operatorname{div}}%

\global\long\def\M{\operatorname{\mathbb{M}}}%

\global\long\def\NTP{\operatorname{NTP}}%
\global\long\def\eq{\operatorname{eq}}%

\global\long\def\NIP{\operatorname{NIP}}%
\global\long\def\mod{\operatorname{mod}}%

\global\long\def\IP{\operatorname{IP}}%

\global\long\def\PSL{\operatorname{PSL}}%

\global\long\def\TP{\operatorname{TP}}%

\global\long\def\SL{\operatorname{SL}}%

\global\long\def\fam{\operatorname{fam}}%

\global\long\def\tp{\operatorname{tp}}%

\global\long\def\tr{\operatorname{tr}}%

\global\long\def\st{\operatorname{st}}%

\global\long\def\NSOP{\operatorname{NSOP}}%

\global\long\def\bdn{\operatorname{bdn}}%

\global\long\def\dom{\operatorname{dom}}%

\global\long\def\sf{\operatorname{sf}}%

\global\long\def\Sh{\operatorname{Sh}}%

\global\long\def\ded{\operatorname{ded}}%

\global\long\def\Av{\operatorname{Av}}%

\global\long\def\Pr{\operatorname{Pr}}%

\global\long\def\UF{\operatorname{\mathfrak{U}}}%

\global\long\def\set{\operatorname{set}}%

\global\long\def\FH{\operatorname{FH}}%

\global\long\def\VFA{\operatorname{VFA}}%

\global\long\def\ACFA{\operatorname{ACFA}}%

\global\long\def\NT{\operatorname{NT}}%

\global\long\def\cof{\operatorname{cf}}%

\global\long\def\Th{\operatorname{Th}}%

\global\long\def\meas{\operatorname{meas}}%

\global\long\def\SU{\operatorname{SU}}%

\global\long\def\dcl{\operatorname{dcl}}%

\global\long\def\Cons{\operatorname{Cons}}%

\global\long\def\FHP{\operatorname{FHP}}%
\global\long\def\WFHP{\operatorname{WFHP}}%
\global\long\def\DFHP{\operatorname{DFHP}}%

\global\long\def\VC{\operatorname{VC}}%

\global\long\def\vc{\operatorname{vc}}%

\global\long\def\md{\operatorname{mod}}%

\global\long\def\Sqf{\operatorname{Sqf}}%

\global\long\def\ULCFS{\operatorname{ULCFS}}%

\global\long\def\rv{\operatorname{rv}}%
\global\long\def\ird{\operatorname{ird}}%

\global\long\def\ac{\operatorname{ac}}%

\global\long\def\val{\operatorname{v}}%
\global\long\def\opD{\operatorname{opD}}%

\global\long\def\RV{\operatorname{RV}}%

\global\long\def\WD{\operatorname{WD}}%

\global\long\def\Cons{\operatorname{Cons}}%

\global\long\def\fin{\operatorname{fin}}%

\global\long\def\fap{\operatorname{fap}}%

\global\long\def\cl{\operatorname{cl}}%

\title[Fractional Helly property and forking in $\NTP_{2}$
theories]{Fractional Helly property and combinatorics of forking in $\NTP_{2}$
theories}
\author{Artem Chernikov and Chuyin Jiang}
\begin{abstract}
We investigate the class of FHP theories, i.e.~theories of structures 
in which all definable families of sets satisfy the Fractional Helly
Property (and its variants) from combinatorics. FHP theories generalize
NIP and form a new subclass of low $\NTP_{2}$ theories. We give
many new examples (including ultraproducts of finite fields and 
of the $p$-adics) and establish some results about forking and $f$-generics
for amenable groups definable in $\FHP$ theories. We make several
conjectures about finitary combinatorial properties of forking in $\NTP_{2}$ theories
and establish some partial results, as well as investigate related two-cardinal type counting functions addressing a question of Adler.
\end{abstract}

\maketitle

\tableofcontents

\section{Introduction}

The \emph{fractional Helly theorem} is a basic compactness principle in discrete geometry: if a positive proportion of the small subfamilies of a finite family of sets have non-empty intersection, then one can find a large intersecting subfamily (see Section \ref{sec: basic FHP and Shelah classification}). It was initially established for convex subsets of $\mathbb{R}^d$ by \cite{katchalski1979problem} (with optimal bounds established in \cite{kalai1984intersection, eckhoff1985upper}), and since then extended to many other families of sets of geometric or combinatorial interest. We refer to excellent surveys \cite{amenta2017helly, de2019discrete, barany2022helly} for a detailed discussion and references.

In particular, for families of finite VC-dimension, Matoušek \cite{matousek2004bounded} showed that an analogous conclusion still holds, and this viewpoint has played an important role in the interaction between combinatorial geometry, VC-theory, and model theory of NIP structures (in particular through its application to the so called ``$(p,q)$-theorem'' for families of finite VC dimension - see Section \ref{sec: p,q theorem and NIP}). The aim of this paper is to study the model-theoretic content of the fractional Helly phenomenon beyond the $\NIP$ context and to use it as a tool in the study of forking in $\NTP_{2}$ theories.

Namely, we say that a partitioned formula $\varphi(x,y) \in L$ has the fractional Helly property (FHP) in a structure $M$ if the definable family $\{\varphi(M,b):b\in M^{y}\}$ of subsets of $M^x$ defined by its instances satisfies a fractional Helly theorem; and a complete theory $T = \Th(M)$ is $\FHP$ if every partitioned formula is. In Section \ref{sec: basic FHP and Shelah classification} we consider basic properties of $\FHP$ formulas and theories, along with some variants of the property). We place $\FHP$ theories inside Shelah's classification hierarchy: $\FHP$ theories form a proper subclass of low $\NTP_2$ theories containing NIP theories (Propositions  \ref{prop: FHP implies NTP2} and \ref{prop: local FHP implies low}); and show  that $\FHP$ for a theory reduces to checking that all formulas $\varphi(x,y)$ with $x$ singleton are $\FHP$ (Lemma \ref{lem: Basic operations preserving FH}).

In Section \ref{sec: FHP and measures} we formulate a relative version of the $\FHP$ property with respect to a class of measures (where the original property corresponds to the class of finitely supported measures, Proposition  \ref{rem: FHP iff FHP for finite measures}), and connect it with the theory of generically stable Keisler measures in NIP theories. In particular, in Section \ref{sec: FHP vs HrushPilSim} we demonstrate that Matoušek's theorem (FHP for families of finite VC dimension) implies a uniform/local version of the main result about generically stable measures from Hrushovski-Pillay-Simon \cite{hrushovski2012note}; and conversely its qualitative version under the global NIP assumption follows from \cite{hrushovski2012note} by compactness.

In Section \ref{subsec: Colorful-fractional-Helly and dp-rank} we establish a connection between the so-called ``colorful'' fractional Helly property from combinatorics and the notion of burden/dp-rank from the study of NIP/$\NTP_2$ theories.  In particular, generalizing and refining a measure theoretic characterization of strong NIP theories by Pillay \cite{pillay2013weight}, we get (which can also be viewed as a multi-measure generalization of \cite{hrushovski2012note}):
\begin{thm*}[Proposition \ref{prop: colorful FHP meas bdd by burden}]
Assume that $\bdn\left(\M^x\right)<k$ and let $\varphi_{i}\left(x,y_{i}\right) \in L$,
$i \in [k]$, satisfy $\FHP$ relatively to a class of definable measures $\mathfrak{M}_{i} \subseteq \mathfrak{M}_{y_i}(\M)$ (see Definition \ref{def: FHP for a class of measures}). Then for every $\alpha>0$ there is
$\gamma>0$ satisfying the following. Let $\mu_{i}\in\mathfrak{M}_{i}$ be such that $\mu_{1},\ldots,\mu_{k}$ are pairwise commuting (in particular each $\mu_i$ commutes with itself).
 Assume that $\mu_{1}\otimes\ldots\otimes\mu_{k}\left(\exists x\bigwedge_{i=1}^{k}\varphi_{i}\left(x,y_{i}\right)\right)\geq\alpha$.
Then there is some $i\in [k]$ and some $a\in \M^{x}$
such that $\mu_{i}\left(\varphi_{i}\left(a,y_{i}\right)\right)\geq\gamma$.
\end{thm*}

\noindent In particular, this implies that families of finite VC-dimension satisfy colorful fractional Helly property (Corollary \ref{cor: colorful Matousek}) and gives a bound on the fractional Helly number of formulas in terms of the burden/dp-rank (rather than the dual VC-density, as in Matoušek's theorem; see  the discussion after Remark \ref{rem: generalizing Pillay}):
\begin{thm*}[Corollary \ref{cor: FHPk bounded by burden}]
In any FHP (so e.g.~in NIP) theory $T$, the fractional Helly number of a formula $\varphi\left(x,y\right)$ is at
most $\bdn\left(\M^{x} \right) +1$.
\end{thm*}

\noindent In Section \ref{sec: p,q theorem and NIP} we consider one of the main applications of the fractional Helly property, Matousek's $(p,q)$-theorem for families of finite VC-dimension (Fact \ref{fac: pq Matousek}, which is an analog of the Alon and Kleitman's $(p,q)$-theorem for convex sets in $\mathbb{R}^d$ \cite{alon1992piercing}). Matousek's result plays an important role in the study of $\NIP$ theories (see the references there). 
As we demonstrate in this paper, the class of FHP structures  is much wider than the class of NIP structures.  However, at the level of the theory, the
$\left(p,q\right)$-theorem characterizes NIP/finite VC dimension:

\begin{thm*}[Proposition \ref{prop: pierceable implies NIP}] Assume that the 
formula $\psi\left(x;y_{1}, y_{2}\right) := \varphi\left(x,y_{1}\right)\land\neg\varphi\left(x,y_{2}\right)$ satisfies the $(p,q)$-theorem. Then  $\varphi\left(x,y\right)$ is NIP (i.e.~its instances define a family of sets of finite VC-dimension).
\end{thm*}

In Section \ref{sec: f-gen and forking FHP} turn to definable groups and forking. Several notions of largeness/genericity for definable sets (and their equivariant versions in definable groups), coming from combinatorics, topological dynamics or measure theory, play an important role in the model-theoretic study of tame classes of structures. In definably amenable NIP groups all of these notions agree, giving a canonical notion of a large set (\cite{chernikov2018definably}, see Section \ref{sec: notions of genericity} for the details). Here we extend the connection of forking and invariant measures to amenable FHP groups:
\begin{thm*}[Theorem \ref{thm: generics in FHP}] Assume $T$ is  FHP and $G = G(\M)$ is a definable group so that $G(M)$ is amenable (as a discrete group) for some $M \models T$. Then for any $L(\M)$-definable set $X \subseteq G(\M)$, $X$ is $f$-generic if and only if $\mu(X) > 0$ for some $G$-invariant measure $\mu$.
\end{thm*}
\noindent In Section \ref{sec: fork iff null in FHP} we note a partial analog for the action of $\Aut(\M)$.

The so called fsg groups (groups with finitely satisfiable generics), arising in the work on Pillay's $o$-minimal groups conjecture \cite{hrushovski2008groups}, form a particularly nice class of definably amenable NIP groups capturing definable compactness in many natural settings (including $o$-minimal theories).  In \cite{chernikov2024definable}, a generalization of fsg groups from NIP theories to \emph{fim groups} (see Section \ref{sec: fim groups}) in arbitrary theories was proposed, demonstrating that part of the theory of fsg groups in NIP theories survives. Here we show that in FHP theories, the characterization of generic definable sets generalizes fully from the NIP case:
\begin{thm*}[Proposition \ref{prop : generics in fim}]
	Let $G$ be a definable fim group in an $\FHP$ theory. Then all notions of genericity (1)--(5) in Definition \ref{defn: notions of genericity} are equivalent for definable subsets of $G$.
\end{thm*}

In Section \ref{sec: FHP in some expansions of Z} we consider the fractional Helly property for definable sets in two (simple, unstable) expansions of $(\mathbb{Z},+)$: by a predicate $\Pr$ for the primes and their inverses (studied in \cite{kaplan2017decidability}, see Section \ref{subsec: T_pr}), and by a predicate $\Sqf$  for the square-free integers (studied in \cite{tran}, see Section \ref{sec: square free is FHP}). They exhibit two quite different behaviors,  explained by the former set being of Banach density $0$, while the latter is of positive Banach density:
\begin{thm*}[Theorem \ref{thm: Sqf is FHP}] The   structure $\left(\mathbb{Z},+,\Sqf\right)$ is $\FHP$. Moreover, every formula $\varphi(x,y)$ with $|x| \leq d$ satisfies $\FHP_{d+1}$.
\end{thm*}

\begin{thm*} The structure $\left(\mathbb{Z},+,\Pr \right)$ is not $\FHP$ (Proposition \ref{prop: Primes not FHP}), and  (assuming Dickson's conjecture) $T_{\Pr}$
is locally $\FHP$ (Theorem \ref{prop: Tpr has local FHP}). 
\end{thm*}

In Section \ref{sec: MS-meas struc FHP} we consider $\FHP$ in measurable structures in the sense of Macpherson and Steinhorn, or \emph{MS-measurable structures} \cite{macpherson2008one}. Main examples of MS-measurable structures are ultraproducts of finite fields, finite simple groups of bounded Lie rank, vector spaces, etc. (we refer to \cite[Example 2.4]{elwes2008survey} for further examples).
\begin{thm*}[Theorem \ref{thm: FHP in MS meas}]
	Let $M$ be an MS-measurable structure. Then every partitioned formula $\varphi(x,y) \in L$ with $|x| \leq d$ satisfies $\FHP_{d+1}$  with respect to the class of definable measures $\mathfrak{M}_y := \{\mu_{B}(y) : B \subseteq \M^y \textrm{ definable with parameters}\}$ (see Definition \ref{def: FHP for a class of measures}). In particular, $\varphi(x,y)$ satisfies $\FHP_{d+1}$.
\end{thm*}

\noindent As an application, we get that definable families of sets of bounded description complexity in large finite fields  satisfy the fractional Helly property:
\begin{thm*}[see Corollary \ref{cor: FHP in finite fields} for the precise statement]
 For every $D \in \mathbb{N}$ and $\alpha > 0$ there exist $\beta = \beta(D, \alpha) >0$ so that: if $F$  is a sufficiently large finite field and $\mathcal{F} \subseteq F^d$ a definable family of sets of description complexity $\leq D$ so that
$\left \lvert \left\{ I \subseteq \mathcal{F} : |I| = d+1 \  \land \  \bigcap_{S \in I} S \neq \emptyset \right\} \right \rvert \geq \alpha {|\mathcal{F}| \choose d+1}$, then $\bigcap_{S \in J} S \neq \emptyset$ for some $J \subseteq \mathcal{F}$ with $|J| \geq \beta |\mathcal{F}|$.
\end{thm*}

In Section \ref{sec: ultraprod of Qp} we consider the fractional Helly property for definable families of sets  in valued fields (we note that for special definable families, namely convex sets in the sense of Monna, fractional Helly property was studied in \cite{chernikov2023combinatorial}). We prove an Ax-Kochen-Ershov style result for the FHP property in  henselian valued fields:
\begin{thm*}(Theorem \ref{thm: AKE for FHP})
	Let $K$ be an equi-characteristic $0$ henselian valued field. Then $K$ satisfies $\FHP$ if and only if both the residue field $k$ and the (ordered) value group $\Gamma$ satisfy $\FHP$.
\end{thm*}

\noindent Combining this with the aforementioned result for FHP in pseudo-finite fields, the existing burden calculations \cite{chernikov2010indiscernible, chernikov2014theories, chernikov2014valued, chernikov2016henselian} and Corollary \ref{cor: FHPk bounded by burden}, we obtain the following explicit bounds for the ultraproducts of the $p$-adics (which are $\NTP_2$ but not NIP, so Matoušek's theorem does not apply):
\begin{thm*}(Corollary \ref{cor: FHP in ultraprod Qp})
Let $K$ be $\prod_{i \in \mathbb{N}} \mathbb{Q}_{p_i} / \mathcal{U}$ or $\prod_{i \in \mathbb{N}} \mathbb{F}_{p_i}((t)) / \mathcal{U}$   for some prime $p_i$ and a non-principal ultrafilter $\mathcal{U}$ on $\mathbb{N}$.  Then every partitioned formula $\varphi(x,y) \in L$ in $K$  with $|x| \leq d$ satisfies $\FHP_{2^d}$.
\end{thm*}
\noindent We expect the optimal bound to be $\FHP_{d+1}$ here (Conjecture \ref{conj: FHP in ultra p-adics}).

The \emph{UDTFS} property, or \emph{Uniform Definability of Types over Finite Sets} (conjectured by Laskowski, and established in \cite{chernikov2013externally, chernikov2015externally}), suggests that in NIP theories local types over finite sets behave similarly to stable theories (see Section \ref{sec: UCLFS in NIP}). It plays an important role in  the study of NIP theories and provides a model theoretic counterpart for the existence of compression schemes for families of finite VC-dimension in theoretical machine learning. In this section we propose a generalization of this conjecture for $\NTP_2$ theories, aiming to capture that, when restricted to finite sets in a uniform manner, dividing in $\NTP_2$ theories behaves similarly to dividing in simple theories. We call this property the \emph{Uniform Local Character (of dividing) over Finite Sets}, or \emph{ULCFS} (Section \ref{sec: ULCFS and counting over finite sets}). We show that both simple (Proposition \ref{lem: every type does not fork over a small set}) and NIP theories (Proposition \ref{prop: NIP implies ULCFS}) satisfy (strong) ULCFS, and that ULCFS implies $\NTP_2$ (Proposition \ref{prop: str ULCFS implies resilient}). We also discuss resilience in Section \ref{sec: resilience}, showing that it reduces to formulas in one variable (Corollary \ref{C:Resilience-One-Variable}) and is implied by strong ULCFS (Proposition \ref{prop: str ULCFS implies resilient}).  We conjecture that all $\NTP_2$ theories satisfy ULCFS (Conjecture \ref{conj: NTP2 implies ULCFS}).  In Section \ref{sec: fin type count}, inspired by the classical infinitary two-cardinal function considered in Section \ref{sec: two card count part types}, we consider a two-parameter function $f_{\varphi}(k,l)$ with $k \leq l \in \mathbb{N}$ counting the number of pairwise-inconsistent partial $\varphi$-types of size $k$ over a set of parameters of size $l$ (which in the case $k=l$ corresponds to the dual shatter function $\pi_{\varphi}^{\ast}$). 
In the same way as UDTFS is a strengthening of the Sauer-Shelah lemma, we view ULCFS as a strengthening of $f_{\varphi}(k,l)$ being bounded by a polynomial $p(l)$ of degree independent of $k$, for $l \gg k$ (Definition \ref{def: f_phi(k,l) and poly bdd}). We observe that indeed its polynomial boundedness implies $\NTP_2$ (Proposition \ref{prop: f poly bdd implies NTP2 and low}) and follows from $\ULCFS$ (Proposition \ref{prop: ULCF implies f is poly bdd}). Again, we conjecture that $\NTP_2$ implies polynomial boundedness of $f_{\varphi}(k,l)$ (Conjecture \ref{conj: NTP2 iff counting types}) and obtain a partial result that at least $\NTP_2$ implies $\varepsilon$-power saving (so  $\NTP_2$ is characterized by the non-maximality of $f_{\varphi}(k,l)$,  Proposition \ref{prop: NTP2 iff power saving})

In Section \ref{sec: two card count part types} we consider a related two-cardinal partial type counting function: for a theory $T$ and infinite cardinals $\kappa\leq\lambda$, we let $f_{T}\left(\kappa,\lambda\right)$
be the supremum of the cardinalities $\left|P\right|$, where $P$ is a
family of pairwise inconsistent partial types each of cardinality $\leq\kappa$, all  over the same fixed set of parameters of size $\lambda$. A celebrated result of Keisler \cite{keisler1976six}, refining earlier work of Shelah \cite{shelah1971stability} and Morley \cite{morley1965categoricity}, demonstrates that restricting to the case $\kappa=\lambda$, there are exactly six possibilities for $f_T$ when $T$ is a complete countable theory. Shelah also proved that simplicity of the theory $T$ is detected by $f_T(\kappa,\lambda)$. Following this, Adler \cite{AdlerBanff} conjectured (see Conjecture \ref{conj: Adler}) that $\NTP_2$ can be detected by $f_T$, and that there are only finitely many possibilities for $f_T$ when $T$ is countable. In Proposition \ref{prop: almost all counting functions} we significantly narrow down the possibilities for $f_T(\kappa,\lambda)$ (and conjecture that this is a complete list of possibilities). In particular, our result refutes the former conjecture of Adler (Corollary \ref{cor: counterex to Adler conj}), and confirms the latter one under the GCH assumption (Corollary \ref{cor: Adler under GCH}).

\subsection*{Acknowledgements}

We thank Gabe Conant, James Freitag, Martin Hils, Itay Kaplan, Dugald Macpherson, Alex Mennen, Daniel Palacin, Anand Pillay, Saharon Shelah, Pierre Simon, Chieu Minh Tran and Frank Wagner for very helpful conversations on the topics of this paper throughout the years. 

This work started a long time ago when Chernikov was a postdoc at the Hebrew
University of Jerusalem, supported by Hrushovski's European Research Council under the
European Unions Seventh Framework Programme (FP7/2007-2013)/ERC Grant agreement no.~291111. He was also supported by the NSF Research Grant DMS-2246598.

\section{Basic properties of $\protect\FHP$ formulas and relation to Shelah's
classification}\label{sec: basic FHP and Shelah classification}

Our notation is mostly standard. Given $n\in\mathbb{N}$, we write
$\left[n\right]$ to denote the set $\left\{ 1,2,\ldots,n\right\} $. We let $x,y,z, \ldots$ denote finite tuples of variables, and if $M$ is a first-order structure then $M^x$ denotes the corresponding sort of $M$. For $a,b \in M^x$ and $A \subseteq M$, we write $a \equiv_A b$ if $\tp(a/A) = \tp(b/A)$.

\subsection{Basic properties of $\FHP$ formulas}

	Given a set $X$ and $k\in\mathbb{N}$, we let ${X \choose k}$ denote
the set of all $k$-element subsets of $X$. Given a tuple $\bar{S}=\left(S_{i}:i\in I\right)$
of subsets of $X$ (possibly with repetitions) and $k\in\mathbb{N}$, let $\Cons_{k}\left(\bar{S}\right):=\left\{ J\in{I \choose k}:\bigcap_{i\in J}S_{i}\neq\emptyset\right\} $. We recall the fractional Helly property from combinatorics (see e.g.~\cite{alon2002transversal}):
\begin{defn}
\label{def: FHP}Let $X$ be a set and let $\mathcal{F}\subseteq\mathcal{P}\left(X\right)$
be a (possibly infinite) family of subsets of $X$.
\begin{enumerate}
\item For $k\in\mathbb{N}$ and $\alpha, \beta \in\mathbb{R}_{>0}$, the family $\mathcal{F}$ satisfies $\FHP(k,\alpha,\beta)$ (where FHP stands for \emph{Fractional Helly Property}) if the following holds: for any $n\in\mathbb{N}$ and any tuple  $\bar{S} = (S_{1},\ldots,S_{n})$ of sets in $\mathcal{F}$
(possibly with repetitions), if 
$\left|\Cons_{k}\left(\bar{S}\right) \right|\geq\alpha{n \choose k}$, 
then there is some $J\subseteq\left[n\right]$ with $\left|J\right|\geq\beta n$
such that $\bigcap_{i\in J}S_{i}\neq\emptyset$.
\item For $k\in\mathbb{N}$, we say that $\mathcal{F}$ satisfies $\FHP_k$
if for every
$\alpha\in\mathbb{R}_{>0}$ there is some $\beta\in\mathbb{R}_{>0}$ so that $\mathcal{F}$ satisfies $\FHP(k,\alpha,\beta)$.
\item We say that $\mathcal{F}$ satisfies $\FHP$ if it satisfies $\FHP_k$
for some $k\in\mathbb{N}$. The smallest $k$ so that $\mathcal{F}$ satisfies $\FHP_k$ is the \emph{fractional Helly number} of $\mathcal{F}$.
\end{enumerate}
\end{defn}

We will also consider the following weakening:
\begin{defn}\label{def: weak FHP comb}
	\begin{enumerate}
	\item 
A family of sets $\mathcal{F} \subseteq \mathcal{P}(X)$  satisfies the \emph{$(p,k)$-property}, where $p \geq k \in \mathbb{N}$, if for any tuple $\bar{S} = (S_1, \ldots, S_p)$ of sets in $\mathcal{F}$ there is some $I \subseteq [p], |I| = k$ so that $\bigcap_{i \in I}S_i \neq \emptyset$.

		\item For $k,p\in\mathbb{N}$ and $\beta \in\mathbb{R}_{>0}$, the family $\mathcal{F}$ satisfies $\WFHP(k,p,\beta)$ (\emph{Weak Fractional Helly Property}) if the following holds: for any $n\in\mathbb{N}$ and any tuple  $\bar{S} = (S_{1},\ldots,S_{n})$ of sets in $\mathcal{F}$, if $\{S_1, \ldots, S_n\}$ satisfies the $(p,k)$-property, 
then there is some $J\subseteq\left[n\right]$ with $\left|J\right|\geq\beta n$
such that $\bigcap_{i\in J}S_{i}\neq\emptyset$.
\item  For $k\in\mathbb{N}$, we say that $\mathcal{F}$ satisfies $\WFHP_k$
if there exist some $p_0  \geq k \in \mathbb{N}$ so that: for every $p \geq p_0$ there is  $\beta\in\mathbb{R}_{>0}$ so that $\mathcal{F}$ satisfies $\WFHP(k,p,\beta)$; and  $\mathcal{F}$ satisfies $\WFHP$ if it satisfies $\WFHP_k$ for some $k \in \mathbb{N}$.
	\end{enumerate}
\end{defn}

\noindent The following two remarks are (rephrased) from e.g.~\cite{alon1992piercing, alon2002transversal}:
\begin{prop}\label{prop: FHP implies WFHP}
	If $\mathcal{F}$ satisfies $\FHP_k$, then it also satisfies $\WFHP_k$.
\end{prop}
\begin{proof}
	Assume  $\mathcal{F}$ satisfies the $(p,k)$-property for some $p$, and let $\bar{S}=\left(S_{1},\ldots,S_{n}\right)$ be any finite tuple
of sets from $\mathcal{F}$.
As every $p$-tuple of sets from $\bar{S}$ contains at least one
$k$-tuple with a non-empty intersection, and each $k$-tuple is contained
in ${n-k \choose p-k}$ of the $p$-tuples, by double counting there
are at least 
$\frac{{n \choose p}}{{n-k \choose p-k}}$
 $k$-tuples with non-empty intersections. Taking $p_0 = p_0(k) \in \mathbb{N}$ sufficiently large, and assuming that $p \geq p_0$ and $\alpha=\alpha\left(p,k\right)>0$ sufficiently small, this gives at least $\alpha{n \choose k}$  $k$-tuples with non-empty intersections.
As $\mathcal{F}$ satisfies $\FHP_{k}$, there is some $\beta=\beta\left(\alpha\right) = \beta(p,k) >0$
and $J\subseteq\left[n\right]$ with $\left|J\right|\geq\beta n$
and $\bigcap_{i\in J}S_{i}\neq\emptyset$.
\end{proof}

\begin{rem}
	We could have defined the $(p,k)$-property requiring instead that the sets appearing in $\bar{S}$ in Definition \ref{def: weak FHP comb} are pairwise distinct --- this would not change the definition of $\WFHP_k$: if $\mathcal{F}$ satisfies the $(p,k)$-property without repetitions, then for $p':=k\left(p-1\right)+1$ it satisfies the $(p',k)$-property with repetitions.  Indeed, let $\bar{S}=\left(S_{1},\ldots,S_{n}\right)$ be any finite tuple 
of sets from $\mathcal{F}$, possibly with repetitions. Let
$p':=k\left(p-1\right)+1$. Then for any $I\subseteq\left[n\right],\left|I\right|=p'$
there is some $J\subseteq I$ with $\left|J\right|=k$ and $\bigcap_{i\in J}S_{i}\neq\emptyset$.
Indeed, either $J$ contains $k$ copies of the same (non-empty) set
which clearly intersect, or $p$ \emph{distinct} sets, in which case
some $k$ among them have a non-empty intersection by assumption.

\end{rem}

\noindent We also note that the assumption on the family in the $\WFHP$ property  is strictly stronger than the assumption in the $\FHP$ property:
\begin{lem}
\label{lem:robust-counterexample}
Fix integers $k\ge 2$ and $p'\ge k'\ge 2$, and reals $\alpha\in(0,1)$ and
$\gamma\in(0,1]$.
Then there exist arbitrarily large $n \in \mathbb{N}$ and 
$\{S_1,\dots,S_n\}\subseteq \mathcal{P}(X)$ such that:
\begin{enumerate}
  \item More than an $\alpha$-fraction of the $k$-subfamilies intersect, i.e.
  \[
    \bigl|\{\,I\subseteq [n]: |I|=k \text{ and } \bigcap_{i\in I} S_i\neq\emptyset\,\}\bigr|
    \;>\;\alpha\binom{n}{k}.
  \]
  \item For every $J\subseteq [n]$ with $|J|\ge \gamma n$, the induced subfamily
  $\{S_i:i\in J\}$ fails the $(p',k')$-property.
  \end{enumerate}
\end{lem}

\begin{proof}

Choose a sufficiently large integer $r\ge k$ so  that
\begin{equation}
\label{eq:choose-r}
  \prod_{j=0}^{k-1}\left(1-\frac{j}{r}\right) \;>\; \alpha.
\end{equation}
Next choose any integer $m\ge \lceil p'/\gamma\rceil$ and set $n:=rm$. Partition $[n]$ into $r$ disjoint ``blocks''
$B_1,\dots,B_r$, each of size $m$. Let
\[
  X
  \;:=\;
  \bigl\{\,E\subseteq [n]:
  |E|=k \text{ and } |E\cap B_t|\le 1 \text{ for every } t\in\{1,\dots,r\}\,\bigr\},
\]
and for each $i\in[n]$ define $S_i \;:=\; \{\,E\in X : i\in E\,\}$.

For any $I\subseteq[n]$ with $|I|=k$ we claim
\begin{equation}
\label{eq:intersection-criterion}
  \bigcap_{i\in I} S_i \neq \emptyset
  \quad\Longleftrightarrow\quad
  I \text{ meets $k$ distinct blocks.}
\end{equation}
Indeed, if $I$ has one element in each of $k$ distinct blocks, then $I\in X$
and $I\in S_i$ for all $i\in I$, so the intersection contains $I$.
Conversely, if $I$ contains two elements from the same block $B_t$, then no
$E\in X$ can contain both, hence $\bigcap_{i\in I}S_i=\emptyset$.

By \eqref{eq:intersection-criterion}, the intersecting $k$-tuples are obtained
by choosing $k$ blocks and then one element from each chosen block. Thus
\begin{gather*}
	\bigl|\{\,I\subseteq [n]: |I|=k,\ \bigcap_{i\in I}S_i\neq\emptyset\,\}\bigr|
  \;=\;
  \binom{r}{k}\,m^k, \textrm{ and using } n=rm, \ \frac{\binom{r}{k}m^k}{\binom{n}{k}} =\\
 = \frac{r(r-1)\cdots(r-k+1)\,m^k}{rm(rm-1)\cdots(rm-k+1)}
>
  \frac{r(r-1)\cdots(r-k+1)\,m^k}{(rm)^k}
  =
  \prod_{j=0}^{k-1}\left(1-\frac{j}{r}\right).
\end{gather*}

\noindent Together with \eqref{eq:choose-r} this gives (1).

Let now $J\subseteq[n]$ satisfy $|J|\ge \gamma n=\gamma rm$.
Then some block $B_t$ contains at least $|J\cap B_t|\ \ge\ \frac{|J|}{r}\ \ge\ \gamma m\ \ge\ p'$ 
elements of $J$. Choose distinct indices $i_1,\dots,i_{p'}\in J\cap B_t$. If $a\neq b \in [p']$, then $S_{i_a}\cap S_{i_b}=\emptyset$, since any
$E\in X$ contains at most one element from the block $B_t$. In particular, for
every $L\subseteq\{1,\dots,p'\}$ with $|L|=k'\ge 2$ we have
$\bigcap_{\ell\in L} S_{i_\ell}=\emptyset$. This witnesses that the induced
subfamily $\{S_i:i\in J\}$ fails the $(p',k')$-property, proving~(2).
\end{proof}

We now specialize these notions to definable families of sets:
\begin{defn}
	\begin{enumerate}
		\item Let $\mathcal{M} = (M, \ldots)$ be a first-order structure in a language $\mathcal{L}$.
We say that a partitioned formula $\varphi\left(x,y\right)\in\mathcal{L}$
(with $x,y$ arbitrary finite tuples of variables) satisfies $\FHP(k,\alpha,\beta)$
($\FHP_{k}$, $\FHP$, etc.) in $\mathcal{M}$ if the corresponding family of definable
sets $\mathcal{F}_{\varphi}:=\left\{ \varphi\left(M,b\right):b\in M^{y}\right\} \subseteq\mathcal{P}\left(M^{x}\right)$
does (where $\varphi\left(M,b\right) = \{ a \in M^x : M \models \varphi(a,b) \}$).

\item An $\mathcal{L}$-theory $T$ is $\FHP$ ($\WFHP$) if every partitioned formula $\varphi(x,y)$ with $x,y$ arbitrary finite tuples of variables satisfies $\FHP$ (respectively, $\WFHP$) in every model of $T$.
\end{enumerate}
\end{defn}
\begin{rem}
	Note that if $\varphi(x,y)$ satisfies $\FHP(k, \alpha, \beta)$ in $\mathcal{M}$ and $\mathcal{N} \equiv \mathcal{M}$, then $\varphi(x,y)$ satisfies $\FHP(k, \alpha, \beta)$ in $\mathcal{N}$, hence for a complete theory it suffices to verify FHP in a single model (and similarly for $\WFHP$).
\end{rem}

Recall that, for $A \subseteq \mathbb{M} \models T$ and $b \in \mathbb{M}^y$, a formula $\varphi(x,b)$ \emph{$k$-divides over $A$} if there is an infinite sequence $(b_i : i \in \mathbb{N})$ in $\mathbb{M}^y$ with $b_i \equiv_A b$ so that the set of formulas $\{\varphi(x,b_i) : i \in \mathbb{N}\}$ is $k$-inconsistent; and $\varphi(x,b)$ \emph{divides over $A$} if it $k$-divides over $A$ for some $k$. We refer to e.g.~\cite{chernikov2012forking} for the basic properties of dividing. The following was observed in the proof of  \cite[Proposition 25]{chernikov2015externally}, and is straightforward by Ramsey and compactness:
\begin{rem}\label{rem: non-div fam iff pk}
	Given a formula $\varphi(x,y)$, $k \in \mathbb{N}$ and a partial type $\pi(y)$ over a small set of parameters $A$, the family $\mathcal{F}_{\varphi, \pi} := \{\varphi(\mathbb{M},b) : b \models \pi(y)\}$ satisfies the $(p,k)$-property for some $p$ if and only if for every $b \models \pi$, $\varphi(x,b)$ does not $k$-divide.
\end{rem}

\begin{defn}\label{def: local FHP etc}
Let $T$ be a complete theory, $\mathbb{M} \models T$ a monster model and $A \subseteq \mathbb{M}$ a small set of parameters.
\begin{enumerate}
	\item  A partitioned formula $\varphi\left(x,y\right)$ satisfies \emph{local
$\FHP_k$ over $A$}   (respectively, \emph{local $\WFHP_k$ over $A$}) if for every complete type $q\left(y\right)\in S_{y}\left(A\right)$, the corresponding family
of definable sets $\mathcal{F}_{\varphi,q} := \left\{ \varphi\left(\M,b\right):b\in\M^{y},b\models q\right\} $
satisfies $\FHP_k$ (respectively, $\WFHP_k$). A theory $T$ is \emph{locally FHP} (\emph{locally $\WFHP$}) if every formula satisfies local $\FHP_k$ (local $\WFHP_k$) over $\emptyset$ for some $k = k(\varphi) \in \mathbb{N}$ (then $\FHP$ holds over all sets of parameters). 
\item A partitioned formula $\varphi\left(x,y\right)$ satisfies \emph{$\FHP_k$ for dividing}, or \emph{$\DFHP_k$}, over $A$ if the following holds: for any $q \in S_y(A)$ so that $\varphi(x,b)$ does not $k$-divide over $A$ for some/any $b \models q$, there is some $\beta = \beta(q) \in \mathbb{R}_{>0}$ so that for any finite tuple $B$ of realizations on $q$ there is some $B_0 \subseteq B, |B_0| \geq \beta |B|$ with $\left\{ \varphi(x,b) : b \in B_0 \right\}$ consistent. And $\varphi\left(x,y\right)$ satisfies $\DFHP$ over $A$ if it satisfies $\DFHP_k$ over $A$ for some $k$.
\end{enumerate}
\end{defn}

\begin{rem}
	\begin{enumerate}
	
	\item Immediately from the definitions we have: if $\varphi(x,y)$ satisfies $\FHP$ ($\WFHP$), then it also satisfies local $\FHP$ (respectively, local $\WFHP$) over all sets of parameters.  The converse
is true if $T$ is $\omega$-categorical (see Proposition \ref{prop: FHP is locFHP in omega cat}),
but not in general (see Example \ref{exa: fhp vs local fhp and wnfcp}).
	\item If $\varphi(x,y)$ satisfies local $\FHP$ over $A$, then it also satisfies local $\WFHP$ over $A$ (as in Proposition \ref{prop: FHP implies WFHP}).
		\item If $\varphi(x,y)$ satisfies local $\WFHP_k$ over $A$, then it also satisfies $\DFHP_k$ over $A$. Indeed, assume $\varphi(x,y)$ satisfies local $\WFHP_k$ over $A$. If $\varphi(x,b)$ does not $k$-divide over $A$, then by Remark \ref{rem: non-div fam iff pk} the family $\mathcal{F}_{\varphi, \tp(b/A)}$ satisfies the $(p,k)$-property for some $p \geq k$, without loss of generality $p \geq p_0$.  As by assumption $\varphi$ satisfies $\WFHP(k, p, \beta)$ over $A$ for some $\beta > 0$, we conclude.
		\item If $T$ is $\FHP$ ($\WFHP$), then any reduct of $T$ is also
$\FHP$ (respectively, $\WFHP$).
	\end{enumerate}
\end{rem}

\begin{lem}
\label{lem: Basic operations preserving FH}

\begin{enumerate}
\item If $k \leq k'$ and $\varphi\left(x,y\right)$ satisfies (local) $\FHP_k$,
then it satisfies (local) $\FHP_{k'}$ as well. And the same for (local) $\WFHP_k$ and $\DFHP_k$.
\item Both classes of formulas with $\FHP$ and with local $\FHP$ are closed
under disjunctions. More precisely, if $\varphi_{i}\left(x,y_{i}\right)$
satisfies (local) $\FHP_{k_{i}}$, for $1\leq i\leq t$, then $\varphi\left(x;y_{1},\ldots,y_{t}\right) := \bigvee_{i=1}^{t}\varphi_{i}\left(x,y_{i}\right)$
satisfies (local) $\FHP_k$ for  $k:=\sum_{1\leq i\leq t}k_{i}-t+1$.
\item Given a complete theory $T$, if every formula $\varphi\left(x,y\right)$
with $\left|x\right|=1$ singleton satisfies $\FHP$, then every formula
satisfies $\FHP$ (and if every formula $\varphi(x,y)$ with $|x|=1$ satisfies $\FHP_k$, then every formula $\psi(x',y)$ with $|x'| = n$ satisfies $\FHP_{k^n}$).
\item If $T$ satisfies $\FHP$, then any reduct of $T$ and $T^{\eq}$ satisfies $\FHP$.
\end{enumerate}
\end{lem}

\begin{proof}
(1) Let $\varphi\left(x,y\right)$ and $k<k'$ be given. Fix
an arbitrary $\alpha>0$. Let $B \subseteq \mathbb{M}^y$ be a finite set, $n=\left|B\right|$,
$\bar{S} := \left( \varphi\left(M,b\right):b\in B\right)$ such that
$\Cons_{k'}\left(\bar{S}\right)\geq\alpha{n \choose k'}$. Obviously
every $J \in\Cons_{k'}\left(\bar{S}\right)$ contains some $J'\subseteq J$
with $J'\in\Cons_{k}\left(\bar{S}\right)$, and for every $J'\in{B \choose k}$
there are at most ${n-k \choose k'-k}$ sets $J\in{B \choose k'}$
that contain it. Hence, by double counting and basic properties of
the binomial coefficients we have that 
\[
\left|\Cons_{k}\left(\bar{S}\right)\right|\geq\frac{\alpha{n \choose k'}}{{n-k \choose k'-k}}\geq\alpha'{n \choose k}
\]
for some $\alpha'=\alpha'\left(\alpha,k,k'\right)>0$ and all $n\in\mathbb{N}$.
Then taking $\beta>0$ so that $\varphi\left(x,y\right)$ satisfies  $\FHP(k, \alpha', \beta)$, we get that $\varphi(x,y)$ also satisfies $\FHP(k', \alpha, \beta)$. The proof for local $\FHP$ is identical, restricting to $B$ a tuple of realizations of a complete type $q(y)$. For (local) $\WFHP_k$, we only need to note that any family satisfying the $(p,k')$-property also satisfies the $(p,k)$-property. For $\DFHP$, note that if $\varphi(x,b)$ does not $k'$-divide over $A$, then it also does not $k$-divide over $A$.

(2) Let $\varphi\left(x,\bar{y}\right)$, $\bar{y} = (y_1, \ldots, y_t)$ and $k$ be as in the statement,
and fix some $\alpha>0$. Let $B$ be a finite set of tuples $\bar{b}=\left(b_{1},\ldots,b_{t}\right)$,
$n:=\left|B\right|$, and let $\bar{S}:=\left( \varphi\left(\mathbb{M},\bar{b}\right):\bar{b}\in B\right) $
be such that $\Cons_{k}\left(\bar{S}\right)\geq\alpha{n \choose k}$.
For $1\leq i\leq t$, let $\bar{S}_{i}:=\left( \varphi_{i}\left(\mathbb{M},b_{i}\right):\bar{b}\in B\right) $.
By the choice of $k$, for each $J\in\Cons_{k}\left(\bar{S}\right)$
there exists some $1\leq i_{J}\leq t$ and some $J'\in{J \choose k_{i_J}}$
such that $J'\in\Cons_{k_{i_J}}\left(\bar{S}_{i_J}\right)$. Hence
by pigeonhole there is some $1\leq i\leq t$ and some $\mathcal{D}\subseteq\Cons_{k}\left(\bar{S}\right),\left|\mathcal{D}\right|\geq\frac{1}{t}\left|\Cons_{k}\left(\bar{S}\right)\right|$
such that for every $J\in\mathcal{D}$, $i_{J}=i$. By double counting
as in (1), this implies that there is some $\alpha'_{i}=\alpha'\left(k,k_{i},\alpha\right)>0$
such that $\left|\Cons_{k_{i}}\left(\bar{S}_{i}\right)\right|\geq\frac{\alpha{n \choose k}}{{n-k_{i} \choose k-k_{i}}}\geq\alpha'_{i}{n \choose k_{i}}$
holds for all $n$. By $\FHP_{k_{i}}$ for $\varphi_{i}(x,y_i)$, there is some
$\beta_{i}=\beta_{i}\left(\alpha_{i}'\right)>0$ and $R\subseteq B,\left|R\right|\geq\beta_{i}n$
such that $\bigwedge_{\bar{b}\in R}\varphi_{i}\left(x,b_{i}\right)$
is consistent, hence $\bigwedge_{\bar{b}\in R}\varphi\left(x,\bar{b}\right)$
is consistent. Thus taking $\beta:=\min\left\{ \beta_{i}:1\leq i\leq t\right\} $
does the job. 

For local $\FHP_k$, note that if all $\bar{b}\in B$
have the same type, then for each $1\leq i\leq t$, all of the elements
in $\left\{ b_{i}:\bar{b}\in B\right\} $ also have the same type,
so the proof goes through.

(3) We prove it by induction on the length of $\left|x\right|$. Let
$\varphi\left(x_{1},x_{2};y\right)$ be given, and assume that $\FHP$
holds for all formulas $\psi\left(x,y\right)$ with $\left|x\right|<\left|x_{1}\right|+\left|x_{2}\right|$.
Fix some $\alpha>0$. Let $k_{1},k_{2}\in\mathbb{N}$ be arbitrary,
and let $k=k_{1}\times k_{2}$. Fix an arbitrary set $A\subseteq M_{y}$
of size $n$, let $\mathcal{F}:=\left\{ \varphi\left(x_{1},x_{2};a\right):a\in A\right\} $.
Let $B:={A \choose k_{2}}$, $\psi\left(x_{2};y_{1},\ldots,y_{k_{2}}\right):=\exists x_{1}\bigwedge_{1\leq i\leq k_{2}}\varphi\left(x_{1},x_{2};y_{i}\right)$
and let $\mathcal{F}'=\left\{ \psi\left(x_{1};\bar{a}\right):\bar{a}\in B\right\} \mbox{.}$ 

Assume that $\left|\Cons_{k}\left(\mathcal{F}\right)\right|\geq\alpha{n \choose k}$.
For every $S\in{A \choose k}$ pick some presentation of it as a disjoint
union of $k_{1}$-many subsets of $S$ of size $k_{2}$. This defines
an injection from the set $\Cons_{k}\left(\mathcal{F}\right)$ to
$\Cons_{k_{1}}\left(\mathcal{F}'\right)$, so $\left|\Cons_{k_{1}}\left(\mathcal{F}'\right)\right|\geq\alpha{n \choose k}\geq\frac{\alpha}{k^{k}}n^{k}$.
We also have $\left|\mathcal{F}'\right|={n \choose k_{2}}\leq\frac{n^{k_{2}}}{k_{2}!}$,
and so 
\[
\left|{\mathcal{F}' \choose k_{1}}\right|\leq{\frac{n^{k_{2}}}{k_{2}!} \choose k_{1}}\leq\frac{\left(\frac{n^{k_{2}k_{1}}}{\left(k_{2}!\right)^{k_{1}}}\right)}{k_{1}!}\leq\frac{1}{k_{1}!\left(k_{2}!\right)^{k_{1}}}n^{k}\leq\alpha'n^{k}
\]
for some $\alpha'=\alpha'\left(k_{1},k_{2}\right)>0$ holds for all
$n$. Combining we thus have that $\left|\Cons_{k_{1}}\left(\mathcal{F}'\right)\right|\geq\alpha''{\left|\mathcal{F}'\right| \choose k_{1}}$
holds for some $\alpha''=\alpha''\left(\alpha,k_{1},k_{2}\right)>0$
and all $n$. So, taking $k_{1}$ such that $\psi\left(x_{2},\bar{y}\right)$
satisfies $\FHP_{k_{1}}$ (exists by the inductive assumption), there
is some $\beta'=\beta\left(k_{1},\alpha''\right)>0$ such that there
is some $B_{0}\subseteq B,\left|B_{0}\right|\geq\beta'\left|B\right|$
such that $\bigwedge_{\bar{a}\in B}\psi\left(x_{2},\bar{a}\right)$
is consistent, say realized by some $b\in M_{x_{2}}$. Now consider
a new partitioned formula $\theta\left(x_{1};x_{2},y\right):=\varphi\left(x_{1},x_{2},y\right)$,
and a set of parameters $A'=\left\{ b\right\} \times A\subseteq M_{x_{2}y}$
of size $n$, and a family $\mathcal{F}''=\left\{ \varphi\left(x_{1};b,a\right):a\in A\right\} $.
By the choice of $b$ we have $\left|\Cons_{k_{2}}\left(\mathcal{F}''\right)\right|\geq\beta'\left|B\right|\geq\beta'{n \choose k_{2}}=\beta'\left|{\mathcal{F}'' \choose k_{2}}\right|$.
Taking $\alpha''':=\beta'>0$ and $k_{2}$ such that $\theta\left(x_{1};x_{2},y\right)$
satisfies $\FHP_{k_{2}}$ (again, exists by the inductive assumption),
there is some $\beta=\beta\left(\alpha''',k_{2}\right)>0$ and some
consistent subtuple $\mathcal{F}_{0}\subseteq\mathcal{F}''$ of size
$\geq\beta n$,  say it is realized by $c$. But then the tuple
$\left(c,b\right)\in M_{x_{1}x_{2}}$ realizes the family $\mathcal{F}^{*}=\left\{ \varphi\left(x_{1},x_{2};a\right):\varphi\left(x_{1},b,a\right)\in\mathcal{F}_{0}\right\} $
and $\left|\mathcal{F}*\right|\geq\beta n$. Unwinding we have that
the choice of $\beta$ only depended on $\varphi$, $\alpha$ and $k_{1},k_{2}$,
which shows that $\varphi\left(x_{1},x_{2};y\right)$ satisfies $\FHP_{k}$,
as wanted.

(4) Immediate from the definitions.
\end{proof}

\begin{rem}
	The bound in Lemma \ref{lem: Basic operations preserving FH}(2) is 
optimal. Consider the
model companion of the theory of two linear orders $\leq_{1},\leq_{2}$.
Let $\varphi_{i}\left(x;yy'\right):=y\leq_{i}x\leq_{i}y'$. Then $\varphi_{i}\left(x;yy'\right)$
satisfies $\FHP_{2}$ for $i\in\left\{ 1,2\right\} $, but it is not
hard to see that $\varphi\left(x;y_{1}y_{2}y_{1}'y_{2}'\right)=\bigvee_{i=1}^{2}\varphi_{i}\left(x,y_{i}y_{i}'\right)$
does not by taking two families of disjoint intervals for each of $\leq_{i}$
(it satisfies $\FHP_{3}$, however).
\end{rem}

We will use the following lemmas frequently.
\begin{lem}
\label{lem: FHP for finite sets}Let $\mathcal{F}$ be a family of
subsets of $X$, and assume that there is some $d\in\mathbb{N}$ such
that $\left|S\right|\leq d$ for all $S\in\mathcal{F}$. Then $\mathcal{F}$
satisfies $\FHP_2$.
\end{lem}

\begin{proof}
Fix $\alpha>0$, and let $\left(S_{i}:1\leq i\leq n\right)$ be a
tuple of sets from $\mathcal{F}$. Let $C:=\left\{ I\subseteq\left[n\right]:\left|I\right|=2,\bigcap_{i\in I}S_{i}\neq\emptyset\right\} $
and assume that $\left|C\right|\geq\alpha{n \choose 2}$. The by pigeonhole
there must be some $i^{*}\in\left[n\right]$ such that $i^{*}\in I$
for at least $\frac{\alpha}{2}n$ of the sets $I\in C$. But then
there must be some element $a\in S_{i^{*}}$ belonging to at least
$\frac{\alpha}{2d}n$ of the sets $S_{i},i\in\left[n\right]$, so
take $\beta:=\frac{\alpha}{2d}>0$.
\end{proof}

We will use the following fact from combinatorics (see e.g.~the introduction
of \cite{furedi1983finite}).
\begin{fact}
\label{fact: furedi}\cite{erdos1968coloring} For each $k\in\mathbb{N}$,
there is $\gamma=\gamma\left(k\right)>0$ satisfying the following
(can take $\gamma=\frac{k!}{k^{k}}$).

If $\mathcal{F}$ is a set of $k$-element subsets of a set $X$,
then there exist sets $X_{1},\ldots,X_{k}\subseteq X$ and $\mathcal{F}'\subseteq\mathcal{F}$
such that:
\begin{enumerate}
\item $X_{i}\cap X_{j}=\emptyset$ for all $1\leq i<j\leq k$,
\item $\left|\mathcal{F}'\right|\geq\gamma\left|\mathcal{F}\right|$,
\item $\left|X_{i}\cap S\right|=1$ for all $S\in\mathcal{F}'$ and $1\leq i\leq k$.
\end{enumerate}
\end{fact}

We need a slightly more general version of it.
\begin{lem}
\label{lem: Furedi repetitions}The conclusion of Fact \ref{fact: furedi}
holds even if $\mathcal{F}$ is a multiset (i.e.~we allow repetitions
of $k$-subsets from $X$).
\end{lem}

\begin{proof}
Let $k$ be given. We take $\gamma>0$ to be as given by the fact
for $k+1$.

Now let $X$ be a set and let $\mathcal{F}=\left(S_{i}:1\leq i\leq n\right)$
be a tuple of $k$-subsets of $X$, possibly with repetitions. Let
$X'\supseteq X$ be a new set obtained from $X$ by adding a new element
$b_{i}$ for each $1\leq i\leq n$. Let $S'_{i}:=S_{i}\cup\left\{ b_{i}\right\} $
be a $\left(k+1\right)$-element subsets of $X'$, and let $\mathcal{F}'=\left\{ S_{i}':1\leq i\leq n\right\} $,
then $\left|\mathcal{F}'\right|=\left|\mathcal{F}\right|$ and all
sets in $\mathcal{F}'$ are pairwise distinct. By Fact \ref{fact: furedi},
we find some disjoint sets $X_{1},\ldots,X_{k+1}\subseteq X'$ and
some $I\subseteq\left[n\right]$ with $\left|I\right|\geq\gamma n$
such that for any $i\in I$ we have $\left|S_{i}'\cap X_{j}\right|=1$
for all $1\leq j\leq k+1$. By pigeonhole, there is some $I'\subseteq I,\left|I'\right|\geq\frac{\gamma}{k+1}n$
and $1\leq j^{*}\leq k+1$ such that $b_{i}\in X_{j^{*}}$ for all
$i\in I'$. Let $X_{j}':=X\cap X_{j}$. Then $\left(X_{j}':1\leq j\leq k+1,j\neq j^{*}\right)$,
$\gamma':=\frac{\gamma}{k+1}$ and $\mathcal{F}_{0}:=\left\{ S_{i}:i\in I'\right\} $
satisfy the conclusion.
\end{proof}

The following is easy to verify from the definition of $\FHP$:
\begin{rem}\label{rem: FHP pres by func/1-var conj}
	Assume $\psi(x,z)$ has $\FHP$, $g: \M^y \to \M^z$ is a definable function and $\rho(y)$ is an arbitrary formula. Then the formula $\varphi(x,y) := \psi(x, g(y)) \land \rho(y)$ is also $\FHP$.
\end{rem}

\subsection{$\FHP$ in Shelah's classification}

Next we discuss the position of $\FHP$ theories in Shelah's classification
hierarchy \cite{MR513226}. We recall the definition of some relevant tree properties, and refer to e.g.~\cite{chernikov2014theories} or \cite{chernikov2016model} 
for further details.

\begin{defn}\label{def: NTP2}
Suppose $T$ is a complete theory and $\varphi(x;y) \in L$ is a partitioned formula in the language of $T$ (with $x,y$ tuples of variables).
\begin{enumerate}
\item $\varphi(x;y)$ has the \emph{$k$-tree property} ($k$-$\TP$) if there is a tree of tuples $(a_{\eta})_{\eta \in \omega^{<\omega}}$ in $\M$ such that
\begin{itemize}
\item for all $\eta \in \omega^{\omega}$, $\{\varphi(x;a_{\eta | \alpha}) : \alpha < \omega\}$ is consistent,
\item for all $\eta \in \omega^{<\omega}$, $\{\varphi(x;a_{\eta \frown \langle i \rangle}) : i < \omega\}$ is $k$-inconsistent.  
\end{itemize}
We say that $\varphi(x,y)$ has the tree property ($\TP$) if it has the $k$-tree property for some $k \in \omega$; otherwise we say that $\varphi(x,y)$ is $\NTP$.
\item $\varphi(x;y)$ has the \emph{tree property of the first kind} ($\TP_1$) if there is a tree of tuples $(a_{\eta})_{\eta \in \omega^{<\omega}}$ in $\M$ such that
\begin{itemize}
\item for all $\eta \in \omega^{\omega}$, $\{\varphi(x;a_{\eta | \alpha}) : \alpha < \omega\}$ is consistent,
\item for all $\eta \perp \nu$ in $\omega^{<\omega}$, $\{\varphi(x;a_{\eta}),\varphi(x;a_{\nu})\}$ is inconsistent. 
\end{itemize}
Otherwise we say that $\varphi(x,y)$ is $\NTP_1$.
\item $\varphi(x;y)$ has the \emph{$k$-tree property of the second kind} ($k$-$\TP_{2}$) if there is an array of tuples $(a_{\alpha,i})_{\alpha < \omega, i < \omega}$ in $\M$ such that 
\begin{itemize}
\item for all functions $f: \omega \to \omega$, $\{\varphi(x;a_{\alpha, f(\alpha)}) : \alpha < \omega\}$ is consistent,
\item for all $\alpha$, $\{\varphi(x;a_{\alpha, i}) : i < \omega\}$ is $k$-inconsistent.  
\end{itemize}
We say that $\varphi(x,y)$ has the tree property of the second kind ($\TP_2$) if it has the $k$-$\TP_2$ for some $k \in \omega$; otherwise we say that $\varphi(x,y)$ is $\NTP_2$.
\item $T$ has one of the above properties if some formula does modulo $T$.  
\end{enumerate}
\end{defn}

\begin{fact}\cite[III.7.7, III.7.11]{MR513226} (see also \cite[Section 4]{adler2007strong} or \cite[Theorem 6.6]{kim2014tree}) \label{fac: TP dichotomy}
A complete theory $T$ has $\TP$ if and only if it has $\TP_1$ or $\TP_2$. 
\end{fact}

\begin{fact}\label{fac: 1-var 2-incons for TP_2}
\begin{enumerate}
	\item A theory $T$ is simple if and only if it is $\NTP$. If $T$  has $\TP$, then some partitioned formula $\varphi(x,y) \in L$ with $|x| = 1$ has $2$-$\TP$ (see e.g.~\cite{kim1997simple, wagner2000simple}).
	\item If $T$  has $\TP_2$, then some partitioned formula $\varphi(x,y) \in L$ with $|x| = 1$ has $\TP_2$  \cite[Theorem 2.9]{chernikov2014theories}. And if $\varphi(x,y)$ has $\TP_2$, then for some $t \in \omega$, the formula $\psi(x;y_0, \ldots, y_t) := \bigwedge_{i<t} \varphi(x,y_i)$ has $2$-$\TP_2$  \cite[Lemma 3.2]{chernikov2014theories}.
	\item \cite{chernikov2016model} If $T$  has $\TP_1$, then some partitioned formula $\varphi(x,y) \in L$ with $|x| = 1$ has $\TP_1$.
\end{enumerate}
\end{fact}

\begin{fact}
\label{fact: indiscernible witness to TP2}(see e.g.~\cite[Lemma 3.9]{chernikov2014valued})
If $\varphi\left(x,y\right)$ has $k-\TP_{2}$, then working in $\M$,
we can find an array $\left(a_{i,j}:i,j\in\mathbb{N}\right)$ as in Definition \ref{def: NTP2}(3) 
such that moreover $\tp\left(a_{i,j}\right)$ is constant for all
$i,j\in\mathbb{N}$.
\end{fact}

\begin{prop}\label{prop: FHP implies NTP2}
If $\varphi\left(x,y\right)$ is locally $\FHP$, then it is $\NTP_{2}$.
\end{prop}

\begin{proof}
Assume that $\varphi\left(x,y\right)$ has $d$-$\TP_{2}$ for some $d \in \mathbb{N}$. Let an
array $\left(a_{i,j}:i,j\in\mathbb{N}\right)$ in $\M^y$ witnessing this be as given
by Fact \ref{fact: indiscernible witness to TP2}, with $q := \tp\left(a_{i,j}\right)$
constant.

Let $k\in\mathbb{N}$ be arbitrary, we show that $\varphi\left(x,y\right)$
does not satisfy local $\FHP_k$ (on $q (\mathbb{M})$). Let $\alpha:=\frac{1}{k^{k}}>0$, fix some $m \in \mathbb{N}$, let 
$A:=\left\{ a_{i,j}:1\leq i\leq k,1\leq j\leq m\right\} $ and consider
the family $\mathcal{F}=\left\{ \varphi\left(x,a\right):a\in A\right\} $.
Let $n:=\left|A\right|=km$. As for every $f:[k] \to [m]$, $\left\{ \varphi\left(x,a_{i,f\left(i\right)}\right):1\leq i\leq k\right\} $
is consistent, we have that $\left|\Cons_{k}\left(\mathcal{F}\right)\right|\geq m^{k}\geq\left(\frac{n}{k}\right)^{k}\geq\alpha{n \choose k}$.
On the other hand, by pigeonhole for any $S\subseteq A$ of size $\geq\left(d-1\right)k+1$
there is some $1\leq i\leq k$ and some $1\leq j_{1}<\ldots<j_{d}\leq m$
such that $a_{i,j_{t}}\in S$ for all $1\leq t\leq d$, hence $\bigwedge_{a\in S}\varphi\left(x,a\right)$
is inconsistent by the choice of $A$.

This shows that for any $\beta>0$, if we take $m$ such that $\left(d-1\right)k+1<\beta n =\beta km$,
then there is no consistent subset of $\mathcal{F}$ of size $\geq\beta n$.
\end{proof}

The following theorem of Matousek is very important for our discussion.
\begin{fact}
\label{fact: FHP for counting}\cite{matousek2004bounded} Let $\mathcal{F}$
be a family of subsets of $X$, and assume that $d\in\mathbb{N}$
is such that $\pi_{\mathcal{F}}^{*}\left(n\right)=o\left(n^{d}\right)$
as $n\to\infty$ (e.g. if $\mbox{vc}^{*}\left(\mathcal{F}\right)<d$).
Then $\mathcal{F}$ satisfies $\FHP_{d}$.
\end{fact}

Recall that a formula $\varphi\left(x,y\right)$ has IP (independence property) if there is an infinite set $A$ of $|x|$-tuples and for any $I \subseteq A$, there is a $|y|$-tuple $b_I$ such that $\models \varphi(a,b_I)$ if and only if $a\in I$ for all $a\in A$. A formula is $\NIP$ if it does not have IP, and a theory $T$ is $\NIP$ if all formulas are $\NIP$. Hence:
\begin{fact}\label{fac: NIP implies FHP}\cite{matousek2004bounded}
If $\varphi\left(x,y\right)$ is $\NIP$ (in particular, if $\varphi\left(x,y\right)$
is stable) then it is $\FHP$.
\end{fact}

The class of \emph{low} theories, and in particular of low simple
theories, is investigated in \cite{buechler1999lascar,shami2000definability}. \begin{defn}
\label{def: low formula}A formula $\varphi\left(x,y\right)$ is \emph{low
}if there is some $k\in\mathbb{N}$ such that for any indiscernible
sequence $\left(a_{i}:i\in\mathbb{N}\right)$, $\left\{ \varphi\left(x,a_{i}\right):i\in\omega\right\} $
is consistent if and only if all of its subsets of size $k$ are consistent.
A theory is low if it implies that every formula is low.
\end{defn}
It was observed that NIP formulas are low in \cite[Remark 3.33]{chernikov2012forking}. More generally we have:

\begin{prop}
\label{prop: local FHP implies low}If $\varphi\left(x,y\right)$ is
locally $\WFHP$, then it is low.
\end{prop}

\begin{proof}
Assume that $\varphi\left(x,y\right)$ satisfies local $\WFHP_{k}$, we
show that then $\varphi\left(x,y\right)$ is low with the same $k$ in
Definition \ref{def: low formula}. Let $(a_{i}:i\in\mathbb{N})$
be an indiscernible sequence, and assume that every subset of $\mathcal{F} := \{\varphi(x,a_{i}):i\in\mathbb{N}\}$
of size $k$ is consistent. By indiscernibility, it is enough to show
that for every $N\in\mathbb{N}$ there is a strictly increasing subsequence
$(i_{j}\in\mathbb{N}:j<N)$ such that $\{\varphi(x,a_{i_{j}}):1\leq j\leq N\}$
is consistent. Let $p_0$ and $\beta = \beta(p_0)$ be as given for $\varphi(x,y)$ by local $\WFHP_k$ (see Definition \ref{def: local FHP etc}). Note that $\mathcal{F}$ satisfies the $(p_0,k)$-property by assumption. Let $n\in\mathbb{N}$ be such that $\beta n\geq N$.
Hence  some subset of $\{\varphi(x,a_{i}):i<n\}$
of size $\geq\beta n\geq N$ is consistent --- as wanted.
\end{proof}

\begin{rem}
	$\DFHP_k$ for $\varphi(x,y)$ implies a slightly weaker condition than lowness: if $\varphi(x,b)$ does not $k$-divide, then it does not divide. This is equivalent to lowness in simple (or even resilient) theories, but is not known to be equivalent to lowness in $\NTP_2$ (see Proposition 4.13 and Question 4.14 in \cite{yaacov2014independence}).
\end{rem}

Recall that \emph{wnfcp}, or \emph{weak nfcp}, is a strengthening
of lowness which characterizes elementarity of lovely pairs of simple
theories \cite{ben2003lovely,vassiliev2005weak}.

\begin{defn}
\label{def: wnfcp}We say that $T$ is \emph{wnfcp }if:

\begin{enumerate}
\item $T$ is low, i.e. for every $\varphi\left(x,y\right)\in L$ there is
some $k_{\varphi}$ such that for any sequence $\left(b_{i}:i\in\omega\right)$,
if $\left\{ \varphi\left(x,b_{i}\right):i\in\omega\right\} $ is $k_{\varphi}$-consistent,
then it is not $n$-inconsistent for any $n\in\omega$.
\item For any $\varphi\left(x,y\right)$ and $\psi\left(y,z\right)$ in $L$
there is a number $n=n\left(\varphi,\psi\right)$ such that for any $c\in \M^{z}$,
if there is a $k_{\varphi}$-inconsistent family $\left\{ \varphi\left(x,b_{i}\right):i<n\right\} $
with $b_{i}\models\psi\left(x,c\right)$, then there is an infinite
such family.
\end{enumerate}
\end{defn}

\begin{fact}
\label{fact: wnfcp}
\begin{enumerate}
\item \cite[Proposition 2.8]{vassiliev2005weak} If $T$ is supersimple,
of SU-rank $1$, then $T$ is wnfcp.
\item \cite[Corollary 3.10]{vassiliev2005weak} If $T$ is stable, then
$T$ is wnfcp if and only if $T$ is nfcp.
\end{enumerate}
\end{fact}

\begin{rem} \label{rem: around wnfcp}

\begin{enumerate}
\item In a simple theory $T$, the following are equivalent (and (a) implies (b)
is true in any theory):
\begin{enumerate}
\item $T$ is wnfcp;
\item $T$ is low and $Q_{\varphi,\psi}\left(z\right)$ is (type-)definable
for all $\varphi\left(x,y\right),\psi\left(y,z\right)\in L$, where $Q_{\varphi,\psi}\left(c\right)$
hold if $\varphi\left(x,b\right)$ does not divide over $c$ for all
$b\models\psi\left(y,c\right)$.
\end{enumerate}
\item Wnfcp implies elimination of $\exists^{\infty}$, namely $Q_{x=y,\psi\left(y,z\right)}\left(c\right)$
holds if and only if $\psi\left(y,c\right)$ defines a finite set.
\item Definition \ref{def: wnfcp}(2) can be rephrased as saying that given
$\varphi\left(x,y\right),\psi\left(y,z\right)$ there is some $n$ such
that for any $c\in \M^{z}$, if the family $\left\{ \varphi\left(x,b\right):b\models\psi\left(y,c\right)\right\} $
satisfies the $\left(n',k\right)$-property for some $n'\in\omega$,
then it already satisfies the $\left(n,k\right)$-property.
\end{enumerate}
\end{rem}

\begin{example}
\label{exa: fhp vs local fhp and wnfcp}Consider the theory $T_{\Pr}$
of the group $\left(\mathbb{Z},+\right)$ expanded by a predicate
for the primes. Assuming Dickson's conjecture in number theory, $T_{\Pr}$
is supersimple SU-rank $1$, hence wnfcp by Fact \ref{fact: wnfcp},
locally FHP theory which is not $\FHP$ (see Section \ref{subsec: T_pr}
for the details). Note also that by Fact \ref{fact: wnfcp}(2), any
stable theory with fcp satisfies FHP and not wnfcp, so there is no
implication between $\FHP$ and wnfcp in general.
\end{example}

\begin{prop}
\label{prop: FHP is locFHP in omega cat} If $T$ is $\omega$-categorical
and locally $\FHP$, then $T$ is $\FHP$ and wnfcp.
\end{prop}

\begin{proof}
Fix $\varphi\left(x,y\right) \in L$. By $\omega$-categoricity, there are
only finitely many types $p_{1},\ldots,p_{m}$ in $S_{y}\left(\emptyset\right)$,
and all of them are isolated, say by $\psi_{1}\left(y\right),\ldots,\psi_{m}\left(y\right)$.
Then $\varphi\left(x,y\right)$ is equivalent to $\bigvee_{1\leq i\leq m}\left(\varphi\left(x,y\right)\land\psi_{i}\left(y\right)\right)$.
As $\varphi\left(x,y\right)$ is locally $\FHP$ by assumption, we have
that $\varphi\left(x,y\right)\land\psi_{i}\left(y\right)$ is $\FHP$
for each $1\leq i\leq m$, and as the class of $\FHP$ formulas is
closed under disjunctions by Lemma \ref{lem: Basic operations preserving FH}(2),
we conclude that $\varphi\left(x,y\right)$ is $\FHP$. Also, it is a general fact that if an $\omega$-categorical theory
is low, then it is wnfcp (as any invariant set over a finite number
of parameters is definable, see e.g. \cite{palacin2012omega}). Hence
local $\FHP$ implies wnfcp by Proposition \ref{prop: local FHP implies low}.
\end{proof}
%

%

\section{$\protect\FHP$ relatively to a class of measures}\label{sec: FHP and measures}

\subsection{Keisler measures}

Let $T$ be a complete $L$-theory. We will work in $M \models T$, and $\M \succ M$ is a saturated elementary extension. For any set $A \subseteq \mathcal{U}$, a \emph{Keisler measure} over $A$ in variables $x$ is a finitely additive probability measure on the Boolean algebra $L_{x}(A)$ of $A$-definable subsets of $\M^x$. We denote the space of Keisler measures over $A$ (in variables $x$) as $\mathfrak{M}_x(A)$. Every element of $\mathfrak{M}_x(A)$ is in unique correspondence with a regular Borel probability measure on the space of types $S_{x}(A)$, and we will routinely use this correspondence.  We recall some notions from \cite{Keisl, hrushovski2008groups, hrushovski2011nip, hrushovski2013generically, chernikov2016definable, Gannon}, and refer to \cite{StBourb, CheSurv} for a survey.

\begin{defn}\label{fac: types of measures} Let $\mu \in \mathfrak{M}_{x}(\M)$ and $A \subseteq \M$ a small subset. 
\begin{enumerate}
\item  $\mu$ is \emph{$A$-invariant} if for any partitioned $L(M)$-formula $\varphi(x;y) \in L$ and any $b,b'\in \M^y$, if $b\equiv_A b'$ then $\mu(\varphi(x;b))=\mu(\varphi(x;b'))$.
\item Assume that $\mu$ is $A$-invariant and $\varphi(x;y) \in L(A)$. We define the map $F_{\mu,A}^{\varphi}:S_{y}(A) \to [0,1]$ by $F_{\mu,A}^{\varphi}(q)=\mu(\varphi(x;b))$, where $b\models q$ (this is well-defined by $A$-invariance of $\mu$). 

\item $\mu$ is \emph{Borel-definable} (respectively, \emph{definable}) over $A$ if  $\mu$ is $A$-invariant and for any partitioned $L(A)$-formula $\varphi(x;y)$, the map $F_{\mu, A}^{\varphi}$ is Borel-measurable (respectively, continuous). 
\item $\mu$ is \emph{finitely satisfiable in $A$}  if for any $L(\M)$-formula $\varphi(x)$, if $\mu(\varphi(x))>0$ then $\M \models\varphi(a)$ for some tuple $a$ in $A$.  
\item $\mu$ is \emph{dfs} over $A$ if $\mu$ is both definable over $A$ and finitely satisfiable in $A$.

\item Given $\overline{a} \in (\M^{x})^{<\omega}$, with $\overline{a} = (a_1, \ldots, a_n)$, the associated \emph{average measure} $\Av_{\overline{a}} \in \mathfrak{M}_{x}(\M)$ is defined by
\begin{equation*} \Av_{\overline{a}}(\varphi(x)) := \frac{|\{i \in [n]: \M \models \varphi(a_i)\}| }{n}
\end{equation*}
for any $\varphi(x) \in L_{x}(\M)$.

\item $\mu$ is \emph{finitely approximated}, or \emph{fam}, over $A$ if for any  $\varphi(x;y) \in L$ and any $\varepsilon \in \mathbb{R}_{>0}$, there exists a finite  tuple   $\bar{a}$ from $A$ such that for any $b\in \M^{y}$, $\mu(\varphi(x;b)) \approx_{\varepsilon} \Av_{\bar{a}}(\varphi(x;b))$. In this case, we call $\bar{a}$ a \emph{$(\varphi,\varepsilon)$-approximation for $\mu$}. 
\end{enumerate}  
We say that $\mu$ is \emph{invariant} if it is invariant over some small model $M \prec \M$, and similarly for the other properties.
\end{defn} 

\begin{rem}\label{rem: fam for fin many formulas}
	Note that if $\mu$ is fam and $\Delta(x,y)$ is a finite set of partitioned formulas, then for every $\varepsilon > 0$ there exists a finite tuple $\bar{a}$ so that $\bar{a}$ is a $(\varphi(x,y), \varepsilon)$-approximation for $\mu$ for all $\varphi(x,y) \in \Delta$ simultaneously (by coding finitely many formulas into one using additional variables).
\end{rem}

\begin{defn}\label{def: tensor prod of meas} Assume $\mu \in \mathfrak{M}_{x}(\M)$ is Borel-definable and 
 $\nu \in \mathfrak{M}_{y}(\M)$ arbitrary. We let $\mu \otimes \nu$ be the unique measure in $\mathfrak{M}_{xy}(\M)$ such that for any $\varphi(x,y) \in L_{xy}(\M)$, we have 
\begin{equation*} 
(\mu \otimes \nu)(\varphi(x,y)) = \int_{S_{y}(A)} F_{\mu,A}^{\varphi} d(\widehat{\nu|_{A}}), 
\end{equation*} 
where $\mu$ is $A$-invariant and $A$ contains all the parameters from $\varphi$, and  $\widehat{\nu|_{A}}$ is the unique regular Borel probability measure on $S_{y}(A)$ extending the Keisler measure $\nu|_{A}$. 
\end{defn} 
\noindent See e.g.~\cite[Section 3.1]{chernikov2022definable} for an explanation why this product is well-defined and its basic properties. We will often abuse the notation slightly and replace $\widehat{\nu|_{A}}$ with either $\nu|_{A}$ or simply $\nu$ when it is clear from the context, and sometimes write $F_{\mu,A}^{\varphi}$ as $F_{\mu}^{\varphi}$. In general, $\otimes$ need not be commutative/associative on Borel definable measures in arbitrary theories.
\begin{fact}\label{fac: otimes props}
	\begin{enumerate}
		\item \cite[Theorem 2.18]{conant2023keisler}  Suppose $\mu, \nu$ are definable and $\lambda$ arbitrary. Then $\mu \otimes \nu$ is definable, and  $(\mu \otimes \nu) \otimes \lambda = \mu \otimes (\nu \otimes \lambda)$.
		\item \cite[Theorem 5.16]{conant2023keisler} If $\mu$ is fim and $\nu$ is Borel definable, or $\mu$ is fam and $\nu$ is definable. Then $\mu \otimes \nu = \nu \otimes \mu$.
	\end{enumerate}
\end{fact}

\subsection{$\protect\FHP$ relative to a class of measures}

We define a generalization of the FHP (Definition \ref{def: FHP}) relatively to a class
of definable Keisler measures.

\begin{defn}
\label{def: FHP for a class of measures}Let $\mathcal{M}$ be a first-order
$\mathcal{L}$-structure, $\M \succ M$ a saturated model. Let $\varphi\left(x,y\right) \in L$ be a partitioned formula, and let $\mathfrak{M} \subseteq \mathfrak{M}_y(\M)$ be a class of definable Keisler measures. We say that $\varphi\left(x,y\right)$ satisfies $\FHP(d,\alpha,\beta)$  \emph{relatively
to $\mathfrak{M}$} if for any $\mu\in\mathfrak{M}$, if $\mu^{\otimes d}\left(\exists x\bigwedge_{i \in [d]}\varphi\left(x,y_{i}\right)\right)>\alpha$
then $\mu\left(\varphi\left(a,y\right)\right)>\beta$ for some $a\in \M^{x}$.
And $\varphi(x,y)$ satisfies $\FHP_d$ \emph{relatively
to $\mathfrak{M}$} if for
every $\alpha\in\mathbb{R}_{>0}$ there is some $\beta\in\mathbb{R}_{>0}$
so that $\varphi\left(x,y\right)$ satisfies $\FHP(d,\alpha,\beta)$  relatively
to $\mathfrak{M}$. And $\varphi\left(x,y\right)$ satisfies \emph{$\FHP$ relatively to
$\mathfrak{M}$ }if it satisfies $\FHP_{d}$ relatively to $\mathfrak{M}$
for some $d\in\mathbb{N}$.
\end{defn}

We note that the usual FHP property is equivalent to the FHP property with respect to the class of finitely supported measures:
\begin{prop}
\label{rem: FHP iff FHP for finite measures}The formula $\varphi\left(x,y\right)$
satisfies $\FHP_{d}$ (as in Definition \ref{def: FHP}) if and only
if it satisfies $\FHP_{d}$ relatively to the class $\mathfrak{M}_{y}^{\fin}(\M)$
of all Keisler measures supported on finite subsets of $\M^{y}$.
\end{prop}

\begin{proof}
Assume that $\varphi\left(x,y\right)$ satisfies $\FHP_{d}$, and fix
some $\alpha>0$. Let $\mu \in \mathfrak{M}_y(\M)$ be concentrated on a finite
$B\subseteq \M^{y}$. Say $B=\left\{ b_{1},\ldots,b_{n}\right\} $
and $\mu\left(\left\{ b_{i}\right\} \right)=r_{i}\in\left[0,1\right]$
with $\sum_{i=1}^{n}r_{i}=1$.  Assume that $\mu^{\otimes d}\left(\exists x\bigwedge_{i \in [d]}\varphi\left(x,y_{i}\right)\right)>\alpha$,
and let $\varepsilon\in\mathbb{R}_{>0}$ be arbitrarily small. Then,
wiggling the weights of the points a little bit one by one, we can choose
a measure $\nu$ concentrated on $B$ such that $\nu\left( \{b_{i}\}\right)=s_{i}$
with $s_{i}$ \emph{rational}, such that $\nu^{\otimes e}\left(C\right)\approx^{\varepsilon}\mu^{\otimes e}\left(C\right)$
for all $1\leq e\leq d$ and all $C\subseteq (\M^y)^{e}$, in particular
$\nu^{\otimes d}\left(\exists x\bigwedge_{i \in [d]}\varphi\left(x,y_{i}\right)\right)>\alpha-\varepsilon$.

Let $s_{i}=\frac{t_{i}}{D}$, where $t_{i},D\in\mathbb{N}$
and $D$ is a common denominator for the $s_{i}$'s (so $\sum_{i=1}^{n}t_{i}=D$).
Then we choose a tuple of sets $\mathcal{S}=\left(S_{i}:1\leq i\leq D\right)$
such that it contains $t_{i}$ repetitions of the set $\varphi\left(M,b_{i}\right)$,
for each $1\leq i\leq n$. Now from the assumption on the measure
we have $\left|\left\{ \left(i_{1},\ldots,i_{d}\right)\in\left[D\right]^{d}:\bigcap_{1\leq j\leq d}S_{i_{j}}\neq\emptyset\right\} \right|\geq\left(\alpha-\varepsilon\right)D^{d}$.
Note that $\left|\left\{ \left(i_{1},\ldots,i_{d}\right)\in\left[D\right]^{d}:\bigvee_{1\leq j<j'\leq d}i_{j}=i_{j'}\right\} \right|\leq{d \choose 2}D^{d-1}$, so $\Cons_{d}\left(\mathcal{S}\right)\geq\frac{1}{d^{d}}\left(\alpha-\varepsilon\right)D^{d}-{d \choose 2}D^{d-1}$
(the factor $\frac{1}{d^{d}}$ is there since every set in ${D \choose d}$
is counted $d^{d}$ times). Hence, taking $\varepsilon := \frac{\alpha}{2}$,
we have $\Cons_{d}\left(\mathcal{S}\right)\geq\frac{1}{2}\frac{1}{d^{d}}\frac{\alpha}{2}D^{d}$
for all $D$ sufficiently large, so $\Cons_{d}\left(\mathcal{S}\right)\geq\alpha'{D \choose d}$
for $\alpha'=\alpha'\left(\alpha,d\right):=\frac{1}{2}\frac{1}{d^{d}}\frac{1}{d^{d}}\frac{\alpha}{2}>0$.
As $\varphi(x,y)$ satisfies $\FHP_{d}$, there is some $\beta=\beta\left(\alpha'\right)$
and $a\in \M^{x}$ such that, taking $C:=\left\{ i \in [D]: a\in S_{i}\right\} $,
and $J:=\left\{j \in [n] :\models\varphi\left(a,b_{j}\right)\right\} $,
we have $\left|C\right|\geq\beta D$ (this is where we use that in
the definition of $\FHP$ repetitions of sets are allowed). Note that
if $a\in S_{i}$, then $a\in S_{j}$ for every copy of $S_{i}$ appearing
in $\mathcal{S}$ as well, hence $\sum_{j\in J}t_{j}\geq\beta D$,
and so $\nu\left(\varphi\left(a,y\right)\right)\geq\frac{\beta D}{D}\geq\beta$,
as wanted.

Conversely, assume $\varphi$ satisfies $\FHP_{d}$ relatively to $\mathfrak{M}_{y}^{\fin}$
and we are given a tuple of sets $\mathcal{S}=\left\{ \varphi\left(x,b_{i}\right):1\leq i\leq n\right\} $
such that $\Cons_{d}\left(S\right)\geq\alpha{n \choose d}$. Define
a finitely supported Keisler measure $\mu \in \mathfrak{M}_y^{\fin}$ via $\mu\left(Y\right):=\frac{\left|\left\{ i \in [n] : b_{i}\in Y\right\} \right|}{n}$.
Then 
\begin{gather*}
	\mu^{\otimes d}\left(\exists x\bigwedge_{i \in [d]}\varphi\left(x,y_{i}\right)\right)=\frac{\left\{ \left(i_{1},\ldots,i_{d}\right)\in\left[n\right]^{d}:\models\exists x\bigwedge_{j \in [d]}\varphi\left(x,b_{i_{j}}\right)\right\} }{n^{d}}\\
	\geq\frac{\alpha{n \choose d}}{n^{d}}\geq\alpha'
\end{gather*}
for $\alpha':=\frac{\alpha}{d^{d}}>0$. Taking $\beta>0$ as given
for $\alpha'$ by $\FHP_{d}$ relatively to $\mathfrak{M}_{y}^{\fin}$,
we have that $\mu\left(\varphi\left(a,y\right)\right)\geq\beta$ for
some $a\in \M^{x}$. Then $\left|\left\{ i:\models\varphi\left(a,b_{i}\right)\right\} \right|\geq\beta n$,
hence $\varphi$ satisfies $\FHP_{d}$.
\end{proof}

This equivalence lifts further from finitely supported measures to fam measures (see Definition \ref{fac: types of measures}(7)).
Namely, we have (by a repeated application of \cite[Proposition 2.14(2)]{chernikov2016definable}; see also \cite[Proposition 2.10]{conant2020remarks}):
\begin{lem}
\label{lem: e-approx for product measure}Let $\psi\left(x_{1},\ldots,x_{d};z\right)\in L$
and $\varepsilon>0$ be arbitrary. Assume that $\mu_{i}$ is a fam 
measure on $M_{x_{i}}$ and $\bar{b}_{i}= (b_{i,1},\ldots, b_{i,n_{i}}) \in (\M^{x_{i}})^{n_i}$
is a $(\psi_{i}\left(x_{i};z_{i}\right), \varepsilon)$-approximation for $\mu_{i}$, where $\psi_{i}\left(x_{i};z_{i}\right)$ is obtained from $\psi\left(x_{1},\ldots,x_{d},z\right)$ by partitioning
its variables into two groups $x_{i}$ and $z_{i}=x_{1}\ldots x_{i-1}x_{i+1}\ldots x_{d}z$,
for $i=1,\ldots,d$. Then $\bar{b} := \left( b_{i,j} : i \in [d], j \in [n_i] \right)$ is a
$\left( \psi\left(x_{1},\ldots,x_{d};z\right), 2^{d}\varepsilon \right)$-approximation for $\mu := \mu_{1}\otimes\ldots\otimes\mu_{d}$.
\end{lem}

In view of this, we can talk about $\FHP$ relatively to the class
$\mathfrak{M}_{y}^{\fap}$ of all fap measures on $M_{y}$ with the
product defined above. Note that $\mathfrak{M}_{y}^{\fin}\subseteq\mathfrak{M}_{y}^{\fap}$,
and the products defined above coincide on $\mathfrak{M}_{y}^{\fin}$.

\begin{prop}
\label{prop: FHP for gen. stable}A partitioned formula $\varphi\left(x,y\right) \in L$
satisfies $\FHP_{d}$ if and only
if it satisfies $\FHP_{d}$ relatively to the class $\mathfrak{M}_{y}^{\fam} \subseteq \mathfrak{M}_{y}(\M)$ of all fam Keisler measures on $\M^{y}$.
\end{prop}

\begin{proof}

Assume that $\varphi\left(x;y\right) \in L$
satisfies $\FHP_{d}$, and let $\alpha>0$ be arbitrary.  Let $\beta'>0$
be as given for $\alpha' := \frac{\alpha}{2}$ by $\FHP_{d}$ relatively
to $\mathfrak{M}_{y}^{\fin}$ (using Proposition \ref{rem: FHP iff FHP for finite measures}).

Let $\varepsilon$ be arbitrary with
$0<\varepsilon<\min\left\{ \frac{\alpha}{2},\beta'\right\} $ and
let $\beta := \beta'-\varepsilon>0$ --- can be chosen depending only
on $\alpha$. Let $\psi\left(y_{1},\ldots,y_{d}\right) := \exists x\bigwedge_{i \in [d]}\varphi\left(x,y_{i}\right)$.
For $i \in [d]$, define $\psi_{i}\left(y_{i};z_{i}\right):=\exists x\bigwedge_{i \in [d]}\varphi\left(x,y_{i}\right)$,
with $z_{i}=y_{1},\ldots,y_{i-1},y_{i+1},\ldots,y_{d}$. Let 
\[
\Delta\left(y;z\right):=\left\{ \psi_{i}\left(y_{i},z_{i}\right):1\leq i\leq d\right\} \cup\left\{ \varphi\left(x;y\right)\right\} \mbox{,}
\]
\noindent where $z:=x^{\wedge}z_{i}$. Now let $\mu \in \mathfrak{M}^{\fam}_y$ be arbitrary with $\mu^{\otimes d}\left(\exists x\bigwedge_{1\leq i\leq d}\varphi\left(x,y_{i}\right)\right)\geq\alpha$. By Remark \ref{rem: fam for fin many formulas}, let $\bar{b} = (b_1, \ldots, b_n) \in ( \M^{y})^n$ be an
$\frac{\varepsilon}{2^{d}}$-approximation for $\mu$ on all formulas in $\Delta$ simultaneously. 
Then, by Lemma \ref{lem: e-approx for product measure}, $\bar{b}' := ((b_{i_1}, \ldots, b_{i_d}) : (i_1, \ldots, i_d) \in [n]^d)$
is an $\varepsilon$-approximation for $\mu^{\otimes d}$ on
$\psi(y_1, \ldots, y_d)$. As $\Av_{\bar{b}}^{\otimes d} = \Av_{\bar{b}'}$, we thus have $\Av_{\bar{b}}^{\otimes d}\left(\exists x\bigwedge_{1\leq i\leq d}\varphi\left(x,y_{i}\right)\right) \geq \alpha - \varepsilon \geq \alpha'$. And $\Av_{\bar{b}} \in \mathfrak{M}_y^{\fin}(\M)$, so by the choice of $\beta'$ there is some $a \in \M^x$ so that $\Av_{\bar{b}}(\varphi(a,y)) \geq \beta'$. As $\bar{b}$ is also a $(\varphi, \varepsilon)$-approximation for $\mu$, this implies $\mu(\varphi(a,y)) \geq \beta' - \varepsilon \geq \beta$.
%

The converse follows by Proposition \ref{rem: FHP iff FHP for finite measures} as $\mathfrak{M}_{y}^{\fin} \subseteq \mathfrak{M}_{y}^{\fam}$.
\end{proof}

\subsection{$\protect\FHP$ for generically stable measures in $\protect\NIP$
structures}\label{sec: FHP vs HrushPilSim}

Matousek's result in Fact \ref{fact: FHP for counting} implies that
every $\NIP$ formula $\varphi\left(x,y\right)$ satisfies FHP. In NIP, fam is one of the equivalent
characterizations of \emph{generically stable} measures (see Remark \ref{rem: fim in NIP is gen stab}). Hence
Proposition \ref{prop: FHP for gen. stable} combined with Matousek's
theorem immediately implies the main theorem of \cite{hrushovski2012note}
(by taking the contrapositives and exchanging the roles of the variables
in the statement).
\begin{fact}\cite[Proposition 2.1]{hrushovski2012note}
\label{fact: HPS}Let $T$ be NIP and let $\mu \in \mathfrak{M}_x(\M)$ be a generically
stable measure. For any formula $\varphi(x,y)$, if $\mu\left(\varphi\left(x,b\right)\right)=0$ for all
$b$, then there is some $d \in \mathbb{N}$ such that $\mu^{\otimes d}\left(\exists y\left(\varphi\left(x_{1},y\right)\land\ldots\land\varphi\left(x_{n},y\right)\right)\right)=0$. Moreover, $d$ depends only on $\varphi$.
\end{fact}

Conversely, under the global NIP assumption on the theory, we can quickly deduce from Fact \ref{fact: HPS} using compactness that every formula satisfies $\FHP$. Indeed, fix $\varphi(x,y) \in L$ and let $d$ be as given by Fact \ref{fact: HPS} for $\varphi(y,x)^{\ast} := \varphi(x,y)$. Assume $\varphi(x,y)$ does not satisfy $\FHP_d$ relative to $\mathfrak{M}_y^{\fam}(\M)$. Then there exists $\alpha >0$ and for every $i \in \omega$ some $\mu_i \in \mathfrak{M}_y^{\fam}(\M)$ with $\mu^{\otimes d}_i\left(\exists x\bigwedge_{t \in [d]}\varphi\left(x,y_{t}\right)\right)\geq\alpha$ but $\mu_i(\varphi(a,y)) \leq \frac{1}{i}$ for all $a \in \M^x$. Following \cite[Section 3.4]{chernikov2025externally}, in an NIP theory $T$ we identify the set of global fam measures $\mathfrak{M}_y(\M)$ with a hyperdefinable set $\widetilde{\mathcal{M}}_y$, let $\mu \in \mathfrak{M}_y^{\fam}(\M) \mapsto [\mu] \in \widetilde{\mathcal{M}}_y$ denote the bijection. The sets
\begin{gather*}
	 X_i := \left\{[\mu] \in \widetilde{\mathcal{M}}_y: \forall a \in \M^x \mu(\varphi(a,y)) \leq \frac{1}{i} \right\},\\
	  Y := \left\{ [\mu] \in \widetilde{\mathcal{M}}_y: \mu^{\otimes d}\left(\exists x\bigwedge_{t \in [d]}\varphi\left(x,y_{t}\right)\right)\geq\alpha \right\}
\end{gather*}
\noindent are type-definable  (by \cite[Remark 3.27, Proposition 3.29]{chernikov2025externally}). Note that $X_{i+1} \subseteq X_{i}$ and $[\mu_i]  \in X_i \cap Y$. It follows by saturation of $\M$ that there exists some $\mu \in \mathfrak{M}_x^{\fam}(\M)$ with $[\mu] \in Y \cap \bigcap_{i \in \omega} X_i$. That is, $\mu^{\otimes d}\left(\exists x\bigwedge_{t \in [d]}\varphi\left(x,y_{t}\right)\right)\geq\alpha $ and $\mu(\varphi(a,y)) = 0$ for all $a \in \M^x$. But this contradicts the choice of $d$.

\begin{rem}
	We could  instead use directly the folklore facts that, under NIP, ultralimits of fam  measures are fam  (as measures on the ultraproduct) and ultralimits commute with $\otimes$ (see e.g.~\cite{gannon2025concerning}).
\end{rem}

\section{Colorful fractional Helly property, 
burden and $\left(p,q\right)$-theorems}\label{subsec: Colorful-fractional-Helly and dp-rank}

In this section we consider a more general (``colorful'') version of the $\FHP$
property relatively to products of different measures and its relation to the model theoretic notion of  burden (in particular
strengthening and generalizing some of the results of Pillay from
\cite{pillay2013weight}). We also observe that, at the level of the theory, the conclusion of the $(p,q)$-theorem is equivalent to NIP (equivalently, finite VC-dimension).

\subsection{Colorful FHP, burden, VC-density}

The following lemma is very similar to \cite[Lemma 4.6]{chernikov2020hypergraph} but in a slightly different setting:
\begin{lem}
\label{lem: Erdos-Stone} Let $k \in \mathbb{N}$ and Keisler measures $\mu_i \in \mathfrak{M}_{x_i}(\M)$ for $i \in [k]$ be definable and pairwise commuting (in particular each $\mu_i$ commutes with itself). 
Assume that $R \subseteq \M^{x_1} \times \ldots \times \M^{x_k}$
is definable and such that $\bigotimes_{i \in [k]} \mu_i\left(R\right)\geq\alpha>0$. 
For $d\in\mathbb{N}$, consider the definable set 
\begin{gather*}
	\Sigma_{d}:=\Big\{ \left(\bar{x}_{1},\ldots,\bar{x}_{k}\right)\in \left(\M^{x_1}\right)^{d} \times \ldots \times \left(\M^{x_k}\right)^d :\\\left(x_{1,i_{1}},\ldots,x_{k,i_{k}}\right)\in R\text{ for all } (i_{1},\ldots,i_{k}) \in\left[d\right]^k \Big\} \text{.}
\end{gather*}

Then $\mu_{1}^{\otimes d}\otimes\ldots\otimes\mu_{k}^{\otimes d}\left(\Sigma_{d}\right)\geq\alpha^{dk}$.
\end{lem}

\begin{proof}
We will use Fact \ref{fac: otimes props} and the pairwise commuting assumption freely.	 We have the following:
\begin{claim}
	Assume $\mu \in \mathfrak{M}_{x}(\M), \nu \in  \mathfrak{M}_{y}(\M)$ are definable and pairwise commuting, and  $E \subseteq \M^{x} \times \M^y$ is definable with $\mu\otimes\nu\left(E\right)\geq\alpha>0$.
Fix $d\in\mathbb{N}$ and let
\[
\Gamma:=\left\{ \left(a;b_{1},\ldots,b_{d}\right)\in \M^x \times (\M^y)^{d}:\left(a,b_{i}\right)\in E\text{ for all } i \in [d]\right\} \text{.}
\]
Then $\mu\otimes\nu^{\otimes d}\left(\Gamma \right)\geq\alpha^{d}>0$.
\end{claim}
\begin{proof}
Assume $M \prec \M$ is small so that  $\mu, \nu, E$ (and hence also $\Gamma_d$) are definable over $M$. Then, by Definition \ref{def: tensor prod of meas}, 
$\mu\otimes\nu^{\otimes d}\left(\Gamma \right) = \int_{r \in S_x(M)} \nu^{\otimes d}(\Gamma_a) d\mu|_{M}(r)$, where $a \in \M^x$ with $a \models r$ is arbitrary and $\Gamma_a$ denotes the fiber of $\Gamma$ at $a$. Note that for each fixed $a \in \M^x$, $\Gamma_a = E(a,y_1) \land \ldots \land E(a, y_d)$ belongs to the product Boolean algebra $L_{x_1}(\M) \times \ldots \times L_{x_d}(\M)$. As $\nu^{\otimes d}$ extends the product measure $\nu^{\times d}$, we have $\nu^{\otimes d}(\Gamma_a) = \left( \nu(E_a) \right)^d$ for all $a \in \M^x$.
Then, using H\" older inequality with $p=d, q=d/(d-1)$, we have $\mu\otimes\nu^{\otimes d}\left(\Gamma\right)=\int_{r \in S_x(M)}\left(\nu\left(E_{a}\right)\right)^{d}d\mu=\int_{r \in S_x(M)}\left(\nu\left(E_{a}\right)\right)^{d}d\mu\cdot\int_{r \in S_x(M)}1^{\frac{d}{d-1}}d\mu\geq\left(\int_{r \in S_x(M)}\nu\left(E_{a}\right)d\mu\right)^{d}=\left(\mu\otimes\nu\left(E\right)\right)^{d}\geq\alpha^{d}>0$. 
\end{proof}

Now let $R \subseteq \M^{x_1} \times \ldots \times \M^{x_k}$ be as given by
assumption. For $i \in [k]$, let $V_{i}:=\M^{x_i}$ and $U_{i}:=\prod_{j<i}(\M^{x_j})^{d}\times\prod_{j>i}\M^{x_j}$, and consider the definable 
binary relation $E_{i}\subseteq U_{i}\times V_{i}$ given
by $\left(\bar{x}_{1},\ldots,\bar{x}_{i-1},x_{i+1},\ldots,x_{k};x_{i}\right)\in E_{i}\iff$$\left(x_{1,j_{1}},\ldots,x_{i-1,j_{i-1}},x_{i},x_{i+1},\ldots,x_{k}\right)\in R$
for all $j_{1},\ldots,j_{i-1}\in\left[d\right]$. Let $\nu_{i}:=\mu_{1}^{ \otimes d}\otimes\ldots\otimes\mu_{i-1}^{\otimes d}\otimes\mu_{i+1}\otimes\ldots\otimes\mu_{k}$.

In particular $\mu_{i},\nu_{i}$ are pairwise commuting for every $i \in [k]$. 
Note that
$E_{1}$ is equal to $R$, $E_{i+1}$ is equal to $\Gamma_{i}$ as
obtained from $E_{i}$ by the claim, that $\Gamma_{k}=\Sigma_{d}$
and $\mu_{k}\otimes\nu_{k}=\mu_{1}^{\otimes d}\otimes\ldots\otimes\mu_{k}^{\otimes d}$
(up to a permutation and repartition of the variables). Hence, applying
the claim repeatedly and using that all products of the measures involved
commute, we have $\mu_{i+1}\otimes\nu_{i+1}\left(E_{i+1}\right)\geq\alpha^{d}\cdot\mu_{i}\otimes\nu_{i}\left(E_{i}\right)$
for all $i$. As $\mu_{1}\otimes\nu_{1}=\mu_{1}\otimes\ldots\otimes\mu_{k}$
and $\mu_{1}\otimes\nu_{1}\left(E_{1}\right)\geq\alpha$, we conclude
that $\mu_{1}^{\otimes d}\otimes\ldots\otimes\mu_{k}^{\otimes d}\left(\Sigma_{d}\right)\geq\alpha^{dk}$.
\end{proof}

Adler   \cite{adler2007strong} introduced \emph{burden}, a notion
based on the invariant $\kappa_{\inp}$ of Shelah \cite{MR513226}
which generalizes simultaneously dp-rank in NIP theories and weight
in simple theories. For notational convenience we consider an extension $\card$ of the
linear order on cardinals by adding a new maximal element $\infty$
and replacing every limit cardinal $\kappa$ by two new elements $\kappa_{-}$
and $\kappa_{+}$. The standard embedding of cardinals into $\card$
identifies $\kappa$ with $\kappa_{+}$. In the following, whenever
we take a supremum of a set of cardinals, we will be computing it
in $\card$.

\begin{defn}\label{def: burden}
  \cite{adler2007strong} Let $p\left(x\right)$ be a (partial) type.
  \begin{enumerate}
  \item An $\inp$-pattern of depth $\kappa$ in $p(x)$ consists of $\left(\bar{a}_{i},\varphi_{i}(x,y_{i}),k_{i}\right)_{i\in\kappa}$
    with $\bar{a}_{i}=\left(a_{ij}\right)_{j\in\omega}$ and $k_{i}\in\omega$
    such that:

    \begin{itemize}
    \item $\left\{ \varphi_{i}(x,a_{ij})\right\} _{j\in\omega}$ is $k_{i}$-inconsistent
      for every $i\in\kappa$,
    \item $p(x)\cup\left\{ \varphi_{i}(x,a_{if(i)})\right\} _{i\in\kappa}$
      is consistent for every $f:\,\kappa\to\omega$.
    \end{itemize}
  \item The \emph{burden} of a partial type $p(x)$ is the supremum (in Card$^{*}$)
    of the depths of $\inp$-patterns in it. We denote the burden of $p$
    as $\bdn(p)$ and we write $\bdn(a/A)$ for $\bdn(\tp(a/A))$.
  \item By compactness, $T$ is $\NTP_2$ if and only
    if $\bdn\left(\mbox{"}x=x\mbox{"}\right)<\infty$, if and only if
    $\bdn\left(\mbox{"}x=x\mbox{"}\right)<\left|T\right|^{+}$.
  \item A theory $T$ is called \emph{strong} if $\bdn\left(p\right)\leq\left(\aleph_{0}\right)_{-}$
    for every finitary type $p$ (equivalently, there is no $\inp$-pattern
    of infinite depth). Of course, if $T$ is strong then it is $\NTP_2$.
  \end{enumerate}
\end{defn}
\begin{fact}
  \label{fct: burden in simple and NIP}\cite{adler2007strong} \begin{enumerate}
  \item Let $T$ be $\NIP$. Then $\bdn(p)=\dprk(p)$ for any $p$.
  \item Let $T$ be simple. Then the burden of $p$ is the supremum of the weights
    of its complete extensions.
  \end{enumerate}
\end{fact}

\begin{fact}\label{fac: submult of burden}
	\begin{enumerate}
		\item \cite{chernikov2014theories} If $T$ is any theory, then ``burden $+1$'' is \emph{sub-multiplicative}, i.e.~for any tuples $a,b$ in $\M$ and cardinals $\kappa,\lambda$  we have that $\bdn(a) < \kappa, \bdn(b/a) < \lambda$ implies $\bdn(a,b) < \kappa \times \lambda$.
		\item \cite{kaplan2013additivity} If $T$ is NIP, then burden (which is equal to dp-rank in this case) is sub-additive: for any tuples $a,b$, $\bdn(a,b) \leq \bdn(a) + \bdn(b/a)$. 
		By \cite{takahashi2025consequences}, the same is true only assuming that $T$ satisfies the dependent dividing conjecture (from \cite{chernikov2014theories}).
		\item If $T$ is simple, then burden is also sub-additive (by sub-additivity of weight in simple theories and Fact \ref{fct: burden in simple and NIP}(2)). 
	\end{enumerate}
\end{fact}
\noindent It is open if burden is sub-additive in general (or in $\NTP_2$) theories \cite[Conjecture 2.7]{chernikov2014theories}.

\begin{prop}\label{prop: colorful FHP meas bdd by burden}
Assume that $\bdn\left(x\right)<k$ and let $\varphi_{i}\left(x,y_{i}\right) \in L$,
$i \in [k]$, satisfiy $\FHP$ relatively to a class of definable measures $\mathfrak{M}_{i} \subseteq \mathfrak{M}_{y_i}(\M)$ (see Definition \ref{def: FHP for a class of measures}). Then for every $\alpha>0$ there is
$\gamma>0$ satisfying the following. 

Let $\mu_{i}\in\mathfrak{M}_{i}$ be such that $\mu_{1},\ldots,\mu_{k}$ are pairwise commuting (in particular each $\mu_i$ commutes with itself).
 Assume that $\mu_{1}\otimes\ldots\otimes\mu_{k}\left(\exists x\bigwedge_{i=1}^{k}\varphi_{i}\left(x,y_{i}\right)\right)\geq\alpha$.
Then there is some $i\in [k]$ and some $a\in \M^{x}$
such that $\mu_{i}\left(\varphi_{i}\left(a,y_{i}\right)\right)\geq\gamma$.
\end{prop}

\begin{proof}
Let $\varphi_{i}\left(x,y_{i}\right) \in L$ and $\mu_i \in \mathfrak{M}_i$ for $i \in [k]$ be given,
and let $\alpha>0$ be arbitrary.

By Definition \ref{def: burden} and compactness, for every $d\in\mathbb{N}$
there is some $D=D\left(d\right)\in\mathbb{N}$ such that there is 
\textbf{no} rectangular array $\left(b_{i,j}\in \M^{y_{i}}: i \in [k], j \in [D]\right)$
satisfying
\begin{lyxlist}{00.00.0000}
\item [{$\left(*\right)$}] $\left\{ \varphi_{i}\left(x,b_{i,f\left(i\right)}\right): i \in [k] \right\} $
is consistent for any $f:\left[k\right]\to\left[D\right]$,
\item [{~}] $\left\{ \varphi_{i}\left(x,b_{i,j}\right): j \in [D] \right\} $
is $d$-inconsistent for any $1\leq i\leq k$.
\end{lyxlist}
By assumption $\varphi_{i}$ satisfies $\FHP_{d_{i}}$ for some $d_{i}\in\mathbb{N}$,
let $d:=\max\left\{ d_{i}:1\leq i\leq k\right\} $, and let $D=D\left(d\right)$
be as above.

For each $i \in [k]$, let $\psi_{i}\left(\bar{y}_{i}\right) := \bigwedge_{s\in{[D] \choose d}}\left(\neg\exists x\bigwedge_{j\in s}\varphi_{i}\left(x,y_{i,j}\right)\right)$,
where $\bar{y}_{i}=y_{i,1}\ldots y_{i,D}$ with all of $y_{i,j}$
of the same sort as $y_i$. Let $\psi\left(\bar{y}\right):=\bigwedge_{1\leq i\leq k}\psi_{i}\left(\bar{y}_{i}\right)$,
where $\bar{y}=\bar{y}_{1}\ldots\bar{y}_{k}$.

Fix $\beta = \beta(\alpha, k, D)>0$, to be determined later.

If $\mu_{i}^{\otimes d }\left(\exists x\bigwedge_{j \in [d]}\varphi_{i}\left(x,y_{i,j}\right)\right)\geq\beta$
for some $i \in [k]$, then, as $\varphi_{i}(x,y_i)$ satisfies $\FHP_{d}$ relatively
to $\mathfrak{M}_{i}$, there is some $\gamma_{i}=\gamma_{i}\left(\beta\right)>0$
and some $a\in \M^{x}$ such that $\mu_{i}\left(\varphi\left(a,y_{i}\right)\right)\geq\gamma_{i}$.
Then taking $\gamma:=\min\left\{ \gamma_{i}: i \in [k]\right\} $
we would be done.

So assume that $\mu_{i}^{\otimes d}\left(\exists x\bigwedge_{j \in [d]}\varphi_{i}\left(x,y_{i,j}\right)\right)<\beta$
for all $i \in [k]$. Then, by the union bound, 
$$\mu_{i}^{\otimes D}\left(\neg\psi_{i}\left(\bar{y}_{i}\right)\right)\leq{D \choose d}\mu_{i}^{\otimes d}\left(\exists x\bigwedge_{j \in [d]}\varphi_{i}\left(x,y_{i,j}\right)\right)\leq{D \choose d}\beta,$$
hence $\mu_{1}^{\otimes D}\otimes\ldots\otimes\mu_{k}^{\otimes D}\left(\neg\psi\left(\bar{y}\right)\right)\leq k{D \choose d}\beta$,
and so $\mu_{1}^{\otimes D}\otimes\ldots\otimes\mu_{k}^{\otimes D}\left(\psi\left(\bar{y}\right)\right)\geq1-k{D \choose d}\beta$.

On the other hand, we are assuming that $\mu_{1}\otimes\ldots\otimes\mu_{k}\left(\exists x\bigwedge_{i=1}^{k}\varphi_{i}\left(x,y_{i}\right)\right)\geq\alpha$.
Let $\theta\left(\bar{y}\right):=\bigwedge_{\left(j_{1},\ldots,j_{k}\right)\in [D]^{k}}\left(\exists x\bigwedge_{1\leq i\leq k}\varphi_{i}\left(x,y_{i,j_{i}}\right)\right)$.
Then applying Lemma \ref{lem: Erdos-Stone} to $R\left(y_{1},\ldots,y_{k}\right):=\exists x\bigwedge_{i=1}^{k}\varphi_{i}\left(x,y_{i}\right)$
(noting that $\Sigma_{d}\left(\bar{y}\right)=\theta\left(\bar{y}\right)$
in this case), we have that $\mu_{1}^{\otimes D}\otimes\ldots\otimes\mu_k^{\otimes D}\left(\theta\left(\bar{y}\right)\right)\geq\alpha^{Dk}$.

Let $\mu:=\mu_{1}^{\otimes D}\otimes\ldots\otimes\mu_{k}^{\otimes D}$.
Then, if $k{D \choose d}\beta<\alpha^{Dk}$, we have $\mu\left(\psi\left(\bar{y}\right)\land\theta\left(\bar{y}\right)\right)>0$.
But any realization $\left(b_{i,j}: i \in [k], j \in [D] \right)$
of $\psi\left(\bar{y}\right)\land\theta\left(\bar{y}\right)$ would
satisfy $\left(*\right)$ --- contradicting the choice of $D$.

Hence taking $\beta>0$ sufficiently small compared to $\alpha,k$
and $D$, we see that $\gamma>0$ above can be chosen depending only
on $\varphi_{1},\ldots,\varphi_{k}$ and $\alpha$, as wanted.
\end{proof}

\begin{rem}
	Conversely, assume $\bdn(\M^x) \geq k$, witnessed by $\varphi_i(x,y_i)$, $k_i$ and $\bar{a}=(a_{i,j})_{j \in \omega}$ for $i \in [k]$. Then for every $\gamma > 0$, taking $\mu_i \in \mathfrak{M}^{\fin}_{y_i}(\M)$ to be the measure supported on the finite set $\{b_{i,j} : j \in [n]\}$ with sufficiently large $n$ so that $\gamma > k_i/n$, we have $\bigotimes_{i \in [k]}\mu_i(\exists x\bigwedge_{i=1}^{k}\varphi_{i}\left(x,y_{i}\right)) = 1$, but $\mu_i(\varphi(a,y_i)) < \gamma$ for all $i \in [k]$ and $a \in \M^x$. Hence Proposition \ref{prop: colorful FHP meas bdd by burden} gives a measure-theoretic characterization of burden in $\FHP$ theories.
\end{rem}
\begin{rem}\label{rem: generalizing Pillay}
	Proposition \ref{prop: colorful FHP meas bdd by burden} generalizes and refines \cite{pillay2013weight} applied with $\mathfrak{M}_i = \mathfrak{M}^{\fam}$, which in an NIP theory corresponds to the class of generically stable measures. But also, by Theorem \ref{thm: FHP in MS meas}, applies to pseudofinite fields with the class $\mathfrak{M}_i$ given by localizing the ultralimit of counting measures to arbitrary definable sets.
\end{rem}

A partitioned formula $\varphi(x,y) \in L$ has (dual) \emph{$\VC$-density} $\leq \ell \in \mathbb{R}$ if there exists $K \in \mathbb{R}$ such that, for all finite $B \subseteq \M^y$, $|S_{\varphi}(B)| \leq K |B|^{\ell}$ (note that $|S_{\varphi}(B)|$ is equal to the dual shatter function $\pi^{\ast}_{\mathcal{F}_{\varphi}}$ for the family $\mathcal{F}_{\varphi} = \{\varphi(\M,b) : b \in \M^y\}$ of subsets of $\M^x$). And we let its $\VC$-density $\vc^{\ast}(\varphi)$ be the infimum of all such $\ell \in \mathbb{R}$. For a complete theory $T$ and a tuple of variables $x$, we let $\vc^{\ast}_T(x)$ denote the supremum of $\vc^{\ast}(\varphi)$ over all formulas $\varphi(x,y) \in L$ (with $x$ fixed and $y$ arbitrary).  The following observation appears in \cite{aschenbrenner2016vapnik, kaplan2013additivity, guingona2014vapnik}:
\begin{fact}\label{fac: vc-dens bounds burden}
For any theory $T$ and tuple of variables $x$, $\bdn(\M^x) \leq \vc^{\ast}_T(x)$.
\end{fact}

\noindent It is a well-known open problem (stated in various variants in \cite{aschenbrenner2016vapnik, kaplan2013additivity, guingona2014vapnik}) if there exists a function $f: \mathbb{N} \to \mathbb{N}$ so that for any NIP theory $T$ we have $ \vc^{\ast}_T(x) \leq f \left( \bdn(\M^x) \right)$ (one can take $f$ to be linear in all known examples). Matousek's theorem (Fact \ref{fact: FHP for counting}) demonstrates that the fractional Helly number of a formula $\varphi(x,y)$
is bounded by its dual VC-density $\vc^{\ast}(\varphi)$. Proposition \ref{prop: colorful FHP meas bdd by burden} implies that in fact it is bounded by the burden/dp-rank of $\M^x$: 
\begin{cor}\label{cor: FHPk bounded by burden}
In any FHP (so e.g.~in NIP) theory $T$, the fractional Helly number of a formula $\varphi\left(x,y\right)$ is at
most $\bdn\left(\M^{x} \right) +1$.
\end{cor}

\begin{proof}
Applying Proposition \ref{prop: colorful FHP meas bdd by burden} with  $\varphi_1 = \ldots = \varphi_k := \varphi$ and arbitrary $\mu_1 = \ldots = \mu_k \in \mathfrak{M}^{\fin}_y$ (using Proposition \ref{rem: FHP iff FHP for finite measures}).
\end{proof}

The following ``colorful'' version of the fractional Helly property
was proved for convex sets in \cite{barany2014colourful}
(see also \cite{kim2015note}):
\begin{thm}
Let $\mathcal{F}_{1},\ldots,\mathcal{F}_{d+1}$ be finite non-empty
families of convex sets in $\mathbb{R}^{d}$, and assume that $\alpha\in\left(0,1\right]$.
If at least $\alpha\left|\mathcal{F}_{1}\right|\ldots\left|\mathcal{F}_{d+1}\right|$
tuples of the form $\left(S_{1},\ldots,S_{d+1}\right)$, $S_{i}\in\mathcal{F}_{i}$,
have a non-empty intersection, then some $\mathcal{F}_{i}$ contains
an intersecting subfamily of size $\beta\left|\mathcal{F}_{i}\right|$,
with $\beta=\frac{\alpha}{d+1}$.
\end{thm}

\noindent Here ``colors'' correspond to several families of sets instead of
one, and the usual FHP follows by taking all of these families to
be equal to each other. Proposition \ref{prop: colorful FHP meas bdd by burden} (combined with Fact \ref{fac: vc-dens bounds burden}) implies a generalization
of Matousek's theorem (Fact \ref{fact: FHP for counting}) to a colorful version:
\begin{cor}
\label{cor: colorful Matousek}Assume that $\mathcal{F}$ is a family
of sets with $\pi^{\ast}_{\mathcal{F}}\left(n\right)=o\left(n^{d}\right)$.
Let $\alpha\in\left(0,1\right]$. Then there is some $\beta$ satisfying
the following. Let $\mathcal{F}_{1},\ldots,\mathcal{F}_{d}\subseteq\mathcal{F}$
be finite non-empty families. If at least $\alpha\left|\mathcal{F}_{1}\right|\ldots\left|\mathcal{F}_{d}\right|$
tuples of the form $\left(S_{1},\ldots,S_{d}\right)$, $S_{i}\in\mathcal{F}_{i}$,
have a non-empty intersection, then some $\mathcal{F}_{i}$ contains
an intersecting subfamily of size $\beta\left|\mathcal{F}_{i}\right|$.
\end{cor}

\subsection{(Colorful) $\left(p,q\right)$-theorem and finite VC-dimension}\label{sec: p,q theorem and NIP}

A celebrated result of Alon and Kleitman  established a long standing conjecture of  Hadwiger and Debrunner:
\begin{fact}\cite{alon1992piercing}\label{fac: convex pq}
	Let $p,q,d$ be integers with $p \geq q \geq d+1$. Then there exists an integer  $N = N(d,p,q)$ such that: if $\mathcal{F}$ is a finite family of convex subsets of $\mathbb{R}^{d}$ satisfying the $(p,q)$-property (Definition \ref{def: weak FHP comb}), then there exists a set $A \subseteq \mathbb{R}^d$ with $|A| \leq N$ so that $A \cap S \neq \emptyset$ for every $S \in \mathcal{F}$.
\end{fact}

\noindent Their proof combines two earlier results about convex sets: the fractional Helly property \cite{katchalski1979problem} and the existence of (weak) $\varepsilon$-nets \cite{alon1992point}. Using his result on the fractional Helly property for families of sets of finite VC-dimension and the existence of $\varepsilon$-nets \cite{haussler1986epsilon}, Matousek obtained  an analog for families of finite VC-dimension:

\begin{fact}\cite[Theorem 4]{matousek2004bounded}\label{fac: pq Matousek}
	Let $p,q,d$ be integers with $p \geq q \geq d$. Then there exists an integer  $N = N(d,p,q)$ such that: if $\mathcal{F}$ is a finite family of sets with 
	$\pi_{\mathcal{F}}^{\ast}(n) = o(n^{d})$ satisfying the $(p,q)$-property, then there exists a set $A \subseteq \mathbb{R}^d$ with $|A| \leq N$ so that $A \cap S \neq \emptyset$ for every $S \in \mathcal{F}$.
\end{fact}

Matousek's $(p,q)$-theorem (Fact \ref{fact: FHP for counting})  played an important role in the study of $\NIP$
theories, starting with the proof of the uniform definability
of types over finite sets (UDTFS) in NIP theories \cite{chernikov2013externally, chernikov2015externally} (also in the study of definably amenable NIP groups \cite{chernikov2018definably}; and more recently in the proof of the \emph{definable $(p,q)$-conjecture} of Chernikov and Simon in \cite{kaplan2024definable}). In \cite{barany2014colourful} the authors obtain  a colorful $(p,q)$-theorem for convex sets, relying on their colorful fractional Helly property. Similarly, using the colorful fractional Helly property (Corollary \ref{cor: colorful Matousek}), one can obtain a colorful $(p,q)$-theorem for families of sets of finite $\VC$-dimension.

As we demonstrate in this paper, the class of structures in which all definable families of sets satisfy the fractional Helly property is much wider than the class of NIP structures.  Here we point out however that at the level of the theory, the
$\left(p,q\right)$-theorem characterizes NIP.
%
%
%
%
%
%
\begin{defn}
For $d \in \omega$, we say that a partitioned formula $\varphi\left(x,y\right) \in L$ is \emph{$d$-pierceable}
if for any $ p \geq q \geq d$ there is some $N=N\left(p,q\right)\in\omega$
such that, taking  $\mathcal{F}=\left\{ \varphi\left(\M,b\right):b\in \M^{y}\right\} $,  if a finite subfamily $\mathcal{F}'\subseteq\mathcal{F}$
satisfies the $\left(p,q\right)$-property, then it admits a transversal
of size at most $N$. A formula is pierceable if it is $d$-pierceable for some $d \in \omega$. A theory $T$ is pierceable if every formula is pierceable.
\end{defn}

So by Fact \ref{fac: pq Matousek}, every NIP theory is pierceable. Conversely, we have:
\begin{prop}\label{prop: pierceable implies NIP}
Assume that the formula $\varphi\left(x,y\right)$ is not NIP (i.e.~it defines a family of sets of infinite VC-dimension). Then the
formula $\psi\left(x;y_{1}, y_{2}\right) := \varphi\left(x,y_{1}\right)\land\neg\varphi\left(x,y_{2}\right)$
is not pierceable.
\end{prop}

\begin{proof}
As $\varphi(x,y)$ has $\IP$, for any $m \in \omega$ there is  $S\subseteq \M^{y}$ with $|S| = m$ so that: for any $S'\subseteq S$,
there is some  $e\in \M^{x}$ such that for all $c\in S$ we
have $\models\varphi\left(e,c\right)\iff c\in S'$. Let $\left(\left(a_{i},b_{i}\right):i<n\right)$, for $n = m^2 - m$, list \textbf{all}
pairs in $S^{2}\setminus\Delta$ (where $\Delta = \{(a,b) : a = b\}$ denotes the diagonal).

\medskip

\noindent \textbf{Claim 1.} The family $\left\{ \psi\left(\M;a_{i},b_{i}\right):i<n\right\} $
satisfies the $\left(4,2\right)$-property.

Let $\left\{ \left(a_{i}', b_{i}'\right):i<4\right\} $ be arbitrary
pairs from $S^{2}\setminus\Delta$ witnessing failure of the $\left(4,2\right)$-property. From pairwise inconsistency and assumption on $S$,
for any $i\neq j < 4$ we must have $a_{i}=b_{j}$ or $a_{j}=b_{i}$.
But this is easily seen to contradict the assumption $a_{i}\neq b_{i}$
for all $i < 4$.

\medskip

\noindent \textbf{Claim 2. }For any $2 \leq q\in\omega$, there is some $p\in\omega$
such that the family $\left\{ \psi\left(\M;a_{i},b_{i}\right):i<n\right\} $
satisfies the $\left(p,q\right)$-property (as long as $n$ is sufficiently large with respect to $p,q$).

Fix $q \geq 2$. Let $\Gamma\left(\bar{y},\bar{y}'\right) := \left\{ \exists x\left(\psi\left(x,\bar{y}\right)\land\psi\left(x,\bar{y}'\right)\right)\right\} $. By Ramsey's theorem we can choose $p = p(q)$ large enough, such that every
sequence $\left(\bar{d}_{i}\right)$ from $S^{2}\setminus\Delta$
of length $p$ contains a $\Gamma$-indiscernible subsequence of
length $q' := \max\left\{ q,4\right\} $.

Assume now that the family $\left\{ \psi\left( \M ; (a_i,b_i)\right):i<n\right\} $
fails the $\left(p,q\right)$-property. That is, there is some $I \subseteq n$ with $|I| = p$ so that the family $\left\{ \psi\left(x; a_i,b_i \right):i \in I \right\} $  is $q$-inconsistent.
By the choice of $p$, there is some $J \subseteq I$ with $|J| = q'$ so that the sequence $(a_i,b_i)_{i \in J}$ is $\Gamma$-indiscernible, and we still have that $\{\psi(x;a_i, b_i) : i \in J\}$ is inconsistent.

But by definition of $\psi$ and assumption on $S$, this can only happen if already
$\{ \psi\left(x; a_i,b_i\right), \psi\left(x;a_j,b_j\right) \}$
is inconsistent for some $i \neq j \in J$. But then, by $\Gamma$-indiscernibility of $\left(a_i,b_i : i \in J\right)$, 
the set $\left\{ \psi\left(x; a_i, b_i \right):i \in J\right\} $
is $2$-inconsistent. As $q' \geq 4$, it follows that  $\left\{ \psi\left(x,\bar{d}_{i}\right):i<n\right\} $
fails the $\left(2,4\right)$-property -- contradicting Claim 1.

\medskip

\noindent \textbf{Claim 3. }The family $\left\{ \psi\left(\M;a_{i},b_{i}\right):i<n\right\} $
does not admit any transversal of size independent of $n$.

Fix $k$, and let $m$ be sufficiently large, to be determined later. Assume that the family admits a transversal of size $k$. That is,
we can choose a partition $D_{0},\ldots,D_{k-1}$ of $S^{2}\setminus\Delta$
such that each of the families $\left\{ \psi\left(x;a_{i},b_{i}\right):\left(a_{i},b_{i}\right)\in D_{l}\right\} $,
$l<k$, is consistent.

This implies in particular that $\pi_{y_1}(D_{l}) \cap \pi_{y_2}(D_{l})=\emptyset$
for each $l<k$, where $\pi_{y_i}$ denotes the projection onto the corresponding coordinate (as otherwise we have some $\left(a_{i},b_{i}\right),\left(a_{j},b_{j}\right)\in D_{l}$
with $i\neq j$, such that $a_{i}=b_{j}$ --- so $\psi\left(x;a_{i},b_{i}\right)\land\psi\left(x;a_{j}b_{j}\right)$
is inconsistent). We show that this is impossible.

Say $S=\left\{ c_{i}:i<m\right\} $.
For every $i < j < m$, $\left(c_{i},c_{j}\right)\in D_{l}$ for some $l<k$. 
By Ramsey (assuming $m\gg k$) there is a subsequence $\left(c'_{i}:i<m'\right)$
for $m' \gg k$, and some $l<k$ such that $\left(c_{i}',c_{j}'\right)\in D_{l}$
for all $i<j<m'$. In particular $\left(c_{0}',c_{1}'\right)$ and
$\left(c_{1}',c_{2}'\right)$ are both in $D_{l}$, thus $\pi_{y_1}(D_{l}) \cap \pi_{y_2}(D_{l}) \neq \emptyset$.

\medskip

Combining Claims 2 and 3, we see that $\psi(x;y_1,y_2)$ is not $d$-pierceable for any $d \in \omega$.
\end{proof}

\begin{cor}
Let $T$ be a complete first-order theory. Then $T$ is pierceable
if and only if $T$ is $\NIP$.
\end{cor}

\begin{rem}
Note that this corollary cannot hold at the level of a formula. Indeed, let $M=(\mathbb{R}^2, P; E)$ with two sorts $\mathbb{R}^2$ and $P$ consisting of all convex subsets of $\mathbb{R}^2$, and $E \subseteq \mathbb{R}^2 \times P$ the membership relation.

 It is well-known that the family of convex subsets of $\mathbb{R}^2$ has infinite VC-dimension (e.g.~it shatter
any finite subset of the unit circle), hence $E\left(x,y\right)$ is not $\NIP$.  However, by Fact \ref{fac: convex pq},  $E\left(x,y\right)$ is $3$-pierceable. 
\end{rem}

\begin{rem}
Also, $E(x,y)$ in this structure is $\FHP$ (by the fractional Helly property for convex sets \cite{katchalski1979problem}), hence in particular $E(x,y)$ is $\NTP_2$ (by Proposition \ref{prop: FHP implies NTP2}). 
 We observe however that the formula $\varphi(x;y_1,y_2) := E(x,y_1) \land \neg E(x,y_2)$ in this structure has $2$-$\TP_2$.
\end{rem}

\providecommand{\R}{\mathbb{R}}
\providecommand{\N}{\mathbb{N}}

\begin{proof}
Let $\mathcal{F}$ denote the family of  subsets of $\mathbb{R}^2$ of the form $A \cap (\mathbb{R}^2 \setminus B)$ with $A,B$ convex. Let $D\subseteq \R^2$ be the (convex) closed unit disk. 
 Since $\N^{<\N}$ (the set of all finite sequences of natural numbers) is countable and the border $\partial D$ is a circle, we may choose a family of pairwise disjoint closed arcs
$\bigl(I_s\bigr)_{s\in \N^{<\N}}$ on $\partial D$. 
For each $s\in\N^{<\N}$, let $\ell_s$ be the chord line joining the endpoints of $I_s$, and let $H_s$ be the \emph{closed} halfplane bounded by $\ell_s$ that contains the center of $D$.
Define the corresponding (convex) cap $C_s \ :=\ D\cap(\R^2\setminus H_s)$.
By choosing each chord $\ell_s$ sufficiently close to $\partial D$ (equivalently, choosing each arc $I_s$ sufficiently small), we may assume the caps $(C_s)_{s\in\N^{<\N}}$ are pairwise disjoint.

For $i,j\in\N$, define
\[
F_{i,j} \ :=\ \bigcup \bigl\{\, C_s \ :\ |s|> i \ \text{and}\ s(i)=j \,\bigr\},
\]
where $|s|$ is the length of $s$ and $s(i)$ is its $i$th entry (defined only when $i<|s|$).
We claim that each $F_{i,j}$ belongs to $\mathcal{F}$. Indeed, let
\[
B_{i,j} \ :=\ D\ \cap\ \bigcap \bigl\{\, H_s \ :\ |s|> i \ \text{and}\ s(i)=j \,\bigr\}.
\]
Then $B_{i,j}$ is convex, being an intersection of convex sets (the disk $D$ and halfplanes). And we have $F_{i,j}=D\cap(\R^2\setminus B_{i,j})\in\mathcal{F}$.
We will use the parameters $y_{i,j}:=F_{i,j}$ in the set sort $P$.

Fix $i\in\N$ and let $j\neq j' \in \mathbb{N}$.
No finite sequence $s$ can satisfy both $s(i)=j$ and $s(i)=j'$.
Hence $F_{i,j}$ and $F_{i,j'}$ are unions of \emph{disjoint} subfamilies of the pairwise disjoint caps $(C_s)$, therefore $F_{i,j}\cap F_{i,j'}=\emptyset$. 
Equivalently, the row $\{\varphi(x;y_{i,j}) : j\in\N\}$ is $2$--inconsistent.

Let $f:\N\to\N$ be any function, and for each $n\in\N$ let $s_n \ :=\ (f(0),f(1),\dots,f(n-1))\ \in\ \N^{<\N}$ 
be its  initial segment of length $n$. 
Then for every $i<n$ we have $|s_n|=n>i$ and $s_n(i)=f(i)$, so by construction $C_{s_n}\subseteq F_{i,f(i)}$, hence $
\bigcap_{i<n} F_{i,f(i)}\ \supseteq\ C_{s_n}\ \neq\ \emptyset$ for all $n\in\N$. 
Thus  $\{\varphi(x;y_{i,f(i)}) : i\in\N\}$ is finitely satisfiable in $M$. \end{proof}

%
%

\section{$f$-generics and forking in FHP theories}\label{sec: f-gen and forking FHP}

Several notions of largeness for definable sets (and their equivariant versions in definable group) play an important role in the model-theoretic study of tame classes of structures.  In this section we recall several notions of ``large'', or ``generic'' sets for definable groups studied in the literature, and discuss their relationship in $\FHP$ theories.

\subsection{Notions of genericity}\label{sec: notions of genericity}
\begin{defn}
An (abstract) group $G$ is \emph{amenable} if there is a left $G$-invariant
finitely additive probability measure on the Boolean
algebra of \emph{all} subsets of $G$.
\end{defn}

\begin{fact}
\label{fact: bi-inv measure}(see e.g.~\cite{Garrido}) 
\begin{enumerate}
\item If $G$ is
amenable, then there is a bi-invariant finitely additive probability measure on $\mathcal{P}\left(G\right)$.
	\item  All (virtually) solvable groups are amenable. Non-abelian free groups are not amenable.
\end{enumerate}
\end{fact}

\begin{defn}  A definable group $G = G(\M)$ is \emph{definably amenable} if there is a left-$G(\M)$-invariant Keisler measure $\mu \in \mathfrak{M}_{G}(\M)$ supported on $G$.
	
\end{defn}

\begin{fact} 
\begin{enumerate}
\item If $G$ is a definable group and $T$ admits a model $M$ so that $G(M)$ is amenable (as an abstract group), then $G$ is definably amenable. (For (1)--(4) see \cite[Section 5]{hrushovski2008groups}.)
	\item If $T$ is stable, then every definable group is definably amenable.
	\item Definably compact groups in $o$-minimal theories or in the $p$-adics (e.g.~$\textrm{SO}_3(\mathbb{R})$) are definably amenable (and satisfy a stronger condition \emph{fsg}, see also \cite{chernikov2025externally} for a discussion and references).
	\item The following NIP groups are not definably amenable: $\PSL_2(\mathbb{R}), \SL_2(\mathbb{Q}_p)$.
	\item If $G$ is dp-minimal, then it is definably amenable (\cite{stonestrom2023f}, answering a question from \cite{chernikov2014external}).
	\item If $G$ is pseudo-finite, then it is definably amenable (witnessed by the ultralimit of counting measures on finite groups whose ultraproduct is elementarily equivalent to $G$).
	\item If $T$ is small then every definable group is definably amenable (see \cite[Corollary 4.14]{chernikov2023invariant}).
	\item There exist definable groups in simple theories that are not definably amenable (\cite{chernikov2023invariant}).
\end{enumerate}
\end{fact}

\begin{defn}\label{defn: notions of genericity}
Let $G = G(\M)$ be a definable group, and $X = \varphi(\M) \subseteq G$ a definable subset, where $\varphi(x) \in L(\M)$.
\begin{enumerate}
	\item $X$ is \emph{generic} if there exist some finite $A \subseteq G$ so that $G = A \cdot X$.
	\item $X$ is \emph{weakly generic} if there is some non-generic definable set $Y \subseteq G$ so that $X \cup Y$ is generic (equivalently, for some finite $A \subseteq G$, $G \setminus \left( A \cdot X \right)$ is not generic) \cite{newelski2009topological}.
	\item $X$ is \emph{non null}, i.e.~there exists a left-$G$-invariant global Keisler measure $\mu \in \mathfrak{M}_{G}(\M)$ supported on $G$ so that $\mu(X) > 0$.
	\item $X$ is  \emph{$f$-generic} if there exists a small model $M \prec \M$ so that $g \cdot X$ does not fork over $M$ for all $g \in G(\M)$ \cite{chernikov2018definably}.
	\item $X$ \emph{does not $G$-divide} if there is no infinite sequence $(g_i)_{i < \omega}$ with $g_i \in G(\M)$ and $k \in \mathbb{N}$ so that the family of sets $\{g_i \cdot X : i < \omega\}$ is $k$-inconsistent \cite{chernikov2018definably}.
\end{enumerate}	
We considered the action of $G$ on the left in each of the cases above. Similarly, we can consider the action on the right, or discuss bi-generics of different kinds.
\end{defn}

We note that ``generic'' corresponds to ``syndetic'', and ``weak generic'' corresponds to ``piecewise syndetic'' in the ergodic theory terminology. 
We summarize what is known about the relation between these notions of large definable sets, in general and in particular classes of theories (some examples below are translations of the standard examples/facts in the ergodic theory literature).

\begin{fact}\label{fac: relation between diff generics} Let $G = G(\M)$ be a definable group in $\M \models T$ and $X = \varphi(\M) \subseteq G$ for $\varphi(x) \in L$.
	\begin{enumerate}
	\item 	Any $T$.
	\begin{enumerate}
		\item 	Complements of weakly generic sets and null sets are ideals in the Boolean algebra of definable sets (immediate from the definitions). 
		\item $X$ non-null $\Rightarrow$ $X$ does not $G$-divide (see Theorem \ref{thm: generics in FHP}); generic $\Rightarrow$ weak generic. If forking equals dividing over models (e.g.~$T$ is $\NTP_2$), then $f$-generic is equivalent to non-$G$-dividing (see e.g.~\cite[Lemma 3.7]{montenegro2016groups}).
		\item  If $G$ is amenable (as a discrete group), then weak generic $\Rightarrow$ non-null (if $X$ is weak generic, then there is some finite set $A\subseteq G$ so that: for every finite $B \subseteq G$, there is $g \in G$ with $B g \subseteq AX$; then, given a right-almost-invariant Følner net $(B_i)_{i \in I}$ for $G$, can choose $g_i \in G$ with  $B_i \cdot g_i \subseteq AX$; taking an ultralimit, we thus find a bi-$G$-invariant measure $\mu$ with $\mu(AX) = 1$, hence $\mu(X) \geq \frac{1}{|A|} > 0$). 
		\end{enumerate}
	
		\item Assume $T$ is NIP.
		\begin{enumerate}
		\item $G$ is fsg \cite{hrushovski2008groups, hrushovski2011nip, hrushovski2012note} (so e.g.~$T$ is stable and $G$ is arbitrary), or more generally $G$ is definably amenable and admits a global generic type (see \cite[Section 3.4]{chernikov2018definably}). Then $G$ is definably amenable, and all notions (1)--(5) in Definition \ref{defn: notions of genericity} are equivalent, and they are also equivalent to their counterpart under the right action of $G$. 
		\item $G$ is definably amenable if and only if $G$-dividing subsets of $G$ form an ideal, if and only if non-$f$-generic subsets of $G$ form an ideal (so if and only if weak genericity is equivalent to either $f$-genericity or non-$G$-dividing) \cite{stonestrom2023f}.
		\item If $G$ is definably amenable, then (2)--(5) in Definition \ref{defn: notions of genericity}  are equivalent \cite[Theorem 1.2]{chernikov2018definably}; but depend on which side $G$ acts, and (1) is strictly stronger in general (see \cite[Proposition 6.3, Example 6.4]{chernikov2018definably}. 
		\item Without assuming definable amenability, $X$ generic $\not \Rightarrow$ $X$ is non-$G$-dividing, even when $T$ is NIP (consider the group $G=\SL_2(\mathbb{R})$, then $X := \left\{ 	\begin{bmatrix}
    a & b  \\
   c & d 
\end{bmatrix} : |a| \geq |c| \right\}$ is generic and $G$-dividing witnessed e.g.~by $g_i = \begin{bmatrix}
    1 & 0  \\
   3i & 1 
\end{bmatrix}$ for $i \in \mathbb{N}$; see the discussion after Theorem 1.1 in \cite{stonestrom2023f}).

		\end{enumerate}
		\item Assume $T$ is NTP$_2$ and $G$ admits a global strongly $f$-generic type (e.g.~$T$ is simple and $G$ arbitrary; or $G$ is definably amenable \cite[Proposition 3.20]{montenegro2016groups}; or $G$ admits a global type with a bounded orbit, see the proof of  \cite[Theorem 3.12]{chernikov2018definably} in the case of types). 
		\begin{enumerate}
		\item $f$-generic sets form an ideal, 
	(4) and (5) are equivalent (see \cite[Proposition 3.10]{montenegro2016groups}), and weak generic $\Rightarrow$ $f$-generic (see the proof of \cite[Proposition 3.30]{chernikov2018definably}).
	
	\item Even for $T$ simple and $G$ amenable (as a discrete group), non $G$-dividing $\not \Rightarrow$ non-null (assuming Dickson's conjecture, see Proposition \ref{prop: Primes not FHP}); and non-null $\not \Rightarrow$ weak generic (even assuming $\FHP$ additionally, see Remark \ref{rem: square free positive density gen}).
		\end{enumerate}
 	\end{enumerate}
\end{fact}

\begin{problem}
	Relaxing amenability: is it true that if $G$ is a definably amenable group in an arbitrary theory and $X$ is weakly generic, then  $X$ is non null? 
\end{problem}

\subsection{Generics in amenable $\FHP$ groups}\label{sec: gen in amenable FHP groups}
We would like to connect forking and $f$-generics with
measures in $\FHP$ theories.  We will use some standard facts about
finitely additive probability measures.
\begin{fact}
\cite{los1949extensions} Let $S$ be a set and $\mathcal{B}_{0}\subseteq\mathcal{B}_{1}\subseteq\mathcal{P}\left(S\right)$
be Boolean subalgebras. Let $\mu$ be a finitely additive probability
measure on $\mathcal{B}_{0}$. Then there is a finitely additive probability
measure $\nu$ on $\mathcal{B}_{1}$ extending $\mu$. Moreover, for
any $X\in\mathcal{B}_{1}$ we can choose $\nu$ with $\nu\left(X\right)=r$
for any $r$ satisfying 
\[
\sup\left\{ \mu\left(L\right):L\in\mathcal{B}_{0},L\subseteq X\right\} \leq r\leq\inf\left\{ \mu\left(U\right):U\in\mathcal{B}_{0},X\subseteq U\right\} \mbox{.}
\]
\end{fact}


We will use the following criterion for the existence of
finitely additive probability measures due to Kelley (alternatively, we could use a version
of this criterion for finitely supported probability measures: 

\begin{defn}\label{def: inters numb}
	Let $\mathcal{B}$ be a Boolean algebra
of subsets of $X$, and let $\mathcal{F}\subseteq\mathcal{B}\setminus\left\{ \emptyset\right\} $.
Given a finite sequence $\bar{S}=\left(S_{1},\ldots,S_{n}\right)$
of sets from $\mathcal{F}$, possibly with repetitions, let $i\left(\bar{S}\right):=\frac{k}{n}$,
where $k$ is the largest size of a subset $J\subseteq\left[n\right]$
such that $\bigcap_{i\in J}S_{i}\neq\emptyset$. We define the \emph{intersection
number of $\mathcal{F}$ as 
\[
i\left(\mathcal{F}\right):=\inf\left\{ i\left(\bar{S}\right):\bar{S}\text{ is a finite sequence of sets from }\mathcal{F}\right\} \text{.}
\]
}
\end{defn}

\begin{fact}
\cite{kelley1959measures} \label{fact: Kelley} Let $\mathcal{F}\subseteq\mathcal{B}\setminus\left\{ \emptyset \right\} $
be given.
\begin{enumerate}
\item Let $\mu$ be a finitely additive probability measure on $\mathcal{B}$ and $\alpha >0$ such that $\mu\left(S\right)\geq\alpha$
for all $S\in\mathcal{F}$. Then $i\left(\mathcal{F}\right)\geq\alpha$.
\item Conversely, if $i\left(\mathcal{F}\right)=\alpha>0$, then there exists
a finitely additive probability measure $\mu$ on $\mathcal{B}$ such that $\mu\left(S\right)\geq\alpha$
for all $S\in\mathcal{F}$.
\end{enumerate}
\end{fact}

\begin{thm}
\label{thm: generics in FHP} Assume $T$ is weakly FHP (Definition \ref{def: weak FHP comb}) and $\NTP_2$ (so for example if $T$ is FHP, by Proposition \ref{prop: FHP implies NTP2}) and $G = G(\M)$ is an ($\emptyset$-)definable group. Assume that there exists a model $M$ of $T$ so that $G(M)$ is amenable (as a discrete group). Then for any $L(\M)$-definable set $X = \varphi(\M) \subseteq G(\M)$, the following are equivalent:
\begin{enumerate}
\item $X$ is non-$G$-dividing;
\item $X$ is $f$-generic;
\item $X$ is non null.
\end{enumerate}
\end{thm}

\begin{rem}
	We note that if $G$ is an amenable group and $H \equiv G$ is elementarily equivalent to it (in the pure group language), then $H$ is not necessarily amenable. Groups for which this holds are called \emph{uniformly amenable} (see e.g.~\cite[Section 16.8]{dructu2018geometric} and references there; in particular \cite{wysoczanski1988uniformly} gives an example of an amenable group that is not uniformly amenable).
\end{rem}

\begin{proof}

(1) is equivalent to (2) by Fact \ref{fac: relation between diff generics}(1b) and Fact \ref{fact: forking in NTP2}, as $T$ is $\NTP_2$. And (3) implies (1) is well known: if $\mu \in \mathfrak{M}_{G}(\M)$ is left-$G(\M)$-invariant and $\mu(X) = \alpha >0$, then $\mu(g \cdot X) = \alpha$ for all $g \in G(\M)$ by left-invariance of $\mu$, hence $i(\mathcal{F}) \geq \alpha$ for $\mathcal{F} := \{g \cdot X : g \in G(\M)\}$ by Fact \ref{fact: Kelley}(1). Hence, if $(g_i)_{i \in \omega}$ is any indiscernible sequence in $G(\M)$ and $n \in \omega$ is arbitrary, we get that $\bigcap_{i \in I}g_i \cdot X \neq \emptyset$ for some $I \subseteq [n]$ with $|I| \geq \alpha n$. As $n$ was arbitrary, by indiscernibility this implies that $\{g_i \cdot X : i \in \omega \}$ is consistent, so $X$ does not $G$-divide.
So it remains to prove that (1) implies (3).

Assume $X = \varphi(\M,b)$, where $\varphi(x,y) \in L$ and $b$ in $\M$, does not $G$-divide (fixed for the rest of the proof). Let $\varphi'(x;y,y') := \varphi((y')^{-1} \cdot x; y)$, and let $k$ be such that $\varphi'(x;y,y')$ satisfies $\WFHP_{k}$ in $T$. Let $M \prec \M$ be small so that $G(M)$ is amenable, as a discrete group (we may assume that $M$ is small by L\"owenheim-Skolem, as subgroups of amenable groups are amenable).

 It follows by saturation of $\M$ that there is some $p\in\mathbb{N}$ (which we fix for the rest of the proof) 
such that the family of sets $\mathcal{F}:=\left\{ g \cdot X:g\in G(\M)\right\} $
satisfies the $\left(p, k\right)$-property, meaning that for any $g_1, \ldots, g_p \in G(\M)$ (possibly with repetitions), there is some $I \subseteq [p]$ with $|I| = k$ so that $\bigcap_{i \in I} g_i \cdot X \neq \emptyset$. Indeed, if no such $p$ existed, by saturation of $\M$ we could extract
an infinite indiscernible sequence $\left(g_{i}:i\in \omega\right)$ in $G(\M)$ so that $\left\{ g_{i} \cdot  X:i\in \omega\right\} $ is $k$-inconsistent --- contradicting non-$G$-dividing.

Now we work in $M$, and assume that $b_0$ in $M$ is such that the family $\mathcal{F}_0 := \{g \cdot \varphi(M, b_0) : g \in G(M)\}$ satisfies the $(p,k)$-property. 
Let $\bar{S}=\left(S_{1},\ldots,S_{n}\right)$ be any finite sequence
of sets from $\mathcal{F}_0$, possibly with repetitions. 
%
As $\mathcal{F}_0$ satisfies $\WFHP_{k}$ (by the choice of $k$), and without loss of generality $p$ is sufficiently large with respect to $k$, there is $\beta = \beta(p,k) >0$
and $J\subseteq\left[n\right]$ with $\left|J\right|\geq\beta n$
and $\bigcap_{i\in J}S_{i}\neq\emptyset$.
Hence $i\left(\bar{S}\right)\geq\beta$ (Definition \ref{def: inters numb}). 
As $\beta$ does not depend on the choice of $\bar{S}$, it follows
that $i\left(\mathcal{F}_0\right)\geq\beta$.

Applying Fact \ref{fact: Kelley}(2) we find some finitely additive probability  measure $\nu$
on $\mathcal{P}\left(G(M)\right)$ such that $\nu\left(g \cdot \varphi(M,b_0) \right)\geq\beta$
for all $g\in G(M)$.

Now let $\omega$ be a bi-invariant finitely additive probability measure on $\mathcal{P}\left(G(M)\right)$
given by amenability of $G(M)$ and Fact \ref{fact: bi-inv measure}. For any $X\subseteq G(M)$,
we define $\mu'\left(X\right):=\int_{G(M)}f_{X}\left(g\right)d\omega$,
where $f_{X}\left(g\right)=\nu\left(g \cdot X\right)$. (Formally, we can work in $M'$  the expansion of $M$ naming all subsets of all of its powers, where $\nu, \omega \in \mathfrak{M}_{G}(M')$ are definable Keisler measures over $\emptyset$, and let $\mu'$ be the definable convolution $\nu \ast \omega$, as studied in \cite{chernikov2022definable, chernikov2023definable, chernikov2024definable}.) It follows that $\mu'$ is a $G(M)$-invariant finitely additive probability measure on $\mathcal{P}\left(G(M)\right)$ and $\mu'\left( \varphi(M,b_0) \right)\geq\beta$. Letting $\mu \in \mathfrak{M}_{G}(M)$ be the restriction of $\mu'$ to definable sets, we get that $\mu$ is $G(M)$-invariant and $\mu\left( \varphi(M,b_0) \right)\geq\beta$.

Now we want to lift this implication on the existence of appropriate measures from $M$ to $\M$. We consider the following generalization of Construction $(*)$ from \cite[Section 2]{hrushovski2008groups}. Let  $N = (M, P, \mathbb{R}; (E_{\psi})_{\psi \in L})$ be the structure with sorts $M$ (equipped with its full $L$-structure), $P$ whose elements are the $G(M)$-invariant Keisler measures in $\mathfrak{M}_{G}(M)$ (with no additional structure on it) and $\mathbb{R}$ (equipped with the full structure of a real closed field), and for each $\psi(x,y) \in L$ the map $E_{\psi}: M^y \times P \to \mathbb{R}$ defined by $E_{\psi}(b,\mu) := \mu(\psi(x,b))$. Let $N^{\ast} = (M^{\ast}, P^{\ast}, \mathbb{R}^{\ast})$ be a saturated elementary extension of $N$. In particular $M^{\ast}$, with its induced $L$-structure, is a saturated model of $T$ --- so we may identify it with $\M$; and we have the standard part map $\st: \mathbb{R}^{\ast} \cap [0,1] \to \mathbb{R} \cap [0,1]$.  For any $m \in P^{\ast}$, $\psi(x,y) \in L$ and $b \in (M^{\ast})^{y}$, we define $\mu_{m}(\psi(x,b)) := \st \left( E_{\psi}(b,m)^{N^{\ast}} \right)$. By definition of $N$ and $N^{\ast} \succ N$, it follows that for every $m \in P^{\ast}$, $\mu_m \in \mathfrak{M}_{G}^{L}(M^{\ast})$ is a $G(M^{\ast})$-invariant Keisler measure (but of course not all such measures have to appear as $\mu_m$ for some $m \in P^{\ast}$).
 
 By the above we have that $N$ satisfies the following sentence (with parameter $\beta$ in $M$): 
 \begin{gather*}
 \forall b' \in M^y \left( \left( \forall g_1, \ldots g_p \in G    \bigvee_{I \in {[p] \choose k}} \exists x \bigwedge_{i \in I} \varphi'(x; b', g_i) \right) \rightarrow \exists m \in P (E_{\varphi}(b',m) \geq \beta) \right).
 \end{gather*}
 
 Then, as $N^{\ast} \succ N$, the same sentence holds in $N^{\ast}$. As $\mathcal{F}$ satisfies the $(p,k)$-property, it follows that $N^{\ast} \models \exists m \in P \left(  E_{\varphi}(b,m) \geq \beta \right)$. This implies $\mu_m(\varphi(x,b)) \geq \beta >0$, and $\mu_m \in \mathfrak{M}_{G}(\M)$ is $G(\M)$-invariant (using the identification of $M^{\ast}$ and $\M$) --- as wanted.
\end{proof}
\begin{problem}
Does the equivalence in Theorem \ref{thm: generics in FHP} hold only assuming that $G$ is \emph{definably} amenable?
\end{problem}

\begin{rem}
Let $\left(X,\mathcal{F}\right)$ be a set
system with $X$ finite. A \emph{fractional transversal} for $\mathcal{F}$
is a function $\varphi:X\to\left[0,1\right]$ such that for each $S\in\mathcal{F}$
we have $\sum_{x\in S}\varphi\left(x\right)\geq1$. The \emph{size} of a fractional transversal $\varphi$ is $\sum_{x\in X}\varphi\left(x\right)$,
and the \emph{fractional transversal number} $\tau^{*}\left(\mathcal{F}\right)$
is the infimum of the sizes of fractional transversals for $\mathcal{F}$.

We note that in our proof of Theorem \ref{thm: generics in FHP} we could instead use  the following result for finite families (which explicitly uses the duality of linear programming essentially equivalent to  Kelley's criterion in the finite case) plus compactness:
\begin{fact}
\label{fact: bdd frac transversal}\cite{alon2002transversal} For
every $d,p$ there exists $\alpha>0$ such that: for any finite family $\mathcal{F}$ satisfying $\FH\left(d+1,\alpha,\beta\right)$
with some $\beta>0$ and the $\left(p,d+1\right)$-property, we have
$\tau^{*}\left(\mathcal{F}\right)\leq T$ for some $T=T\left(p,d,\beta\right)$.
\end{fact}
\end{rem}

\subsection{Generics in fim FHP groups}\label{sec: fim groups}

\begin{defn}\cite{hrushovski2008groups}
	A definable group $G$ is \emph{fsg} (\emph{finitely satisfiable generics}) if there is some $p \in S_{G}(\M)$ and small $M \prec \M$ such that for every $g \in G(\M)$, $g \cdot p$ is finitely satisfiable in $G(M)$.
\end{defn}

\begin{fact}\label{fac: generic type fsg}\cite[Proposition 4.2]{hrushovski2008groups}
	Let $T$ be any theory and $G = G(\M)$ a definable fsg group, witnessed by some $p \in S_{G}(\M)$. Then $p$ is a (two-sided) generic type (i.e.~$\varphi(x) \in p$ implies $\varphi(\M)$ is both left and right generic in $G$).
\end{fact}

In \cite[Section 3.6]{chernikov2024definable}, a generalization of fsg groups from NIP theories to fim groups in arbitrary theories was proposed, demonstrating that part of the theory of fsg groups in NIP theories survives:

\begin{defn}\label{def:fim} \cite{hrushovski2013generically}
Let $\mu \in \mathfrak{M}_x(\M)$ and $M \prec \M$  a small model. A Borel-definable measure $\mu$ is \emph{fim} (a \emph{frequency interpretation measure}) over $M$ if $\mu$ is $M$-invariant and for any $\mathcal{L}$-formula $\varphi(x,y)$ there exists a sequence of formulas $(\theta_{n}(x_1,\ldots,x_n))_{1 \leq n < \omega}$ in $\mathcal{L}(M)$ such that: 
	\begin{enumerate}
		\item for any $\varepsilon > 0$, there exists some $n_{\varepsilon} \in \omega$ satisfying: for any $k \geq n_{\varepsilon}$, if $\M \models \theta_{k}(\bar{a})$ then 
		\begin{equation*} 
		\sup_{b \in \mathcal{U}^{y}} |\Av(\bar{a})(\varphi(x,b)) - \mu(\varphi(x,b))| < \varepsilon;
		\end{equation*} 
		\item $\lim_{n \to \infty} \mu^{(n)} \left( \theta_n \left(\bar{x} \right) \right) = 1$.
	\end{enumerate}
	
We say that $\mu$ is fim  if $\mu$ is fim over some small $M \prec \M$. 
\end{defn} 

\begin{rem}\label{rem: fim in NIP is gen stab}
In NIP theories, fim is equivalent to 	
	each of the following two properties for measures: \emph{dfs} (definable and finitely satisfiable) and fam, recovering the usual notion of generic stability for Keisler measures  \cite{hrushovski2013generically}. Outside of the NIP context, fim (properly) implies fam over a model, which in turn (properly)  implies dfs (see \cite{conant2020remarks, conant2023keisler}; and \cite[Section 3]{chernikov2024definable} for further properties of fim measures). 
\end{rem}

\begin{fact}\cite{conant2023keisler}\label{fac: fim self commute}
	In any theory, if $\mu$ is fim,  $n \in \omega$  and $\sigma: [n] \to [n]$ is a permutation, then $\bigotimes_{i \in [n]} \mu_{x_{i}} = \bigotimes_{i \in [n]} \mu_{x_{\sigma(i)}}$.
\end{fact}

\begin{defn}\label{def: fim group}\cite[Definition 3.32]{chernikov2024definable}
	An ($\emptyset$-)definable group $G = G(\M)$ is \emph{fim} if there exists a left $G$-invariant fim  measure $\mu \in \frak{M}_{G}(\M)$.
\end{defn}

\begin{rem}
In any theory, if $G$ is fim then it is both definably amenable and   fsg. If $T$ is NIP, $G$ is fsg if and only if is fim (\cite{hrushovski2008groups, hrushovski2013generically}).
\end{rem}

\begin{fact}\label{prop:unique}\cite[Proposition 3.33]{chernikov2024definable} Suppose that $G = G(\M)$ is a $\emptyset$-definable fim  group, witnessed by a left-$G$-invariant fim measure $\mu \in \frak{M}_{G}(\M)$. Then $\mu$ is both the unique left $G$-invariant measure in $\mathfrak{M}_{G}(\M)$ and the unique right $G$-invariant measure in $\mathfrak{M}_{G}(\M)$.
\end{fact} 

\begin{lem}\label{lem: fim generic measure}
	If $G = G(\M)$ is fim, witnessed by $\mu \in \mathfrak{M}_{G}(\M)$, then a definable set $X \subseteq G$ is generic if and only if $\mu(X) > 0$.
\end{lem}
\begin{proof}
	If $X$ is generic then $A \cdot X = G$ for some finite $A \subseteq G$, hence $\mu(a \cdot X) \geq \frac{1}{|A|} > 0$ for some $a \in A$, so $\mu(X) > 0$ by $G$-invariance of $\mu$.
	
	Conversely, assume $\mu(X) > 0$, and say $X = \varphi(\M)$ for some $\varphi(x) \in L(\M)$. Then there exists some type $p \in S(\mu)$ in the support of $\mu$ so that $\varphi(x) \in p$. As $\mu$ is fim over some small $M \prec \M$, it is in particular finitely satisfiable in $M$, hence $g \cdot p$ is also finitely satisfiable in $M$ for every $g \in G(\M)$. So $p$ witnesses that $G$ is fsg, hence $p$ is a (two-sided) generic in $G$ by Fact \ref{fac: generic type fsg}, and so is $\varphi(x)$.
\end{proof}

Here we show that under the FHP assumption, all notions of genericity agree in fim groups (generalizing Fact \ref{fac: relation between diff generics}(2a)):
\begin{prop}\label{prop : generics in fim}
	Let $G = G(\M)$ be a definable fim group in an $\FHP$ theory. Then all notions of genericity (1)--(5) in Definition \ref{defn: notions of genericity} are equivalent for definable subsets of $G$.
\end{prop}
\begin{proof}

Every $\FHP$ theory is $\NTP_2$ by Proposition \ref{prop: FHP implies NTP2}, so $f$-generic is equivalent to non-$G$-dividing by Fact \ref{fac: relation between diff generics}(1b), and weak generic implies  $f$-generic by definable amenability, $\NTP_2$ and Fact \ref{fac: relation between diff generics}(3a). In view of Lemma \ref{lem: fim generic measure},  it remains to show that if $X \subseteq G$ is non-generic, then $X$ is $G$-dividing. 

We follow the proof of \cite[Proposition 3.2]{hrushovski2012note}, with some modifications. Let $M \prec \M$ be a small model so that $X$ is $M$-definable.  Let $\mu \in \mathfrak{M}_{G}(\M)$ be the fim $G$-invariant global Keisler measure, fim over $M$. Let $\varphi(x,y) \in L(M)$ denote the formula defining the set $\{(x,y) \in G \times G : y \in x \cdot X\}$, so for $b \in G$, $\varphi(x,b)$ defines the set $b \cdot X^{-1}$.

	As $X$ is non-generic, also $b \cdot X^{-1}$ is non-generic for all $b \in G$. By Lemma \ref{lem: fim generic measure}, $\mu(\varphi(x,b)) = 0$ for all $b \in \M^y$ (using that $\mu$ is supported on $G$). As $T$ is $\FHP$ and $\mu$ is fim (so in particular fam),  by Proposition \ref{prop: FHP for gen. stable} (with the roles of the variables exchanged) there exists some $k \in \omega$ so that $\mu^{\otimes k}_{x_1, \ldots, x_k} \left(\exists y \bigwedge_{i \in [k]} \varphi(x_i,y) \right) = 0$. Then, for an arbitrary $n \in \omega$, we have $\mu^{\otimes n}_{x_1, \ldots, x_n}\left( \bigvee_{I \in {n \choose k} } \exists y  \bigwedge_{i \in I} \varphi(x_i,y) \right) = 0$ (as $\left( \mu^{\otimes n} \right) |_{(x_i : i \in I)} = \bigotimes_{i \in I} \mu_{x_i}$ by Fact \ref{fac: fim self commute} and $\mu^{\otimes n}$ extends the product measure on the product Boolean algebra, so we are taking the measure of a union of finitely many sets of measure $0$). Then, taking an arbitrary tuple in the measure $1$ complement of this definable set, we find some $(g_i : i \in [n])$ in $G(\M)$ so that $\{ \varphi(g_i, y) : i \in [n] \}$ is $k$-inconsistent, and $\varphi(g_i,y)$ defines the set $g_i \cdot X$. By saturation of $\M$, we can thus find an infinite $M$-indiscernible sequence $(g_i)_{i \in \omega}$ in $G(\M)$ so that $\{g_i \cdot X : i \in \omega\}$ is $k$-inconsistent, showing that $X$ is $G$-dividing.
\end{proof}

\subsection{Forking in $\FHP$ theories}\label{sec: fork iff null in FHP}
In this section we make some remarks on an analogous question, connecting measures and forking in $\FHP$ theories, with respect to the action of the automorphism group $\Aut(\M)$ rather than a definable group.

\begin{defn}
$T$ is \emph{amenable} if every type over $\emptyset$ extends to a global $\Aut(\M/\emptyset)$-invariant measure \cite{hrushovski2022amenability, hrushovski2020first}. And a set $A \subseteq \M$ is an \emph{extension base} if every type $p \in S(A)$ extends to a global type non-forking over $A$. Note that if $T$ is amenable, then $\emptyset$ is an extension base (if $p \in S(\emptyset)$ and $\mu \in \mathfrak{M}(\M)$ extends it and is $\emptyset$-invariant, then any $p' \in S(\mu)$ in the support of $\mu$ gives a non-forking extension of $p$; and the converse holds when $T$ is NIP).
\end{defn}

\begin{fact} (See the introduction of \cite{chernikov2023invariant} and references there.)
	\begin{enumerate}
		\item In any theory, for any small set of parameters $A$, if $\mu(\varphi(x,b))>0$ for some $\Aut(\M/A)$-invariant Keisler measure $\mu \in \mathfrak{M}_x(\M)$, then $\varphi(x,b)$ does not fork over $A$.
		\item If $T$ is stable, the converse holds as well (i.e.~the forking and the  universal measure $0$ ideals coincide).
	\end{enumerate}
\end{fact}

\begin{fact}\cite{chernikov2023invariant}
	There exist a simple (SU rank $1$) amenable theory $T$ in which there is a formula $\varphi(x,b) \in L(\M)$ so that $\varphi(x,b)$ divides over $\emptyset$, but $\mu(\varphi(x,b))=0$ for every $\Aut(\M/\emptyset)$-invariant measure $\mu \in \mathfrak{M}_x(\M)$.
\end{fact}
\noindent Further examples with the same property which are additionally $\aleph_0$-categorical 
were found in \cite{marimon2025invariant}.

\begin{fact}\cite{pillay2023forking}
\begin{enumerate}
	\item There exist an NIP theory $T$ and formula $\varphi(x,b) \in L(\M)$ so that $\varphi(x,b)$ does not fork over $\emptyset$, but $\mu(\varphi(x,b))=0$ for every global $\emptyset$-invariant Keisler measure $\mu$.
	\item If $T$ is amenable, the formula $\varphi(x,b)$ forks over $\emptyset$ if and only if $\mu(\varphi(x,b))=0$ for all $\emptyset$-invariant global Keisler measures.
\end{enumerate}
\end{fact}

Combined with the results of Section \ref{sec: gen in amenable FHP groups}, these motivate the following question:
\begin{problem}
	Assume $T$ is $\FHP$ and amenable, is forking over $\emptyset$ equivalent to universal measure $0$? In particular, do examples in \cite{chernikov2023invariant, marimon2025invariant} have $\FHP$?
\end{problem}

We observe a weaker result in this direction.  First we recall some facts about dividing and forking in $\NTP_{2}$
theories. We write $a\ind_{C}b$ to denote that $\tp\left(a/bC\right)$
does not fork over $C$. 
\begin{fact}
\cite{chernikov2012forking}\label{fact: forking in NTP2} Let $T$
be an $\NTP_{2}$ theory, and let $A$ be an extension base.
\begin{enumerate}
\item $\varphi\left(x,b\right)$ forks over $A$ if and only if it divides
over $A$.
\item For any $b$ in $\M$ and a small model $M$, there is an $M$-indiscernible
sequence $\left(b_{i}:i\in\omega\right)$ with $b_{0}=b$ such that
for any formula $\varphi\left(x,y\right) \in L(M)$, $\varphi\left(x,b\right)$ divides
over $M$ if and only if $\left\{ \varphi\left(x,b_{i}\right):i\in\omega\right\} $
is inconsistent.
\end{enumerate}
\end{fact}

\begin{prop}\label{prop: forking in DFHP}
Let $T$ be $\NTP_2$ and   $\DFHP$  (i.e.~$\FHP$ for dividing, see Definition \ref{def: local FHP etc}(2)). Then the following are equivalent for any
small extension base $A$ and $b$ from $\M$:
\begin{enumerate}
\item $\varphi\left(x,b\right)$ does not fork/divide over $A$,
\item there is some global Keisler measure $\mu \in \mathfrak{M}_x(\M)$ non-forking over $A$ (i.e.~for any formula $\psi(x,b) \in L(\M)$, if $\mu(\psi(x,b))>0$ then $\psi(x,b)$ does not fork over $A$) 
and some $\varepsilon>0$ such that $\mu\left(\varphi\left(x,b'\right)\right)\geq\varepsilon$
for all $b'\equiv_{A}b$.
\end{enumerate}
\end{prop}

\begin{proof}
(2) $\Rightarrow$ (1) is by Fact \ref{fact: Kelley}(1). 

(1) $\Rightarrow$ (2): Assume that $\varphi\left(x,b\right)$ does not
divide over $A$. Assume $\varphi(x,y)$ satisfies $\DFHP_k$ for $k \in \omega$.  Then, as in particular $\varphi(x,b)$ does not $k$-divide, using $\DFHP_k$ and Kelley's criterion (Fact \ref{fact: Kelley})  
as in the proof of Theorem \ref{thm: generics in FHP}, there exist $\varepsilon > 0$ and a global Keisler measure
$\mu \in \mathfrak{M}_x(\M)$ such that:
\begin{lyxlist}{00.00.0000}
\item [{$\left(*\right)$}] $\mu\left(\varphi\left(x,b'\right)\right)\geq\varepsilon$
for all $b' \in \M^y$ with $b' \equiv_{A}b$.
\end{lyxlist}
We want to find a global Keisler measure $\nu$ satisfying $\left(*\right)$
and non-forking over $A$. Assume no such $\nu$ exists. As both of
these sets of measures are closed in $\mathfrak{M}_x(\M)$, and the space of global Keisler measures
is compact, there must exist some finitely many formulas $\psi_{i}\left(x,c\right)\in\mathcal{L}\left(\M\right),1\leq i\leq n$
each dividing over $A$ and some $\alpha>0$ such that for any $\nu$
satisfying $\left(*\right)$ we must have $\bigvee_{1\leq i\leq n}\nu\left(\psi_{i}\left(x,c\right)\right)>\alpha$. 

As $A$ is an extension base, let $M\supseteq A$ be a small model
such that $M\ind_{A}c$, and let $\left(c_{i}:i\in\omega\right)$
with $c_{0}=c$ be as given by Fact \ref{fact: forking in NTP2}(2).
Then for any $j\in\omega$ we must have $\bigvee_{1\leq i\leq n}\mu\left(\psi_{i}\left(x,c_{j}\right)\right)>\alpha$
(if $\bigvee_{1\leq i\leq n}\mu\left(\psi_{i}\left(x,c_{j}\right)\right)\leq\alpha$
for some $j\in\omega$, then as $\left(*\right)$ is an $\Aut\left(\M/A\right)$-invariant
condition, hence also $\Aut\left(\M/M\right)$-invariant, taking an
$M$-automorphism of $\M$ sending $c_{j}$ to $c$ and applying it
to $\mu$, we would get some measure $\nu$ still satisfying $\left(*\right)$
and violating the assumption above).

Note that by the choice of $M$ and left transitivity of forking (see
Lemma \cite{chernikov2012forking}), $\psi_{i}\left(x,c\right)$ divides
over $M$ for all $1\leq i\leq n$. On the other hand, there is some
$1\leq i\leq n$ and an infinite subsequence $I\subseteq\omega$ such
that $\bigwedge_{j\in I} \left( \mu\left(\psi_{i}\left(x,c_{j}\right)\right)>\alpha \right)$ holds. By Fact \ref{fact: Kelley}(1), compactness and indiscernibility this implies that $\{\psi_i(x,c_j): j \in \omega \}$ is consistent, hence $\psi_i(x,c)$ does not divide over $M$ (by the choice of $(c_j)_{j \in \omega}$) --- a contradiction.
\end{proof}

\begin{problem}
	Can (2) be strengthened to: there exists a \emph{weakly invariant} over $A$ measure $\mu \in \mathfrak{M}_x(\M)$ with $\mu(\varphi(x,b)) > 0$, where weakly invariant means that for every $\psi(x,c) \in L(\M)$, if $\mu(\psi(x,c)) >0$ then there exists some $\alpha >0$ so that $\mu(\varphi(x,c')) \geq \alpha$ for all $c' \equiv_{A} c$.
\end{problem}

\begin{rem}
	We note that conversely, any low theory satisfying the implication (1)$\Rightarrow$(2) in Proposition \ref{prop: forking in DFHP} is $\DFHP$. Indeed, assume $\varphi(x,y)$ is low, so that dividing for any instance of $\varphi$ implies $k$-dividing. Then if $\varphi(x,b)$ does not $k$-divide, it does not divide, so there exists $\varepsilon = \varepsilon (\tp(b))>0$ and measure $\mu$ with $\mu(\varphi(x,b'))\geq \varepsilon$ for all $b' \equiv b$. By Fact \ref{fact: Kelley}(1) this implies that $i(\{\varphi(x,b') : b' \equiv b \}) \geq \varepsilon$ --- hence $\DFHP_k$ holds.
\end{rem}

\section{FHP in some generic expansions of $\left(\mathbb{Z},+\right)$}\label{sec: FHP in some expansions of Z}

\subsection{$\left(\mathbb{Z},+,\protect\Sqf\right)$}\label{sec: square free is FHP}

In this section we consider the structure $\left(\mathbb{Z},+,\Sqf\right)$,
where $\Sqf$ is a unary predicate for the integers not divisible by a square
of any prime number, and show that $T_{\Sqf}$ satisfies $\FHP$.
This theory is studied in \cite{tran}, and we recall some notation
and results from there.

\begin{defn}\cite[Section 2]{tran}\label{def: basic square free notation}
Below $x$ be a single variable and $z = (z_i : i \in [n]), z' = (z'_i : i \in [n'])$ denote arbitrary tuples of variables, and for $a \in \mathbb{Z}$ and $p$ prime, let $v_p(n)$ denote the $p$-adic valuation.
\begin{enumerate}
\item For $m\in\mathbb{N}_{>0}$ let $P_{m}^{\mathbb{Z}}:=\left\{ a\in\mathbb{Z}:v_{p}\left(a\right)<2+v_{p}\left(m\right)\text{ for all }p\right\} $,
so $P^{\mathbb{Z}}_{1}=\Sqf$. And for $p$ prime and $l\in\mathbb{Z}$,
let $U^{\mathbb{Z}}_{p.l}:=\left\{ a\in\mathbb{Z}:v_{p}\left(a\right)\geq l\right\} =p^{l}\mathbb{Z}$.
\item We consider the expansion $M$ of  $\left(\mathbb{Z},+,\Sqf\right)$ in the language  
$$\mathcal{L}=\left\{ +,-,1,\left(P_{m}\right)_{m\in\mathbb{N}_{>0}},\left(U_{p,l}\right)_{p\text{ prime, }l\in\mathbb{Z}}\right\} ,$$
and let $T_{\Sqf}:=\Th_{\mathcal{L}}\left( M \right)$. Let $\M \succ M$ be a monster model.
\item An $L$-formula $\theta\left(z\right)$ is a \emph{$p$-condition, }for
$p$ prime, if it is a Boolean combination of formulas of the form
$t\left(z\right)\notin U_{p,l}$ for $t\left(z\right)$ an arbitrary
$L$-term (i.e. a $\mathbb{Z}$-linear combination of the variables from $z$) and $l\in\mathbb{N}$ (a $p$-condition is trivial if the Boolean combination is the empty conjunction). An $L$-formula $\theta(z)$ which is a Boolean combination of formulas of the form $t(z) = 0$ with $t\left(z\right)$ an arbitrary
$L$-term is called an \emph{equational condition}.
\item An $L$-formula $\psi\left(x,z,z'\right)$ is a \emph{special formula} if it is of
the form 
\[
\bigwedge_{p \in S}\theta_{p}\left(x,z,z'\right)\land\bigwedge_{i=1}^{n}\left(kx+z_{i}\in P_{m}\right)\land\bigwedge_{i'=1}^{n'}\left(kx+z_{i}'\notin P_{m}\right)\text{,}
\]
where $k \in \mathbb{Z} \setminus \{0\}$, $m\in\mathbb{N}_{\geq 1}$ and $\theta_{p}(x;z,z')$ are $p$-conditions for some finite set of primes $S$. And $\psi\left(x,z,z'\right)$ is a \emph{positive special formula}  if it is of the form $\bigwedge_{p \in S}\theta_{p}\left(x,z,z'\right)\land\bigwedge_{i=1}^{n}\left(kx+z_{i}\in P_{m}\right)$.
\item Given a (positive) special formula $\psi\left(x,z,z'\right)$ as in (4) and prime $p$, its \emph{associated
$p$-condition} $\psi_{p}\left(x,z,z'\right)$ is the formula 
$$\theta_{p}\left(x,z,z'\right)\land\bigwedge_{i=1}^{n}\left(kx+z_{i}\notin U_{p,2+v_{p}\left(m\right)}\right)$$
 (where we let $\theta_p$ be the trivial $p$-condition for $p \notin S$). 
Then $\psi\left(x,z,z'\right)$ implies $\psi_{p}\left(x,z,z'\right)$ for all $p$  (in $T_{\Sqf}$).
\item A (positive) \emph{$G$-system} is a (positive) formula $\psi\left(x,c,c'\right)$ with $\psi\left(x,z,z'\right)$ a (positive) special formula and $c \in \M^z, c' \in \M^{z'}$. It is \emph{non-trivial} if, assuming $c = (c_i : i \in [n]), c' = (c'_j : j \in [n'])$, we have $c_i \neq c'_j$ for all $(i,j) \in [n]\times [n']$.
\item For a prime $p$, a (positive) $G$-system $\psi\left(x,c,c'\right)$ is \emph{$p$-satisfiable
}if $\psi_{p}\left(x,c,c'\right)$ is satisfiable for each prime $p$.
By (5), if $\psi\left(x,c,c'\right)$ is satisfiable, then it is
$p$-satisfiable for all $p$.
\end{enumerate}
\end{defn}

\begin{fact}
\cite[Theorem 2.14]{tran}\label{fact: G-system local to global consistency}
If a  non-trivial $G$-system $\psi\left(x,c,c'\right)$ is $p$-satisfiable for
all prime $p$, then it has infinitely many solutions in $\M$.
\end{fact}

\begin{fact}
\label{fact: QE for Sqf}\cite[Theorem 1.1 and Theorem 3.4]{tran} The $L$-theory   $T_{\Sqf}$ eliminates
quantifiers and is supersimple, of SU-rank $1$.
\end{fact}

\begin{fact}
\label{fac: QF reduces to G-systems} \cite[Lemma 3.1]{tran} Let $\varphi\left(x,y\right)$ be
a quantifier-free $L$-formula, with $x$ a singleton and $y$ a tuple of variables. The $\varphi(x,y)$ is equivalent, modulo $T_{\Sqf}$, to a disjunction of quantifier-free $L$-formulas of the form 
\begin{gather*}
	\rho(y) \land \varepsilon(x,y) \land \psi(x, t(y), t'(y)),
\end{gather*}
where $\rho(y)$ is a quantifier-free $L$-formula, $\varepsilon(x,y)$ an equational condition, $\psi(x,z,z')$ a special $L$-formula with $n := |z|, n' := |z'|$, and $t(y),t'(y)$ are tuples of $L$-terms of length $n$ and $n'$, respectively.
\end{fact}

We will need a refinement of \cite[Lemma 2.11]{tran} in the special
case of positive special formulas:
\begin{lem}
\label{lem: density for positive special formulas}Let $\psi\left(x,z,z'\right)$
be a positive special formula. Then there is some $\varepsilon=\varepsilon\left(\psi\right) \in \mathbb{R}_{>0}$
satisfying the following.

For any tuples $c,c'$  in $\mathbb{Z}$ such that $\psi\left(x,c,c'\right)$ is a positive $G$-system
which is $p$-satisfiable for all $p$, let 
$$\Psi\left(t\right):=\left\{ a\in\mathbb{Z}:\psi\left(a,c,c'\right)\text{ holds and }0<a<t\right\}.$$
Then 
\[
\left|\Psi\left(t\right)\right|\geq\varepsilon t-\sum_{i=1}^{n}\left(\sqrt{\left|c_{i}\right|}+\sqrt{\left|kt+c_{i}\right|}\right) - 1 
\]
for all $t\in\mathbb{N}$ (where $n = |z|$ and $k \in \mathbb{Z} \setminus \{0\}$ is as in Definition \ref{def: basic square free notation} for $\psi$).
\end{lem}

\begin{rem}
So the refinement is that the choice of $\varepsilon$ depends only
on the formula $\psi(x,z,z')$, and not on the parameters $c,c'$. We remark
that no such choice is possible for general, not necessarily positive,
special formulas --- as the following example demonstrates\footnote{We thank  Chieu Minh Tran for pointing this out.}.
Take $\psi(x,z,z') := x+z\in P_{1}\land x+z'\notin P_{1}$. If an element
$a$ satisfies $\psi\left(x,c,c'\right)$, then in particular $a+c'$
must be divisible by $p^{2}$ for some prime $p$. Suppose $c'-c$
is a multiple of all squares of primes $<N$. Then such a $p$ as
above must be $>N$. But then the density of $\left\{ a \in \mathbb{Z} :\models\psi\left(a,c,c'\right)\right\} $
is at most $\sum_{p>N}\frac{1}{p^{2}}$, which can be arbitrarily
close to $0$ when $N$ is large.
\end{rem}

\begin{proof}
We follow closely the proof of \cite[Lemma 2.11]{tran} with $h=1$, $s=0$. Let a positive special formula 
$$\psi\left(x,z,z'\right) = \bigwedge_{p\leq B_{0}}\theta_{p}\left(x,z,z'\right)\land\bigwedge_{i=1}^{n}\left(kx+z_{i}\in P_{m}\right)$$
be fixed, and let $B:=\max\left(\left\{ \left|k\right|,n,B_{0}\right\} \right)+1$.
Let $l_{p}'$ be the largest value of $l$ such that the predicate
$U_{p,l}$ occurs in $\theta_{p}$. Let $D := \prod_{p\leq B}p^{l_{p}'}$.
Note that both $B$ and $D$ depend only on $\psi$. 

By the proof of \cite[Lemma 2.10]{tran} for any $c,c'$ such that
$\theta_{p}\left(\mathbb{Z},c,c'\right)$ is non-empty for all $p\leq B$,
there is some $r\in\left\{ 0,1,\ldots,D-1\right\} $ such that $a\equiv_{D}r$
implies $\models\bigwedge_{p\leq B}\theta_{p}\left(a,c,c'\right)$
for all $a\in\mathbb{Z}$.

For a prime $p$, we let $l_{p}:=2+v_{p}\left(m\right)$, and fix $M>B$ to be determined
later. 

Given some $c,c'$ as in the statement of the lemma, let $r$ be as
given by the previous paragraph. Define
\[
\Psi_{M}\left(t\right):=\left\{ a\in\mathbb{Z}:\left(0<a<t\right)\,\land\,\left(a\equiv_{D}r\right)\,\land\,\left(\bigwedge_{B<p\leq M}\bigwedge_{i=1}^{n}ka+c_{i}\not\equiv_{p^{l_{p}}}0\right)\right\} \text{.}
\]

First we establish a lower bound on $\left|\Psi_{M}\left(t\right)\right|$.
For a prime $p>B$, we have that $p>\left|k\right|$, so $k$ is invertible
$\md p^{l_{p}}$. Hence  there are at least $p^{l_{p}}-n$ choices
of $r_{p}\in\left\{ 1,\ldots,p^{l_{p}}\right\} $ such that: for all $a \in \mathbb{Z}$, if $a\equiv_{p^{l_{p}}}r_{p}$ 
then $\models\bigwedge_{i=1}^{n}\left(ka+c_{i}\not\equiv_{p^{l_{p}}}0\right)$.
As $\left\{ p^{l_{p}}:B<p\leq M\right\} \cup\left\{ D\right\} $ are
pairwise coprime, by the Chinese remainder theorem we have 
\[
\left|\Psi_{M}\left(t\right)\right|\geq\prod_{B<p\leq M}\left(p^{l_{p}}-n\right)\left\lfloor \frac{t}{D\prod_{B<p\leq M}p^{l_{p}}}\right\rfloor. 
\]
Namely, first we choose some $\left(r_{p}:B<p\leq M\right)$ as above, and
independently for each such fixed choice count the elements in 
$$\left\{ a\in\mathbb{Z}: 0 < a < t \land a\equiv_{D}r \land  \left(\bigwedge_{B<p\leq M}a\equiv_{p^{l_{p}}}r_{p}\right) \right\}.$$
This implies 
\[
\left|\Psi_{M}\left(t\right)\right|\geq\frac{t}{D}\prod_{B<p\leq M}\frac{p^{l_{p}}-n}{p^{l_{p}}}\geq\frac{t}{D}\prod_{B<p\leq M}\left(1-\frac{n}{p^{l_{p}}}\right)\geq\frac{t}{D}\prod_{B<p}\left(1-\frac{n}{p^{l_{p}}}\right)\text{.}
\]
Set $\varepsilon:=\frac{1}{2D}\prod_{B<p}\left(1-\frac{n}{p^{l_{p}}}\right) > \prod_{B < p} \left( 1 - \frac{1}{p^{3/2}} \right) >0$ (using Euler's product formula, as $l_p \geq 2$ for all $p$ by definition).
Then $\left|\Psi_{M}\left(t\right)\right|\geq2\varepsilon t$. 

We may assume $k>0$ (replacing $c$ by $-c$ and $c'$ by $-c'$ if necessary). If $a\in\Psi_{M}\left(t\right)\setminus\Psi\left(t\right)$,
then $0<a<t$, hence $c_{i}<ka+c_{i}<kt+c_{i}$ for all $i \in [n]$; and we must have $ka+c_{i} \notin P_m$ for some $i \in [n]$ (in particular $n \geq 1$), hence $ka+c_{i}$  is
a multiple of $p^{l_{p}}$ for some $p$ by definition of $P_m$, so necessarily for some $p>M$ by definition of $\Psi_M(t)$.

For each $p$ and $i \in [n]$, the number of non-zero multiples of $p^{l_p}$ in the interval $(c_i, kt+c_i)$ is $\lfloor ktp^{-l_p} \rfloor + \gamma$ for some $\gamma \in \{0,1,-1,-2\}$. If $\gamma = 1$, as $l_p \geq 2$ we moreover have $p^2 \leq |c_i|$ or $p^2 \leq |kt + c_i|$, so $p \leq \sqrt{|c_i|} + \sqrt{|kt + c_i|}$. As $l_p \geq 2$, we have $\lfloor ktp^{-l_p} \rfloor \leq k t p^{-2}$. Hence  we get 
\[
\left|\Psi_{M}\left(t\right)\setminus\Psi\left(t\right)\right|\leq t\sum_{p>M}\frac{nk}{p^{2}}+\sum_{i=1}^{n}\left(\sqrt{\left|c_{i}\right|}+\sqrt{\left|kt+c_{i}\right|}\right) + 1\text{.}
\]

Note that taking $M$ sufficiently large, depending only on $D,B,n$, and
hence only on $\psi$, we get $\sum_{p>M}\frac{nk}{p^{2}}<\varepsilon$.
Combining this with the lower bound estimate for $\left|\Psi_{M}\left(t\right)\right|$,
we get the desired conclusion.
\end{proof}

This gives a lower bound on the asymptotic density of the set of solutions of $\psi(x, c, c')$ which is independent of $c$ and $c'$. Using this we have:
\begin{cor}
\label{cor: basic Sqf formulas have FHP}Let $\psi\left(x,z,z'\right)$
be a positive special formula. Then there is some $\delta=\delta\left(\psi\right)>0$
such that: for any $n \in \mathbb{N}$ and any tuples $\left(c_{i},c_{i}': i \in [n] \right)$ with
$\psi\left(x,c_{i},c_{i}'\right)$ consistent for all $i$, there
is some $I\subseteq\left[n\right],\left|I\right|\geq\delta n$ such
that $\left\{ \psi\left(x,c_{i},c_{i}'\right):i\in I\right\} $ is
consistent.
\end{cor}

\begin{proof}
Let $\delta:=\frac{\varepsilon}{2}>0$ with $\varepsilon$ given by
Lemma \ref{lem: density for positive special formulas} for $\psi$.
Let $\mu$ be a finitely additive probability measure on $\mathcal{P}\left(\mathbb{Z}\right)$
defined, for all $X \subseteq \mathbb{Z}$, via $\mu\left(X\right):=\lim_{n \to \mathcal{U}}\frac{\left|X\cap\left[-n,n\right]\right|}{2n+1}$
for some fixed non-principal ultrafilter $\mathcal{U}$ on $\mathbb{N}$. Let $n$ and $\left(c_{i},c_{i}': i \in [n] \right)$
be as in the statement. By Lemma \ref{lem: density for positive special formulas},
we have $\mu\left(\psi\left(\mathbb{Z};c_{i},c_{i}'\right)\right)\geq\delta$
for each $i \in [n]$. Then, by Fact \ref{fact: Kelley}(1), there exists $I \subseteq [n]$ with $|I| \geq \delta n$ so that $\bigcap_{i \in I}  \psi\left(\mathbb{Z};c_{i},c_{i}'\right) \neq \emptyset$.
\end{proof}

\begin{thm}
\label{thm: Sqf is FHP}The theory $T_{\Sqf}$ is $\FHP$ (in particular $\left(\mathbb{Z},+,\Sqf\right)$ is $\FHP$, as a reduct). Moreover, every formula $\varphi(x,y)$ with $|x| \leq d$ satisfies $\FHP_{d+1}$.
\end{thm}

\begin{proof}
To show $\FHP$, by Lemma \ref{lem: Basic operations preserving FH} it suffices to show that every partitioned formula $\varphi(x,y) \in L$ with $|x|=1$ has $\FHP$. 
By quantifier elimination  (Fact \ref{fact: QE for Sqf}) we may assume $\varphi(x,y)$ is quantifier-free. Then, using that
the set of formulas satisfying $\FHP$ is closed under disjunctions (Lemma \ref{lem: Basic operations preserving FH}(2)), by Fact \ref{fac: QF reduces to G-systems} we may assume that $\varphi(x,y)$ is of the form 
\begin{gather*}
	\rho(y) \land \varepsilon(x,y) \land \psi(x, t(y), t'(y)),
\end{gather*}
where $\rho(y)$ is a quantifier-free $L$-formula, $\varepsilon(x,y)$ an equational condition, $\psi(x,z,z')$ a special $L$-formula with $s := |z|, s' := |z'|$, i.e.
$$\psi\left(x,z,z'\right)=\bigwedge_{p \in S}\theta_{p}\left(x,z,z'\right)\land\bigwedge_{i=1}^{s}\left(kx+z_{i}\in P_{m}\right)\land\bigwedge_{i'=1}^{s'}\left(kx+z_{i}'\notin P_{m}\right)$$
where $S$ is a finite set of primes, and $t(y),t'(y)$ are tuples of $L$-terms of length $s$ and $s'$, respectively.

Then, by Remark \ref{rem: FHP pres by func/1-var conj}, it is  sufficient to show that every formula of the form 
$$ \varphi(x; \bar{z}) = \varepsilon(x,z'') \land \psi(x, z, z') $$
with $\bar{z} = (z, z', z'')$ is $\FHP$. Again using that $\FHP$ is closed under disjunctions, we may assume that $\varepsilon(x,z'') = \bigwedge_{i=1}^{v}  \left( t_i(x,z'') \square_i 0 \right)$ with $v \in \mathbb{N}$, $\square_i \in \{=, \neq\}$ and $t_i$ an $L$-term, i.e.~a fixed $\mathbb{Z}$-linear combination of the variables from $x,z''$.

By Fact \ref{fact: QE for Sqf}, Fact \ref{fact: wnfcp} and Remark \ref{rem: around wnfcp}, $T_{\Sqf}$ eliminates $\exists^{\infty}$. That is, there is some $d = d( \varphi(x, \bar{z})) \in\mathbb{N}$  such that for any
$\bar{c} \in M^{\bar{z}}$, $\varphi\left(M,\bar{c}\right)$ is finite if and only if
$\left|\varphi\left(M,\bar{c}\right)\right|\leq d$. As the set of formulas
with $\FHP$ is closed under disjunctions, and $\varphi\left(x,\bar{z}\right)=\left(\varphi\left(x,\bar{z}\right)\land\left|\varphi\left(x,\bar{z}\right)\right|\leq d\right)\lor\left(\varphi\left(x,\bar{z}\right)\land\left|\varphi\left(x,\bar{z}\right)\right|=\infty\right)$,
we may treat each of these two disjuncts separately. Now the formula 
$\varphi\left(x,\bar{z}\right)\land\left|\varphi\left(x,\bar{z}\right)\right|\leq d$
satisfies $\FHP_{2}$ by Lemma \ref{lem: FHP for finite sets}, and we show
that $\varphi\left(x,\bar{z}\right)\land\left|\varphi\left(x,\bar{z}\right)\right|=\infty$
satisfies $\FHP_{1}$.

So fix $\alpha>0$, and assume that we are given some $\left(\bar{c}_{i}:i \in [n] \right)$,
$\bar{c}_{i}=c_{i}c_{i}'c_{i}''$ such that $\varphi\left(x,\bar{c}_{i}\right)\land\left|\varphi\left(x,\bar{c}_{i}\right)\right|=\infty$
holds for all $i\in I_{0}$ for some $I_{0}\subseteq\left[n\right]$
with $\left|I_{0}\right|\geq\alpha n$. In particular, $\varphi\left(M,\bar{c}_{i}\right)$
is infinite for all $i\in I_{0}$.

For each $i \in [n]$, let $A_i$ be the set of all elements
appearing in the tuple $c_{i}$ and $B_i$ the set of all elements
appearing in the tuple $c_{i}'$. Let $C_i := A_i \cup B_i$. Note that for $i \in I_0$,  an element of $C_{i}$ cannot simultaneously occur in both tuples $c_{i}$
and $c_{i}'$ as $\psi\left(x,c_{i},c_{i}'\right)$ is consistent, so $C_i = A_i \sqcup B_i$. As $|A_i| \leq s$ and $|B_i| \leq s'$ for all $i \in I_0$,  there is some $I_{1}\subseteq I_{0}$ with $\left|I_{1}\right|\geq\frac{1}{s s'}\left|I_{0}\right|\geq\frac{\alpha}{s s'}n$ and $r \leq s, r' \leq s'$ so that $|A_i| = r, |B_i|=r', |C_i| = k := r + r'$ for all $i \in I_1$.

%
%

%
%

Applying Lemma \ref{lem: Furedi repetitions} to the sequence 
$\left(C_{i}:i\in I_{1}\right)$ of $k$-element subsets of $M$, there is some $\gamma=\gamma\left( k \right) = \gamma(s,s')>0$,
$I_{2}\subseteq I_{1}$ with $\left|I_{2}\right|\geq\gamma\left|I_{1}\right|\geq\frac{\alpha\gamma}{ss'}n$
and some pairwise disjoint subsets  $X_{1},\ldots,X_{k}$ of $M$ such
that $\left|C_{i}\cap X_{j}\right|=1$ for all $i\in I_{2}, j \in [k]$. In particular, there is some $J_i \subseteq [k]$ with $|J_i| = r_1$ so that $A_i \subseteq \bigcup_{j \in J_i} X_j$ and $B_i \subseteq \bigcup_{j \in [k] \setminus J_i} X_j$. By pigeonhole we can fix $J_i$, and possibly permuting the index sets $[k]$, 
we find some $I_{3}\subseteq I_{2}$ with $\left|I_{3}\right|\geq\frac{1}{{k \choose r_1}}\left|I_{2}\right|\geq\frac{\alpha\gamma}{ss'(s+s')^{s}}n$
such that for each $i\in I_{3}$, we have $A_i \subseteq X_1 \cup \ldots \cup X_{r_1}$ and $B_i \subseteq X_{r_1 + 1} \cup \ldots \cup X_{k}$. In particular, as $X_1, \ldots, X_k$ are pairwise disjoint, for all for all $i,j \in I_3$ we have $A_i \cap B_j = \emptyset$, so the tuples  $c_i$ and $c'_j$ do not have any  common elements.

Now consider the positive special formula 
\[
\psi'\left(x,z,z'\right) := \bigwedge_{p \in S }\theta_{p}\left(x,z,z'\right)\land\bigwedge_{i=1}^{s}\left(kx+z_{i}\in P_{m}\right)\text{.}
\]
As $\psi'\left(x,c_{i},c_{i}'\right)$ is consistent for all $i\in I_{3}$,
by Corollary \ref{cor: basic Sqf formulas have FHP} there is some
$\delta=\delta\left(\psi'\right) = \delta( \psi)>0$ and some $J\subseteq I_{3},\left|J\right|\geq\delta\left|I_{3}\right|\geq \frac{\alpha\gamma \delta}{ss'(s+s')^{s}} n$
such that the set $\left\{ \psi'\left(x,c_{i},c_{i}'\right):i\in J\right\} $
is consistent.

Now consider the $G$-system $\widetilde{\psi}(x; \widetilde{c}, \widetilde{c}') := \bigwedge_{i\in J}\psi\left(x,c_{i},c'_{i}\right)$, where $\widetilde{c} = (c_i : i \in J)$, $\widetilde{c}' = (c'_i : i \in J)$ (and $c_i = (c_{i,j} : j \in [s])$, $c'_i = (c'_{i,j} : j \in [s'])$). This $G$-system is non-trivial by the choice of $I_{3}$. For any prime $p$, its associated $p$-condition $\widetilde{\psi}_p$
is 
\[
\bigwedge_{i\in J}\left(\theta_{p}\left(x,c_{i},c_{i}'\right)\land\bigwedge_{j=1}^{s}\left(kx+c_{i,j}\notin U_{p,2+v_{p}\left(m\right)}\right)\right),
\]
and it is consistent as $\left\{ \psi'\left(x,c_{i},c_{i}'\right):i\in J\right\} $
is consistent. Hence the whole $G$-system $\bigwedge_{i\in J}\psi\left(x,c_{i},c'_{i}\right)$
has infinitely many solutions by Fact \ref{fact: G-system local to global consistency}. 

As $\varphi\left(x,\bar{c}_{j}\right)$ is infinite for all $j\in J$ by assumption,
$\varepsilon \left(x,c''_{j}\right)$ is also infinite. As each formula $t_i(x,c''_j) = 0$ has at most one realization, we must have  $\square_i =$``$\neq$'' for all $i \in [v]$.

As $\bigwedge_{i\in J}\psi\left(x,c_{i},c'_{i}\right)$
is infinite, it contains infinitely many realizations outside of the finite set $\bigcup_{i \in [v]}\bigcup_{j\in J}\left\{ a\in\mathbb{Z}:t_i\left(a,c_{j}''\right)=0\right\} $.
Hence $\left\{ \varphi\left(x,\bar{c}_{j}\right):j\in J\right\} $ has infinitely many realizations. Thus we have shown that the formula $\varphi\left(x,\bar{z}\right)\land\left|\varphi\left(x,\bar{z}\right)\right|=\infty$
satisfies $\FHP(1, \alpha,\beta)$ with $\beta:=\frac{\alpha\gamma \delta}{ss'(s+s')^{s}}>0$ chosen depending only on $\varphi$ and $\alpha$.

For the moreover part, we have that $T_{\Sqf}$ is simple, of SU rank $1$ by Fact  \ref{fact: QE for Sqf}. Hence, using Fact \ref{fac: submult of burden}(3), $\bdn(\M^1) = 1$ and $\bdn(\M^d) = d$ for all $d \in \mathbb{N}$, and we conclude by Corollary \ref{cor: FHPk bounded by burden}.
\end{proof}

\begin{problem}
We expect that using the results in \cite{tran} this proof generalizes
to show that $\left(\mathbb{Q},+,\Sqf,<\right)$ is FHP.
\end{problem}

\begin{problem}
Is $T_{\Sqf}$ MS-measurable? (See Definition \ref{def: MS meas structure}.)
\end{problem}

\subsection{$\left(\mathbb{Z},+,\Pr\right)$}\label{subsec: T_pr}

Let $T_{\Pr}:=\Th_{\mathcal{L}}\left(\mathbb{Z},+,\Pr\right)$ in the language 
$$\mathcal{L}=\left(+,-,0,1,\Pr,\Pr_{n},P_{n}:2\leq n<\omega\right),$$
where $\Pr$ is a predicate for the primes and their additive inverses, $P_{n}\left(x\right)\iff x\equiv0\left(\md n\right)$,
$\Pr_{n}\left(x\right)\iff P_{n}\left(x\right)\land\Pr\left(\frac{x}{n}\right)$. 

\begin{prop}\label{prop: Primes not FHP}
$T_{\Pr}$ is not $\FHP$.
\end{prop}

\begin{proof}
The $\emptyset$-definable set $\Pr$ does not $G$-divide in the group $G := (\mathbb{Z},+)$. This is a direct consequence of the following theorem of Maynard:
\begin{fact}\cite[Theorem 1.2]{maynard2015small}
	For every $k \in \mathbb{N}$ there exist $N = N(k) \in \mathbb{N}$ and $C = C(k) >0$ satisfying the following. If $A = \{a_1, \ldots, a_N\} \subseteq \mathbb{Z}$ is any set of $N$ distinct integers, then 
	\begin{gather*}
		\left \lvert \left\{ \{h_1, \ldots, h_k\} \subseteq A : \exists^{\infty} n \in \mathbb{Z} \bigwedge_{i=1}^k (n+h_i \textrm{ is prime}) \right\} \right \rvert \geq C \left \lvert \left\{ \{h_1, \ldots, h_k\} \subseteq A \right\} \right \rvert
	\end{gather*}
\end{fact}

\noindent This implies that for any $k \in \mathbb{N}$, there is no (indiscernible) sequence $(g_i)_{i < \omega}$ with $g_i \in G(\M)$ so that $\{g_i + \Pr(\M) : i \in \omega\}$ is $k$-inconsistent.

Assume that $T_{\Pr}$ was FHP. The group $G:=\left(\mathbb{Z},+\right)$
is amenable. It is well-known that the set of primes has
upper Banach density $0$, which implies that $\mu\left(\Pr\right)=0$
for any $G$-invariant finitely additive probability measure $\mu$ on $\mathcal{P}\left(G\right)$
(e.g. see \cite[Section 2.2]{MR3361013}). This contradicts Theorem
\ref{thm: generics in FHP}.
\end{proof}

\begin{rem}\label{rem: square free positive density gen}
	We note the crucial difference with the case of square-free integers. Let $M =\left(\mathbb{Z},+,\Sqf\right)$ be as in Section \ref{sec: square free is FHP}, let $G(M) := (\mathbb{Z},+)$.  By Fact \ref{fact: QE for Sqf} its theory $\Th(M)$ is supersimple, of SU rank $1$, and by Theorem \ref{thm: Sqf is FHP} it is $\FHP$. As classically $\Sqf$ has upper Banach density $\frac{\pi}{6}$, we have $\mu(\Sqf(M)) \geq \frac{\pi}{6}$ for some  left-$G(M)$-invariant finitely additive probability measure on $\mathcal{P}(\mathbb{Z})$, hence the definable set $\Sqf$ is non-null in $G = G(\M)$ (see Section \ref{sec: notions of genericity}).

	But $\Sqf$ is
not weakly generic. We use the standard argument that it is not piecewise syndetic. For  $(p_n)_{n \in \omega}$ an increasing enumeration of the primes, let $r_n := p_0^2 \cdot \ldots \cdot p_n^2$. For any $A = \{a_0, \ldots, a_n\} \subset \mathbb{Z}$, by the Chinese Remainder Theorem every interval in $\mathbb{Z}$ of length $r_n$ contains a solution to the system $\{x \equiv a_i (\mod p_i^2) : 0 \leq i \leq n \}$, so $G(M) \setminus (A \cdot \Sqf(M))$ is generic. As $M \prec \M$, it follows that for every finite $A \subseteq G(\M)$, the set $G(\M) \setminus (A \cdot \Sqf(\M))$ is generic in $G(\M)$ --- hence $\Sqf(\M)$ is not weakly generic.
\end{rem}

The following is a very strong number theoretic conjecture of Dickson:
\begin{conjecture}(Dickson, 1904 \cite{dickson1904new})\label{conj: Dickson}
	Let $k \in \mathbb{N}_{\geq 1}$ and $\bar{f} = (f_i : i < k)$ where $f_i(x) = a_i x + b_i$ with $a_i, b_i$ non-negative integers and $a_i \geq 1$ for all $i<k$. Assume that there does not exist any integer $n \geq 1$ simultaneously dividing all of the products $\prod_{i < k} f_i(s)$ for every non-negative integer $s$. Then there exist infinitely many natural numbers $m$ such that $f_i(m)$ is prime for all $i<k$.
\end{conjecture}

\begin{fact}
\cite{kaplan2017decidability}\label{fact: Primes} Assuming Dickson's
conjecture, $T_{\Pr}$ eliminates quantifiers and is supersimple,
of $\SU$-rank $1$.
\end{fact}

In the rest of the section, we will prove the following theorem: 
\begin{thm}
\label{prop: Tpr has local FHP} (Assuming Dickson's conjecture) $T_{\Pr}$
is locally $\FHP$. 

In fact, we only require the types of singletons in the parameter tuples to be fixed along the sequence (as opposed to the types of whole parameter tuples).  Namely, for every partitioned formula $\varphi(x,y) \in L$ there exists $d = d(\varphi)$ satisfying the following. Let $s := |y|$. For every $\alpha>0$ there exists $\beta = \beta(\varphi,\alpha) >0$ so that: for any $n \in \mathbb{N}$ and tuples $c_i = (c_{i,j} : j \in [s] ), i \in [n]$ so that for each fixed $j \in [s]$, $\tp(c_{i,j}) = \tp(c_{i',j})$ for all $i,i' \in [n]$, if $|\{ I  \in {n \choose d} : \bigcap_{i \in I} \varphi(\mathbb{Z}, c_i) \neq \emptyset\}| \geq \alpha {n \choose d}$, then there exists $J \subseteq [n]$ with $|J| \geq \beta n$ so that $\bigcap_{i \in J} \varphi(\mathbb{Z}, c_i) \neq \emptyset$.
\end{thm}

Given a partitioned formula $\varphi(x,y) \in L$, let $\beta^{\ast}(\varphi) > 0$ be a real number satisfying the following (assuming that it exists):
\begin{itemize}
	\item[$(\ast)$] If $c_i = (c_{i,j} : j \in [s] ), i \in [n]$ are tuples from $\mathbb{Z}$  so that for each fixed $j \in [s]$, $\tp(c_{i,j}) = \tp(c_{i',j})$ for all $i,i' \in [n]$ and $\varphi\left(\mathbb{Z},c_{i}\right)$ is infinite for
all $i \in [n]$, then there is some $J\subseteq\left[n\right],\left|J\right|\geq\beta^{*}n$
such that $\bigcap_{i\in J}\varphi\left(\mathbb{Z},c_{i}\right)$ is infinite.
\end{itemize}

In the following claims, we establish this property for increasing classes of formulas with $|x| = 1$.

\begin{claim}\label{cla: primes loc FHP 1}
	Let $\varphi\left(x,y\right)\in\mathcal{L}$, where $|x|=1$ and $y = (y_j)_{j \in [s]}$, be of the form 
$$\bigwedge_{j=1}^{s}\Pr\left(m_{j}x+y_{j}\right)^{t_{j}}$$ 
for some $m_j \in \mathbb{Z}$ and truth values $t_{j}\in\left\{ 0,1\right\} $. Then there exists $\beta^{*}\left(\varphi\right) = \beta^{\ast}(s) >0$ satisfying $(\ast)$.
\end{claim}
\begin{proof}

Let $n \in \mathbb{N}$ and tuples $\left(c_{i}: i \in [n] \right)$ as in the assumption of $(\ast)$ be given, with $c_i = (c_{i,j} :  j \in [s] )$. 

We may assume that $m_j \geq 1$ for all $j \in [s]$. Indeed, if $m_j = 0$ then $\Pr(c_{i,j})^{t_j}$ holds for all $i \in [n]$ by assumption, so we can ignore such conjuncts. Note that $a \models \Pr(m_j x + c_{i,j})^{t_j} \Leftrightarrow  a \models \Pr((-m_j) x + (-c_{i,j}))^{t_j}$ for any $a, m_j, c_{i,j} \in \mathbb{Z}$. Hence, letting $J := \{j \in [s] : m_j < 0\}$, we could consider the formula $\varphi'(x,y)$ obtained from $\varphi(x,y)$ by  replacing $m_j$ by $-m_j$ for all $j \in J$,  and the sequence $(c'_i : i \in [n])$ with $c'_i$ obtained from $c_i$ by replacing  $c_{i,j}$ by $-c_{i,j}$ for all $j \in J$ (we still have that for every $j \in [n]$, $\tp(c'_{i,j}) = \tp (c'_{i',j})$ for all $i,i' \in [n]$). So if the assumption holds for $\varphi, c_i$ then it also holds for $\varphi', c'_i$; and if $\beta^{\ast}$ satisfies the conclusion for $\varphi'$, then it also satisfies the conclusion for $\varphi$.

For
each $i\in\left[n\right]$, let $C_{i}$ be the set of pairs $\left\{ \left(m_{j},c_{i,j}\right): j \in [s]\right\} $. By pigeonhole, there exists some $I_0 \subseteq [n]$ with $|I_0| \geq \frac{1}{s} n$ so that  $\left|C_{i}\right|$ is constant for all $i \in I_0$, denote
it by $k$ (so $k \leq s$).

By Lemma \ref{lem: Furedi repetitions}, we find some $\gamma=\gamma\left(s\right)>0$,
some $I_{1}\subseteq I_0, \left|I_{1}\right|\geq\gamma n$
and some disjoint sets $X_{1},\ldots,X_{k}$ such that $\left|C_{i}\cap X_{j}\right|=1$
for all $i\in I_{1}$ and $1\leq j\leq k$. Note that if for some $i \in [n]$ we have $\left(m_{j},c_{i,j}\right)=\left(m_{j'},c_{i,j'}\right)$
for some $1\leq j\neq j'\leq s$, then necessarily $t_{j}=t_{j'}$
as $\varphi\left(x,c_{i}\right)$ is consistent. Set $I_{2}^{0}:=I_{1}$,
and by pigeonhole and induction on $1\leq l\leq k$ we can choose
$u_{l}\in\left\{ 0,1\right\} $ and $I_{2}^{l}$ such that 
\begin{itemize}
\item $I_{2}^{l}\subseteq I_{2}^{l-1}$,
\item $\left|I_{2}^{l}\right|\geq\frac{1}{2}\left|I_{2}^{l-1}\right|$,
\item for any $i\in I_{2}^{l}$, if $\left(m,c\right)\in C_{i}\cap X_{l}$
and $\left(m,c\right)=\left(m_{j},c_{i,j}\right)$, then $t_{j}=u_{l}$
--- i.e.~either all occurrences of $\Pr\left(mx+c\right)$ in $\varphi\left(x,c_{i}\right),i\in I_{2}^{l}$
are positive, or all are negative.
\end{itemize}
Let $I_{2}:=I_{2}^{k}$. Then $\left|I_{2}\right|\geq\frac{1}{2^{k}}\left|I_{1}\right|\geq\frac{\gamma}{2^{s}}n$
and, as the sets $X_{j}$'s are pairwise disjoint, we have that for any $i\neq i'\in I_{2}$
and $1\leq j\neq j'\leq k$, if $\Pr\left(m_{j}x+c_{i,j}\right)\in\varphi\left(x,c_{i}\right)$
and $\neg\Pr\left(m_{j'}x+c_{i',j'}\right)\in\varphi\left(x,c_{i'}\right)$,
then $\left(m_{j},c_{i,j}\right)\neq\left(m_{j'},c_{i',j'}\right)$.

Now Dickson's conjecture implies that the set $\bigcap_{i\in I_{2}}\varphi\left(\mathbb{Z},\bar{c}_{i}\right)$
is infinite, hence taking $\beta^{*}\left(\varphi\right):=\frac{\gamma}{2^{s}}>0$
shows that $\varphi(x,y)$ satisfies $(\ast)$. Namely, letting $J^+ := \{j \in [s] : t_j = 1\}$, by \cite[Lemma 2.3]{kaplan2017decidability} and the choice of $I_2$, as $m_j \geq 1$ for all $j \in [s]$ by assumption, it is sufficient to show that the assumption of Conjecture \ref{conj: Dickson} holds for $\bar{f} = (f_{i,j} : i \in I_2, j \in J^{+})$ where $f_{i,j}(x) = m_{j}x + c_{i,j}$. By \cite[Remark 2.1]{kaplan2017decidability} it suffices to show that for any fixed $N \in \mathbb{N}$, for every prime $r <N$, $r$ does not divide $\prod_{(i,j) \in I_2 \times J^{+}} f_{i,j}(s)$ for all $s \in \mathbb{Z}$ simultaneously. That is, for every prime $r < N$, for some $0 \leq t < r$ we have $\bigwedge_{i \in I_2, j \in J^{+}} m_j t + c_{i,j} \not \equiv 0  \ (\mod r) $. Assume this does not happen for $r$, and fix any $i^{\ast} \in I_2$.  Then, as $c_{i,j} \equiv c_{i^{\ast},j}  \ (\mod r)$ for all $i \in I_2, j \in J^{+}$ (since $\tp(c_{i,j}) = \tp(c_{i^{\ast},j})$ by assumption), we get that for all $0 \leq t < r$, for some $j \in J^{+}$, $m_j t + c_{i^{\ast},j} \equiv 0  \ (\mod r) $.

But this implies that $\varphi(\mathbb{Z},c_{i^{\ast}})$ is finite by \cite[Remark 2.6]{kaplan2017decidability}, contradicting the assumption on $c_{i^{\ast}}$.
\end{proof}

\begin{claim}\label{cla: primes loc FHP 2}
Assume $\varphi(x,y) \in L$, with $|x| = 1$ and $y = (y_j : j \in [s])$, is of the form 
$$\bigwedge_{j \in S}\Pr\left(m_{j}x+y_{j}\right)^{t_{j}}\land\bigwedge_{j \in S'}P_{k_{j}}\left(m_{j}'x+y_{j}\right)^{t'_{j}}.$$
for some $S,S' \subseteq [s]$, $m_j, m'_j \in \mathbb{Z}$, $k_j \in \mathbb{N}_{\geq 2}$ and truth values $t_{j}, t'_j \in\left\{ 0,1\right\} $.  Then there exists $\beta^{*}\left(\varphi\right) = \beta^{\ast}(s) >0$ satisfying $(\ast)$.
\end{claim}
\begin{proof}
	
Let $K=K\left(\varphi\right):=\prod\left\{ k_{j}: j \in S' \right\} $, and let tuples $\left(c_{i}: i \in [n] \right)$ as in the assumption of $(\ast)$ be given, with $c_i = (c_{i,j} :  j \in [s] )$.

By pigeonhole, there are some  $0 \leq r < K-1$ and $I\subseteq\left[n\right],\left|I\right|\geq\frac{1}{K}$
such that for each $i \in I$, $\varphi\left(x, c_{i}\right)\land \left( x\equiv r\left(\md K\right) \right)$ 
has infinitely many solutions in $\mathbb{Z}$. Note also that for
any $a\in\mathbb{Z}$, if $a\equiv r\left(\md K\right)$, then $a\models\bigwedge_{j \in S'}P_{k_{j}}\left(m_{j}'x+c_{i,j}\right)^{t'_{j}}$
for all $i\in I$.

Let $\varphi'\left(x',y'\right)$ be the formula $\bigwedge_{j \in S}\Pr\left(\tilde{m}_{j}x'+y_{j}'\right)^{t_j}$,
with $\tilde{m}_{j} := m_{j}K$, and let $c'_{i,j}:=m_{j}r+c_{i,j}$.
Then $\varphi'\left(x', c_{i}'\right)$ has infinitely many solutions
for each $i\in I$ (as any $a$ satisfying $\varphi\left(x, c_{i}\right)\land x\equiv r\left(\md K\right)$
gives $a' := \frac{a-r}{K}$ satisfying $\varphi'\left(x',c_{i}'\right)$, and $a_1 \neq a_2$ implies $a'_1 \neq a'_2$). And for any $i\in I$ and $a'\in\mathbb{Z}$ satisfying $\varphi'\left(x', c_{i}'\right)$,
$a:=Ka'+r$ satisfies $\varphi\left(x,c_{i}\right)$ by the above. Note that for each $j \in S$ and assumption on $c_i$,  all $\left(c_{i,j}': i \in I\right)$ all have the same type.  Hence
we can take $\beta^{*}\left(\varphi\right):=\frac{1}{K}\cdot\beta^{*}\left(\varphi'\right)$ satisfying $(\ast)$, where 
$\beta^{*}\left(\varphi'\right)$ exists by Claim \ref{cla: primes loc FHP 1}.
\end{proof}

\begin{claim}\label{cla: primes loc FHP 3}
Assume that $\varphi\left(x,y\right) \in L$ with $|x|=1$ and $y = (y_j : j \in [s])$ is an arbitrary finite conjunction
of formulas of the following form: $\Pr\left(mx+y_i\right)^{t}$, $\text{Pr}_{k}\left(mx+y_i\right)^{t}$,
$P_{k}\left(mx+y\right)^{t}$ where $m\in\mathbb{Z}$, $k \in \mathbb{N}_{\geq 2}$ and $t\in\left\{ 0,1\right\} $
is a truth value. Then there is some $\beta^{*}\left(\varphi\right)>0$
satisfying $(\ast)$.

\end{claim}
\begin{proof}
	 Assume first that $\varphi(x,y)$ contains
a conjunct $\Pr_{k}\left(mx+y_j\right)$.  Let tuples $\left(c_{i}: i \in [n] \right)$ as in the assumption of $(\ast)$ be given, with $c_i = (c_{i,j} :  j \in [s] )$.

For $r\in\left\{ 0,\ldots,k-1\right\} $, let $\varphi_{r}\left(x,y \right):=\varphi\left(x,y\right)\land x\equiv r\left(\md k\right)$.
By assumption and pigeonhole, there is some $r$ and some $I\subseteq\left[n\right],\left|I\right|\geq\frac{1}{k}n$
such that $\varphi_{r}\left(x, c_{i}\right)$ is infinite for all
$i\in I$.

We let $\varphi'_{r}\left(x';y,y'\right)$ be  obtained from $\varphi\left(x,y\right)$
by replacing $\Pr_{k}\left(mx+y_j\right)$ with $\Pr\left(mx'+y'\right)$
and replacing $x$ by $kx'+r$ everywhere else.

Assume that $a\models\varphi_{r}\left(x, c_{i}\right)$. Then $a\equiv r\left(\md k\right)$,
so $a=ka'+r$ for some $a'$. As $\models\Pr_{k}\left(ma+c_{i,j}\right)\iff\models\Pr_{k}\left(mka'+\left(mr+c_{i,j}\right)\right)$,
in particular $mr+c_{i,j}$ is divisible by $k$. Let $c'_{i}:=\frac{mr+c_{i,j}}{k}$.
Then 
$$\models \Pr_{k}\left(ma+c_{i,j}\right)\iff \models \Pr_{k}\left(kma'+kc'_{i}\right)\iff \models \Pr\left(ma'+c'_{i}\right).$$

Hence $\varphi'_{r}\left(x';c_{i},c_{i}'\right)$ has infinitely
many solutions for all $i\in I$ (as $\varphi_{r}\left(x, c_{i}\right)$
has infinitely many solutions by assumption and $a'$ above is uniquely determined
by $a$), and for any $a'\in\mathbb{Z}$ and $i\in I$, if $a'$ satisfies
$\varphi'_{r}\left(x'; c_{i},c_{i}'\right)$, then $a:=ka'+r$
satisfies $\varphi\left(x, c_{i}\right)$. Finally, note that still all elements in any fixed coordinate of $\left(c_{i}^{\frown}c_{i}':i\in I\right)$ have the same type over $\emptyset$. Thus we can take $\beta^{*}\left(\varphi\right):=\frac{1}{k}\cdot\min\left\{ \beta^{*}\left(\varphi'_{r}\right):r\in\left\{ 0,\ldots,k-1\right\} \right\} > 0 $.

Iterating this procedure for at most $s$ steps, we reduce to the case with no  conjuncts of the form $\Pr_{k}\left(mx+y_j\right)$. A similar analysis allows us to get rid of all conjuncts of the form $\neg\Pr_{k}\left(mx+y\right)$, thus reducing to the case in Claim \ref{cla: primes loc FHP 2}.
\end{proof}

\begin{claim}
	Every formula $\varphi(x,y) \in L$ with $|x|=1$ satisfies local $\FHP$, in the stronger form stated in Theorem \ref{prop: Tpr has local FHP}.
\end{claim}
\begin{proof}
Let $y = (y_i : i \in [s])$. 
The proof of Lemma \ref{lem: Basic operations preserving FH}(2) shows that the class of formulas satisfying local FHP in the strong form stated in Theorem \ref{prop: Tpr has local FHP} is closed under disjunctions. Using this and quantifier elimination in $T_{\Pr}$ (Fact \ref{fact: Primes}), we
may assume that $\varphi\left(x,y\right)$ is a  conjunction
of formulas of the form $\Pr\left(mx+f\left(y\right)\right)^{t}$,
$\text{Pr}_{k}\left(mx+f\left(y\right)\right)^{t}$, $P_{k}\left(mx+f\left(y\right)\right)^{t}$
and $\left(mx+f\left(y\right)=0\right)^{t}$ where $m\in\mathbb{Z}$, $k \in \mathbb{N}_{\geq 2}$, 
$t\in\left\{ 0,1\right\} $ is a truth value and $f$ is an $L$-term (i.e.~a $\mathbb{Z}$-linear combination of the $y_j, j \in [s]$).
Let $\varphi'\left(x;y,y'\right)$ be obtained from $\varphi\left(x,y\right)$
by replacing each occurrence of $f\left(y\right)$ above with
a new variable $y'$. As in Remark \ref{rem: FHP pres by func/1-var conj}, it is easy to see that local $\FHP$ in the strong form stated in Theorem \ref{prop: Tpr has local FHP} for $\varphi'\left(x,y'\right)$
implies the same for $\varphi\left(x,y\right)$. So we may
assume that $\varphi\left(x,y\right)$ is a finite conjunction
of formulas of the form $\Pr\left(mx+y\right)^{t}$, $\text{Pr}_{k}\left(mx+y\right)^{t}$,
$P_{k}\left(mx+y\right)^{t}$ and $\left(mx+y=0\right)^{t}$.

By Fact \ref{fact: Primes}, Fact \ref{fact: wnfcp} and Remark \ref{rem: around wnfcp}, $T_{\Pr}$ eliminates $\exists^{\infty}$. That is, there is some $D = D( \varphi) \in\mathbb{N}$  such that for any tuple 
$c$, $\varphi\left(\mathbb{Z},c\right)$ is finite if and only if
$\left|\varphi\left(\mathbb{Z},c\right)\right|\leq D$. 

The formula 
$\varphi\left(x,y\right)\land \left( \neg \exists^{>D} x \  \varphi\left(x,y\right) \right)$
satisfies $\FHP_{2}$ by Lemma \ref{lem: FHP for finite sets}. Hence it suffices to show the Theorem \ref{prop: Tpr has local FHP} holds for the formula $\psi(x,y) := \varphi(x,y) \land \left( \exists^{\infty} x \  \varphi\left(x,y\right) \right)$ with $d=1$.

Fix $\alpha>0$, and let tuples $\left(c_{i}: i \in [n] \right)$  with $c_i = (c_{i,j} :  j \in [s] )$ be so that for each $j$, $\tp(c_{i,j}) = \tp(c_{i',j})$ for all $i,i' \in [n]$ and $\psi(x,c_i)$ is consistent, hence $\varphi(\mathbb{Z},c_i)$ is infinite, for all $i \in [n]$.

This implies in particular that $\varphi\left(x, y\right)$ cannot 
contain any conjuncts of the form $mx+y = 0$. Let $\varphi'\left(x,y\right)$
be obtained from $\varphi\left(x,y\right)$ by forgetting all of
the conjuncts of the form $mx+y \neq 0$, then $\varphi'\left(x,y\right)$ is of the form considered in Claim \ref{cla: primes loc FHP 3}.

Hence there is some $\beta^{*} = \beta^{\ast}\left(\varphi'\right)>0$
and $I \subseteq [n],\left|I\right|\geq\beta^{*}n$
such that the set $\bigcap_{i\in I}\varphi'\left(\mathbb{Z},c_{i}\right)$ is infinite. In particular, it contains infinitely many elements outside 
of the \emph{finite} set 
$$\bigcup_{i\in I, j \in [s]}\left\{ a\in\mathbb{Z}:ma+c_{i,j}=0\text{ for some conjunct }mx+y_{j}\neq0\text{ occuring in }\varphi\left(x,y\right)\right\}.$$
Then every such element satisfies $\bigwedge_{i\in I}\varphi\left(x,c_{i}\right)$,
hence $\beta (\varphi):= \beta^{*} >0$ shows that
$\psi(x,y)$ satisfies local $\FHP_{1}$ in the strong form stated in Theorem \ref{prop: Tpr has local FHP}.
\end{proof}

Finally it remains to observe that the proof of Lemma \ref{lem: Basic operations preserving FH}(3) goes through to show that in any theory, if every formula $\varphi(x,y)$ with $|x|=1$  satisfies local $\FHP$ in the strong form stated in Theorem \ref{prop: Tpr has local FHP} (i.e.~only singleton coordinates in the parameter tuples are required to have the same complete type), then every formula does.

\begin{problem}
	We do not know if this is true for local $\FHP$ however (i.e.~when the type of full parameter tuples is required to be the same).
\end{problem}

\section{MS-measurable structures and large finite fields satisfy $\protect\FHP$}\label{sec: MS-meas struc FHP}

We recall the notion of an MS-measurable structure and some of its
basic properties (see \cite{macpherson2008one, elwes2008survey}). 
\begin{defn}
\label{def: MS meas structure}An $L$-structure $M$ is \emph{MS-measurable} 
if for every non-empty set $X\subseteq M^{n}$ definable (with parameters),
we have a pair $\left(\dim\left(X\right),\meas\left(X\right)\right) \in\mathbb{N} \times \mathbb{R}_{>0} \cup \{(0,0)\}$ satisfying the following
properties:

\begin{enumerate}
\item For any partitioned $L$-formula $\varphi\left(x,y\right)\in L$ with $\left|x\right|=1$,
there is a finite set $D_{\varphi}\subseteq\mathbb{N}\times\mathbb{R}_{>0} \cup \{(0,0)\}$
and finitely many $L$-formulas $\left\{ \psi_{d,\mu}\left(y\right):\left(d,\mu\right)\in D_{\varphi}\right\} $
partitioning $M^{y}$ and such that for any $b\in M^{y}$, $\models\psi_{d,\mu}\left(b\right)$
if and only if $\dim\left(\varphi\left(M,b\right)\right)=d$ and $\meas\left(\varphi\left(M,b\right)\right)=\mu$.
\item If $X$ is finite, then $\dim\left(X\right)=0$ and $\meas\left(X\right)=\left|X\right|$.
\item If $X,Y\subseteq M^{n}$ are disjoint definable sets, then: 
\[
\dim\left(X\cup Y\right)=\max\left\{ \dim\left(X\right),\dim\left(Y\right)\right\} ,
\]
\[
\meas\left(X\cup Y\right)=\begin{cases}
\meas\left(X\right)+\meas\left(Y\right) & \mbox{if }\dim\left(X\right)=\dim\left(Y\right)\mbox{,}\\
\meas\left(X\right) & \mbox{if }\dim\left(X\right)>\dim\left(Y\right)\mbox{,}\\
\meas\left(Y\right) & \mbox{if }\dim\left(X\right)<\dim\left(Y\right)\mbox{.}
\end{cases}
\]
\item (``Fubini'') Let $f:X\to Y$ be a definable surjection such that $\dim\left(f^{-1}\left(a\right)\right)=d,\meas\left(f^{-1}\left(a\right)\right)=r$
for all $a \in Y$. Then $\dim\left(X\right)=\dim\left(Y\right)+d$
and $\meas\left(X\right)=r\meas\left(Y\right)$.
\end{enumerate}
\end{defn}

Some of the main examples of MS-measurable structures are ultraproducts of finite fields (see below), finite simple groups of bounded Lie rank, vector spaces, etc. --- we refer to \cite[Example 2.4]{elwes2008survey} for further examples.

\begin{fact}
\label{fact: basic properties of MS measurable structures-1}\cite{macpherson2008one}
Let $M$ be an MS-measurable $L$-structure.

\begin{enumerate}
\item The condition (1) in Definition \ref{def: MS meas structure} holds for $L$-formulas $\varphi\left(x,y\right)$
with $\left|x\right|$ arbitrary \cite[Proposition 5.7]{macpherson2008one}.
\item Any $M' \equiv M$, in particular $\M \succ M$, is also MS-measurable.
\item A definable set $X$ is finite if and only if $\dim\left(X\right)=0$
(in which case $\meas\left(X\right)=\left|X\right|$). Moreover, for
every formula $\varphi\left(x,y\right)\in L$ there is some $k_{\varphi}\in\mathbb{N}$
such that: for any $b \in M^y$, $\varphi\left(M,b\right)$ is finite if and only if $\left|\varphi\left(M,b\right)\right|\leq k_{\varphi}$
(follows from the finiteness of the set $D_{\varphi}$ above).
\item Suppose that $\dim\left(M\right)=e$ and $\meas\left(M\right)=\nu$.
Then we can define a normalized measuring function by taking $\meas'\left(X\right)=\frac{\meas\left(X\right)}{\nu^{\dim\left(X\right)/e}}$
for each definable $X\subseteq M$. Then $\dim,\meas'$ also satisfy
all the properties above, but in addition $\meas'\left(M\right)=1$ (and hence also $\meas'\left(M^n\right)=1$ for all $n \in \mathbb{N}$).
\item For any definable (with parameters) set $B\subseteq \M^{y}$, we have a global Keisler measure
$\mu_{B}$ (working in $\M \succ M$) with $\mu_{B}(B)=1$ defined by 
\[
\mu_{B}\left(X\right)=\begin{cases}
\frac{\meas\left(X\cap B\right)}{\meas\left(B\right)} & \mbox{if }\dim\left(X\cap B\right)=\dim\left(B\right)\mbox{,}\\
0 & \mbox{if }\dim\left(X\cap B\right)<\dim\left(B\right)\mbox{}
\end{cases}
\]
for any definable $X\subseteq M^{y}$.
\end{enumerate}
\end{fact}

From now on, for simplicity of exposition we will assume that $\dim(X) \leq n$ and that $\meas$ is normalized.

\begin{fact}\label{fac: psf are MS-meas}
	
By \cite{chatzidakis1992definable} (generalizing the classical Lang-Weil estimates using partial quantifier elimination of Ax), if $\varphi(x;y) \in L_{\ring}$ with $|x| = n, |y| = m$ then there exist $C \in \mathbb{N}$ and a finite set $H_{\varphi}$ of pairs $(d,\mu) \in \{0, 1, \ldots, n\} \times \mathbb{Q}_{>0}$ so that for any finite field $\mathbb{F}_q$ and $\bar{a} \in \mathbb{F}_q^m$, 
$$\left \lvert \left \lvert \varphi(\mathbb{F}_q^n, \bar{a}) \right \rvert -  \mu q^d \right \rvert \leq C q^{d - \frac{1}{2}}$$
for some $(d,\mu) \in H_{\varphi}$. Moreover, for each $(d,\mu) \in H_{\varphi}$ there is a formula $\psi_{d,\mu}(y) \in L_{\ring}$ so that for all finite fields $\mathbb{F}_q$, $\psi_{d,\mu}(\mathbb{F}_q^m)$ defines the set of all tuples $\bar{a} \in \mathbb{F}_q^m$ satisfying this.  This implies that any infinite ultraproduct of finite fields $F := \prod_{i \in \mathbb{N}}\mathbb{F}_{q_i} / \mathcal{U}$ is MS-measurable, where given $\varphi(x,y) \in L_{\ring}$ and $\bar{a} \in F^{m}$, where $\bar{a} = (\bar{a}_i)_{i \in \mathbb{N}}/\mathcal{U}$ with $\bar{a}_i \in \mathbb{F}_{q_i}^m$, one sets $(\dim, \meas)(\varphi(F,\bar{a})) := (d, \mu)$ for $(d,\mu)$ as above (see \cite[Proposition 3.9]{elwes2008survey}).
 \end{fact}

\begin{lem}\label{lem: MS meas Keisler meas}
Let $M$ be MS-measurable. For each definable $B \subseteq \M^x$, the Keisler measure $\mu_{B}$ is definable (over the same parameters as $B$), in a strong form (by Fact \ref{fact: basic properties of MS measurable structures-1}). And for any definable sets $B_{i}\subseteq M^{x_{i}}$ and permutation $\sigma: [k] \to [k]$, we have $\mu_{B_{1} \times\ldots\times B_{k}}= \mu_{B_1} \otimes \ldots \otimes \mu_{B_k} = \mu_{B_{\sigma(1)}} \otimes \ldots \otimes \mu_{B_{\sigma(k)}}$. 
\end{lem}
\begin{proof}
	Consider $k = 2$. Let $\varphi(x_1,x_2) \in L(\M)$ be arbitrary, and let $M \prec \M$ be a small model containing the parameters of $B_1, B_2, \varphi$.  As both $\mu_{B_1} \otimes \mu_{B_2}$ and $\mu_{B_1 \times B_2}$ give measure one to $B_1 \times B_2$, we may replace $\varphi(x_1, x_2)$ by $\varphi(x_1,x_2) \land B_1(x_1) \land B_2(x_2)$  and assume $\varphi(x_1,x_2)$ defines a subset of $B_1 \times B_2$.

	Let $Y$ be the ($M$-definable) projection onto the second coordinate of the set defined by $\varphi(x_1,x_2)$. Then there are $r \in \mathbb{N}$ and $(d_i, m_i) \in (\mathbb{N} \times \mathbb{R}^{>0}) \cup \{(0,0)\}$ so that, taking $Y_i := \{ b \in Y : (\dim, \meas)(\varphi(x_1,b)) = (d_i, m_i) \}$, $Y = Y_1 \sqcup \ldots \sqcup Y_r$ is a partition of $Y$ into non-empty disjoint $M$-definable sets. Note $Y_i \subseteq B_2$. Let $(\dim,\meas)(Y_i) = (e_i, \nu_i)$. Let $c := \max \{d_i + e_i : 1 \leq i \leq r\}$, and let $I := \{1 \leq i \leq r : d_i + e_i = c\}$. Then, by ``Fubini'' (see \cite[Proposition 5.7]{macpherson2008one}), $(\dim, \meas)(\varphi(x_1,x_2)) = (c, \sum_{i \in I} m_i \nu_i )$.

	By definition of $\otimes$, additivity of Lebesgue integral and ``Fubini'' in a measurable structure we have (where $[Y_i] = \{p_2 \in S_{x_2}(M) : p_2 \vdash Y_i \}$):
	
	\begin{gather*}
		\mu_{B_1} \otimes \mu_{B_2} (\varphi(x_1,x_2)) = \int_{p_2 \in S_{x_2}(M)} \mu_{B_1} \left( \varphi(x_1, p_2)  \right) d (\mu_{B_2})|_{M}(p_2) =\\
		\sum_{i =1}^{r} \int_{p_2 \in [Y_i]} \mu_{B_1} \left( \varphi(x_1, p_2)  \right) d (\mu_{B_2})|_{M}(p_2) =
		\sum_{\{i : d_i = \dim(B_1)\}} \frac{m_i}{\meas(B_1)} \cdot \mu_{B_2}(Y_i) =   \\
		\sum_{\{i : d_i = \dim(B_1) \land e_i = \dim(B_2) \}} \frac{m_i}{\meas(B_1)} \cdot \frac{\nu_i}{\meas(B_2)}  =\\
		 \frac{ \sum_{\{i : d_i = \dim(B_1) \land e_i = \dim(B_2) \}}m_i \cdot \nu_i}{\meas(B_1 \times B_2)}.
		\end{gather*}

		Note that $d_i \leq \dim(B_1)$ and $e_i \leq \dim(B_2)$ for all $i \in [r]$. It follows that $\dim(\varphi(x_1,x_2)) = \dim(B_1 \times B_2) = \dim(B_1) + \dim(B_2)$ if and only if $c = \dim(B_1) + \dim(B_2)$, if and only if $I = \{i : d_i = \dim(B_1) \land e_i = \dim(B_2) \}$. Hence, by the calculation above, $\mu_{B_1 \times B_2} (\varphi(x_1,x_2)) = \mu_{B_1} \otimes \mu_{B_2} (\varphi(x_1,x_2))$.
		
		Repeating the argument changing the order of integration and taking projection onto the first coordinate, we similarly get $\mu_{B_1 \times B_2} (x_1, x_2) = \mu_{B_2}(x_2) \otimes \mu_{B_1}(x_1)$, hence $\mu_{B_1} \otimes \mu_{B_2} = \mu_{B_2} \otimes \mu_{B_1}$.
		
		The claim follows by induction for arbitrary $k$ using that $\otimes$ is associative on definable measures in arbitrary theories by Fact \ref{fac: otimes props}.
\end{proof}

First we consider definable families of subsets of maximal dimension and measure bounded away from zero:
\begin{lem}
\label{lem: MS FHP max dim}Let $M$ be MS-measurable. For any partitioned formulas $\varphi\left(x,y\right)$, $\psi(x,z) \in L$ there is some $\alpha = \alpha(\varphi, \psi)>0$ such that the following holds.

Assume $e \in \M^z$ and $B \subseteq \M^y$ is definable (with parameters) so that, letting $X := \psi(\M,e)$, for all $b \in B$ we have:  $\dim\left(\varphi\left(\M,b\right) \cap X\right)=\dim(X)$ and $\varphi\left(\M,b\right) \cap X$ is non-empty. Then there is some $a\in X$ such that 
\[
\dim\left(\varphi\left(a,\M\right)\cap B\right)=\dim\left(B\right)\mbox{ and }\meas\left(\varphi\left(a, \M\right)\cap B\right)\geq\alpha \cdot\meas\left(B\right).
\]
\end{lem}

\begin{proof}
Let $\varphi\left(x,y\right), \psi(x,z)$ be given. Let $\theta(x; y,z) := \varphi(x,y) \land \psi(x,z)$, and let $n \in \mathbb{N}$ and a finite set $D_{ \theta} = \{(d_i, \mu_i) : i \in [n] \}$ be as given by Fact \ref{fact: basic properties of MS measurable structures-1}(1) for $\theta(x;y,z)$. Let $\alpha := \min \{ \mu_i : i \in [n], \mu_i > 0\} > 0$ (depends only on $\varphi, \psi$). Then, for all $b \in B$, as $\varphi\left(\M,b\right) \cap X \neq \emptyset$ we have $\meas(\varphi(\M,b) \cap X) \geq \alpha$ (by Definition \ref{def: MS meas structure}(1),(2)).

Let $d:=\dim\left(X\right)$. Consider the definable set 
\[
Z := \left\{ \left(a,b\right):a\in X,b\in B,\models\varphi\left(a,b\right)\right\} \subseteq X \times B \mbox{.}
\]

Let $\pi_y:Z\to B$ be the projection of $Z$ onto the $y$-coordinate.
Then $\pi_y$ is a definable surjection from $Z$ onto $B$ (as for all $b \in B$, $\varphi(\M,b) \cap X \neq \emptyset$ by assumption),
$\dim\left(\pi_y^{-1}\left(b\right)\right)=\dim\left(\varphi\left(\M,b\right) \cap X \right)=d$ by assumption, 
and $\meas\left(\pi_y^{-1}\left(b\right)\right)=\meas\left(\varphi\left(\M,b\right) \cap X \right)\geq\alpha$.
It follows by ``Fubini'' that $\dim\left(Z\right)=\dim\left(B\right)+d$
and $\meas\left(Z\right)\geq\alpha\cdot\meas\left(B\right)$.

On the other hand, let $A := \pi_x(Z) = \left\{ a\in X : \models\exists b  \,\varphi\left(a,b\right)\right\}$, where $\pi_x:Z\to A$ is the projection onto the $x$-coordinate. Again,
$\pi_x$ is a definable surjection onto $A$, and for any $a \in A$,  $\pi_x^{-1}\left(a\right)=B\cap\varphi\left(a,\M\right)$. By MS-measurability,
there is $m \in \mathbb{N}$, $(d_i, \mu_i) \in \mathbb{N} \times \mathbb{R}_{>0} \cup \{(0,0)\}$ for $i \in [m]$, and a partition $A=A_{1} \sqcup \ldots \sqcup A_{m}$ where each
$$A_{i}=\left\{ a\in A:(\dim,\meas)\left(\pi_x^{-1}\left(a\right)\right)=(d_{i},\mu_i)\right\}$$
is a non-empty definable set. 
 Let $c := \max\left\{ \dim\left(A_{i}\right)+d_{i}:i \in [m]\right\} $.
Then by ``Fubini'' we have  $\dim\left(Z\right)=c$ and $\meas\left(Z\right)=\sum_{i\in I}\mu_{i} \cdot \meas\left(A_{i}\right)$,
where $I := \left\{ i \in [m]:\dim\left(A_{i}\right)+d_{i}=c\right\} $. By the previous paragraph we have $\dim\left(Z\right)=d+ \dim(B)$, and also for all $i \in [m]$ we have $\dim\left(A_{i}\right)\leq\dim\left(X\right)=d$ (as $A \subseteq X$)
and $d_{i}\leq \dim(B)$ (as $\pi_x^{-1}\left(a\right)\subseteq B$ for all $a \in A$). Hence 
$$I=\left\{ i \in [m]:\dim\left(A_{i}\right)=d \, \land  \, d_{i}=\dim(B)\right\} .$$ 

Now assume towards contradiction that $\meas\left(\varphi\left(a,\M\right)\cap B\right)<\alpha \cdot \meas\left(B\right)$
for all $a\in A$ with $\dim\left(\varphi\left(a,M\right)\cap B\right)=\dim\left(B\right)$. In particular  $\mu_{i}<\alpha \cdot \meas\left(B\right)$ for all $i\in I$. Also,
as $\meas\left(\M^{x}\right)=1$ and $(A_{i} : i\in I)$ are disjoint non-empty 
subsets of $\M^x$ of the same dimension $\dim\left(A_{i}\right)=d$,
it follows (see e.g. \cite[Lemma 3.3]{elwes2008survey}) that $\sum_{i\in I}\meas\left(A_{i}\right) = \meas(\bigsqcup_{i \in I} A_i) \leq 1$. Hence $\meas\left(Z\right)=\sum_{i\in I}\mu_{i} \cdot \meas\left(A_{i}\right) < \alpha \cdot \meas(B) \cdot \sum_{i\in I}\meas\left(A_{i}\right) \leq \alpha \cdot \meas(B)$. But by
the first paragraph we had $\meas\left(Z\right)\geq\alpha\cdot\meas\left(B\right)$
--- a contradiction.
\end{proof}

\begin{thm}\label{thm: FHP in MS meas}
	Let $M$ be an MS-measurable structure. Then every partitioned formula $\varphi(x,y) \in L$ with $|x| \leq d$ satisfies $\FHP_{d+1}$  with respect to the class of definable measures $\mathfrak{M}_y := \{\mu_{B}(y) : B \subseteq \M^y \textrm{ definable with parameters}\}$ (see Definition \ref{def: FHP for a class of measures}). In particular, $\varphi(x,y)$ satisfies $\FHP_{d+1}$.
\end{thm}

\begin{proof}
	The result will follow from a more precise claim:
	\begin{claim}\label{cla: claim MS meas FHP}
		For any partitioned formulas $\varphi\left(x,y\right)$, $\psi(x,z) \in L$ and $\alpha  \in \mathbb{R}_{>0}, d \in \mathbb{N}$ there is some $\beta = \beta(\varphi, \psi, d, \alpha)>0$ such that the following holds. 
		
Assume $e \in \M^z$ and $B \subseteq \M^y$ is definable (with parameters) so that, letting $X := \psi(\M,e)$, $\dim(X) \leq d$ and $\left(\mu_{B}\right)^{\otimes (d+1)}_{y_1, \ldots, y_{d+1}}\left( \exists x \in X  \bigwedge_{i \in [d+1]}\varphi(x,y_i)  \right) \geq \alpha$. Then there exists some $a \in X$ so that $\mu_{B}(\varphi(a,y)) \geq \beta$.
\end{claim}
\begin{proof}[Proof of Claim \ref{cla: claim MS meas FHP}]	
We will prove the claim by induction on $d \in \mathbb{N}$. Let $\varphi(x,y)$, $\psi(x,z)$, $d$, $\alpha$, $X = \psi(\M,e)$ with $\dim(X) = d$ and $B$ be given.
	

	\noindent \textbf{Base case $d=0$.}
	
	The only possibility is $\dim(X) = 0$, hence $X$ is finite and $|X| \leq k$ for some $k = k(\psi) \in \mathbb{N}$ by Fact \ref{fact: basic properties of MS measurable structures-1}(3), write $X = \{a_1, \ldots, a_{k'}\}$ for some $k' \leq k$. By assumption 
	\begin{gather*}
		\alpha \leq (\mu_{B})_{y_1}(\exists x \in X \, \varphi(x,y_1)) \leq (\mu_{B})_{y_1}\left( \bigvee_{t \in [k']} \varphi(a_i,y_1) \right),
	\end{gather*}
hence $(\mu_{B})_{y_1}\left( \varphi(a_t,y_1) \right) \geq \frac{\alpha}{k}$ for at least one $t \in [k']$ --- and we can take $\beta := \frac{\alpha}{k}$.

		\noindent \textbf{Inductive step $d>0$.}
	
	Assume $d \geq 1$, $\dim(X) \leq d$ and the claim holds for all $0 \leq d' < d$. Let $\alpha_d := \frac{\alpha}{2}$.
	
	\textbf{Case 1.} $\mu_{B} \left( B_d \right) \geq \alpha_d$ for $B_d := \{ b \in B : \dim(\varphi(x,b) \land X(x)) = d  \} $.
	
	As $0 < \alpha_d \leq \mu_{B} \left( B_d \right)$, we have in particular $\dim(B_d) = \dim(B)$ (by Fact \ref{fact: basic properties of MS measurable structures-1}(5)) and $\meas(B_d) \geq \alpha_d \cdot \meas(B)$. As $d >0$ we have $\varphi(\M,b) \cap X \neq \emptyset$ for all $b \in B_d$.

	By Lemma \ref{lem: MS FHP max dim} there is some $\gamma = \gamma(\varphi, \psi) > 0$ and some $a \in X$ so that 
	\begin{gather*}
		\dim\left(\varphi\left(a, \M\right)\cap B_d\right)=\dim\left(B_d\right) = \dim(B) \textrm{ and}\\
		\meas\left(\varphi\left(a,\M\right)\cap B_d\right)\geq \gamma \cdot\meas\left(B_d\right) \geq  \alpha_d \cdot \gamma \cdot \meas (B), 
	\end{gather*}
	hence $\mu_{B}(\varphi(a,y)) \geq \beta_d := \alpha_d \cdot \gamma > 0$, and $\beta_d = \beta_d(\varphi,\psi, \alpha)$.
	
	\textbf{Case 2.} $\mu_{B} \left( B_d \right) < \alpha_d$.
	
	By assumption $\left(\mu_{B}\right)^{\otimes (d+1)}\left( C  \right) \geq \alpha$ for 
	$$C := \{(b_1, \ldots, b_{d+1}) \in B^{d+1} : \exists x \in X  \bigwedge_{i \in [d+1]}\varphi(x,b_i)\}.$$
	
	And $\left(\mu_{B}\right)^{\otimes (d+1)} = \mu_{B^{d+1}}$ by Lemma \ref{lem: MS meas Keisler meas}.

	Let $B'_{d} := \{(b_1, \ldots, b_{d+1}) \in B^{d+1} : \dim(\varphi(\M,b_{d+1}) \cap X) = d \}$. As $\left(\mu_{B}\right)^{\otimes (d+1)}$ extends the product measure we have $\left(\mu_{B}\right)^{\otimes (d+1)}(B'_{d}) = \mu_{B}(B_d) < \alpha_d$, hence $\left(\mu_{B}\right)^{\otimes (d+1)} \left( C \setminus B'_d \right) \geq \alpha - \alpha_d = \frac{\alpha}{2}$. Then by Fubini there exists some $b^{\ast} \in B \setminus B_d$ so that $\left(\mu_{B}\right)^{\otimes d} \left( \{(b_1, \ldots, b_{d}) \in  B^d :  (b_1, \ldots, b_d, b^{\ast}) \in C \setminus B'_d \}\right) \geq \frac{\alpha}{2}$.

	Let $\psi'(x; z,y) := \psi(x,z) \land \varphi(x, y)$, $X' := X \cap \varphi(\M, b^{\ast}) = \psi'(\M;e,b^{\ast})$. Then $\dim(X') \leq d-1$, and we have $\left(\mu_{B}\right)^{\otimes d}_{y_1, \ldots, y_{d}}\left(  \exists x \in X'  \bigwedge_{i \in [d]}\varphi(x,y_i) \right) \geq \frac{\alpha}{2}$. 
	By the inductive assumption there exists some $\beta_{<d} = \beta_{<d}(\varphi, \psi', d-1, \frac{\alpha}{2}) > 0 $ and some $a \in X' \subseteq X$ so that $\mu_{B}(\varphi(a,y)) \geq \beta_{<d}$. Then $\beta := \min \{ \beta_{<d}, \beta_d \} > 0$ is chosen depending only on $\varphi, \psi, d, \alpha$ and satisfies the requirement. This concludes the proof of the claim.
\end{proof}

	Now the theorem follows: given $\varphi(x,y)$ with $|x| \leq d$ and $\alpha > 0$, we can choose  $\beta = \beta(\varphi, \alpha) > 0$ satisfying $\FHP_{d+1}$ with respect to $\mathfrak{M}_y$ by applying the claim to  $\varphi(x,y), \psi (x_1, \ldots, x_d; z) := \left(\bigwedge_{i \in [d]} x_i = x_i\right), d, \alpha$. As all finite sets are definable, $\mathfrak{M}_y$ contains all finitely supported probability measures on $\M^y$, so the ``in particular'' part follows by Remark \ref{rem: FHP iff FHP for finite measures}.
\end{proof}

\begin{rem}\label{rem: indep from MS-meas M}
	In fact, the proof of Theorem \ref{thm: FHP in MS meas} shows that given any finite set $D \subseteq \mathbb{N} \times \mathbb{R}_{>0} \cup \{(0,0)\}$ and $\alpha \in \mathbb{R}_{>0}$ there exists $\beta = \beta(D,\alpha) \in \mathbb{R}_{>0}$ satisfying the following. Let  $M$ be any MS-measurable $L$-structure and $\varphi(x,y) \in L$ is an $L$-formula with $|x| \leq d$ so that, letting $\left(\bigwedge_t \varphi \right)(x;y_1, \ldots, y_t) := \bigwedge_{i \in [t]} \varphi(x,y_i)$,  we have $D_{\bigwedge_{t}\varphi} \subseteq D$  for all $t \in [d]$ (where $D_{\bigwedge_{t}\varphi}$ is the set given by Fact \ref{fact: basic properties of MS measurable structures-1}(1) for the formula $\left(\bigwedge_t \varphi \right)(x;y_1, \ldots, y_t) $ in the MS-measurable structure $M$).  \emph{Then} $\varphi(x,y)$ satisfies $\FHP(d+1,\alpha,\beta)$ with respect to the class of measures $\mathfrak{M}_y$ in $M$ (see Definition \ref{def: FHP for a class of measures}).
\end{rem}

As a corollary of Theorem \ref{thm: FHP in MS meas}, we get that definable families of sets of bounded description complexity in large finite fields  satisfy the fractional Helly property, in the following sense. Given an $L$-structure $M$, we say that a family $\mathcal{F}$ of subsets of $M^d$ is \emph{definable} if there exist partitioned $L$-formulas $\varphi(x,y), \psi(y,z)$, with $|x|=d$ and $y,z$ finite tuples of variables, and $e \in M^z$ so that $\mathcal{F} = \{\varphi(M,b) : b \in \psi(M,e) \}$. We say that the \emph{description complexity} of a definable family $\mathcal{F}$ is $\leq (D_1,D_2)$ if there exist some $\varphi$ of length $\leq D_1$ and some $\psi$ of length $\leq D_2$  satisfying this for some $e$ (see e.g.~\cite[Section 1.2]{tao2015expanding} for a discussion and examples). We view fields as structures in the ring language $L_{\ring} = (+, \times, 0,1)$.
\begin{cor} \label{cor: FHP in finite fields}
 For every $(D_1,D_2) \in \mathbb{N}$ and $\alpha > 0$ there exist $\beta = \beta(D_1, \alpha) >0$ and $N = N(D_1, D_2, \alpha) \in \mathbb{N}$ satisfying the following.

Let $d \in \mathbb{N}$, $F$ a finite field with $|F| \geq N$ and $\mathcal{F} \subseteq F^d$ a definable family of sets of description complexity $\leq (D_1,D_2)$ so that
$$\left \lvert \left\{ I \subseteq \mathcal{F} : |I| = d+1 \  \land \  \bigcap_{S \in I} S \neq \emptyset \right\} \right \rvert \geq \alpha {|\mathcal{F}| \choose d+1}.$$
Then $\bigcap_{S \in J} S \neq \emptyset$ for some $J \subseteq \mathcal{F}$ with $|J| \geq \beta |\mathcal{F}|$.
\end{cor}
\begin{proof}
Fix $(D_1,D_2)$ and $\alpha > 0$.
Let $\Delta$  be the set of all $L_{\ring}$-formulas of length $\leq D_1$, then $\Delta$ is finite. For each $\varphi(x,y) \in L_{\ring}$ and $t \in \mathbb{N}$, let finite set $H_{\bigwedge_t \varphi}$ be as given by Fact \ref{fac: psf are MS-meas} for $\bigwedge_t \varphi(x; y_1, \ldots, y_t)$, and let $D := \bigcup \{H_{\bigwedge_t \varphi} : \varphi(x,y) \in \Delta, t \in [D_1] \}$, then $D$ is still finite. It follows by Fact \ref{fac: psf are MS-meas} that if $F := \prod_{i \in \mathbb{N}} F_i / \mathcal{U}$ is an \emph{arbitrary}  infinite ultraproduct of finite fields, $\varphi(x,y) \in L_{\ring}$ is any formula of length $\leq D_1$ and $t \in [D_1]$, then $F$ is MS-measurable and $D_{\bigwedge_{t} \varphi} \subseteq D$ (where  $D_{\bigwedge_{t} \varphi}$ is given by Fact \ref{fact: basic properties of MS measurable structures-1}(1) for the formula $\left(\bigwedge_t \varphi \right)(x;y_1, \ldots, y_t) $ in  $F$). It follows by Theorem \ref{thm: FHP in MS meas} and Remark \ref{rem: indep from MS-meas M} that there is $\beta = \beta(D, \alpha) = \beta(D_1, \alpha) > 0$ so that $\varphi(x,y)$ satisfies $\FHP(d+1, \alpha, \beta)$ with respect to $\mathfrak{M}_y$ in $F$.

Now assume towards contradiction that no $N \in \mathbb{N}$ satisfies the conclusion of the corollary with respect to $\beta/2$.  Then for every $i \in \mathbb{N}$  there exists a finite field $F_i$ with $|F_i| \geq i$ and a definable family $\mathcal{F}_i \subseteq F_i^d$ of description complexity $\leq (D_1,D_2)$, witnessed by some $\varphi_i(x,y_i), \psi_i(y_i,z_i), e_i \in F_i^{|z_i|}$, so that taking $B_i := \psi_i(F_i, e_i)$ and $C_i := \{(b_1, \ldots, b_{d+1}) \in B_i^{d+1} : F_i \models \exists x \bigwedge_{t \in [d+1]} \varphi_i(x,b_i) \}$, we have $|C_i| \geq \alpha |B_i^{d+1}|$, but  $|\varphi_i(a,F_i) \cap B_i| < \frac{\beta}{2} |B_i|$ for all $a \in F_i^{|x|}$.
 As there are only finitely many formulas of length at most $\max\{D_1, D_2\}$, passing to a subsequence we may assume $\varphi_i(x,y_i) = \varphi(x,y), \psi_i(y_i,z_i) = \psi(y,z)$ for all $i \in \mathbb{N}$. Let $\mathcal{U}$ be a non-principal ultrafilter on $\mathbb{N}$ and $F := \prod_{i \in \mathbb{N}} F_i / \mathcal{U}$, $B := \prod_{i \in \mathbb{N}} B_i / \mathcal{U}$, $C := \prod_{i \in \mathbb{N}} C_i / \mathcal{U}$. Then  $F$ is an infinite  MS-measurable structure, and by the definition of $(\dim, \meas)$ from Fact \ref{fac: psf are MS-meas} we have $\dim(C) \geq \dim(B^{d+1})$ (hence equal to it as $C \subseteq B^{d+1}$) and $\meas(C) \geq \alpha \meas(B^{d+1})$, so using \L o\'s's theorem, Fact \ref{fact: basic properties of MS measurable structures-1}(5) and Lemma \ref{lem: MS meas Keisler meas}, we have 
$$(\mu_{B})^{\otimes (d+1)} \left( \left\{(b_1, \ldots, b_{d+1}) \in B^{d+1} : F \models  \exists x \bigwedge_{t \in [d+1]} \varphi_i(x,b_i) \right\} \right) \geq \alpha. $$ 
Then, by the choice of $\beta$ in the first paragraph, there is some  $a \in F^{|x|}$, $a = (a_i)_{i \in \mathbb{N}}/\mathcal{U}$ with $a_i \in F_i^{|x|}$, so that $\mu_{B}(\varphi(a, F)) \geq \beta$, i.e.~$\dim(\varphi(a,F) \cap B) = \dim(B)$ and $\meas(\varphi(a,F) \cap B) \geq \beta  \cdot \meas(B)$. Again, by the  definition of $(\dim, \meas)$ in $F$ from Fact \ref{fac: psf are MS-meas} this implies that $|\varphi_i(a_i,F_i) \cap B_i| \geq \beta |B_i|$ for a $\mathcal{U}$-large set of $i \in \mathbb{N}$ --- a contradiction.
\end{proof}

\begin{cor}
Assume $M$ is an MS-measurable structure. Then $M^x$ with $|x| \leq d$ satisfies colorful $\FHP_{d+1}$ with respect to $\mathfrak{M}$.
\end{cor}
\begin{proof}
	By Theorem \ref{thm: FHP in MS meas} and Proposition \ref{prop: colorful FHP meas bdd by burden} (noting that the measures in $\mathfrak{M}$ are definable and pairwise commuting by Lemma \ref{lem: MS meas Keisler meas}).
\end{proof}

Similarly to Corollary \ref{cor: FHP in finite fields}, this implies colorful fractional Helly number $d+1$ for definable families of subsets of $\mathbb{F}_q^d$ of bounded description complexity.

\section{Ultraproducts of the $p$-adics have $\protect\FHP$}\label{sec: ultraprod of Qp}

In this section we prove an Ax-Kochen-Ershov style result for the FHP property in  henselian valued fields (Theorem \ref{thm: AKE for FHP}). We only consider the equi-characteristic $0$ case here (sufficient for the intended application to ultraproducts of the $p$-adics, Corollary \ref{cor: FHP in ultraprod Qp}) and leave the obvious more general questions for the future (see Problem \ref{prob: FHP in valued fields}).

\subsection{Quantifier elimination in henselian valued fields}\label{sec:QE hens val}

We recall some basic facts about the RV language for valued fields (we use \cite{flenner2011relative} as a reference, see also \cite[Section 5.1]{zbMATH07542897}). Fix a valued field $K$, with valuation $\val: K^{\times} \to \Gamma$, value
group $\Gamma$, residue field $k$ and valuation ring $\mathcal{O}$. Let $\RV$ be the quotient
group $K^{\times}/\left(1+\mathfrak{m}\right)$ where $\mathfrak{m}=\left\{ x\in K:\val\left(x\right)>0\right\} $
is the maximal ideal of $\mathcal{O}$. We have a short exact
sequence $1\to k^{\times}\to\RV \overset{\val_{\rv}}{\to}\Gamma\to 0$, where $\val_{\rv}$ is induced by the valuation map and, for $a \in \mathcal{O} \setminus \mathfrak{m}$, the embedding $k^{\times} \to \RV$ sends $a + \mathfrak{m} \in k^{\times}$ to $a(1+ \mathfrak{m}) \in \RV^{\times}$. 
\begin{defn}
	We consider the two-sorted structure $M=\left(K,\RV,\rv\right)$
in the language $L$ (fixed for the rest of the section) consisting of: 
\begin{itemize}
\item the quotient map $\rv:K\to\RV$,
\item on the sort $K$, the ring structure,
\item on the sort $\RV$, the structure $\cdot,1$ of a multiplicative group,
a symbol $0$, a symbol $\infty$ and a ternary relation $\oplus$.

\end{itemize}
The multiplicative group structure on $\RV$ is interpreted as the group structure
induced from $K^{\times}$ and $0\cdot x=x\cdot0=0$, $\infty = \rv(0)$. The relation $\oplus$ is interpreted as the partially defined addition inherited from $K$: $\oplus(a,b,c) \iff \exists x,y,z \in K \left( a =\rv(x) \land b = \rv(y) \land c = \rv(z) \land x+y = z \right)$.  We let $\RV_{\ast}$ denote the structure $(\RV, 0,1,\cdot, \oplus)$, and denote its language by $L_{\RV}$.
\end{defn}

\begin{rem}\label{rem: WD}
\begin{enumerate}

\item Let $\WD(x,y)$ be the (definable in $\RV_{\ast}$) set of pairs of elements in $\RV$ for which the sum is well-defined as 
$$\forall z, z' (\oplus(x,y,z) \land \oplus(x,y,z') \Rightarrow z=z').$$
 Given a pair of elements $x,y \in \RV$ such that $\WD(x,y)$ holds,  $x+y$ denotes the unique element $z \in \RV$ satisfying $\oplus(x,y,z)$. 
\item For $a,b \in K$, we have $\WD(\rv(a),\rv(b)) \Leftrightarrow \val(a+b) = \min \left\{\val(a), \val(b) \right\}$, in which case $\rv(a+b) = \rv(a) + \rv(b)$ (see \cite[Proposition 2.4]{flenner2011relative}).

\item The relation $\val_{\rv}(x) \leq \val_{\rv}(y)$  on $\RV$ is definable in $\RV_{\ast}$ \cite[Proposition 2.8(1)]{flenner2011relative}. Hence the multiplicative group $k^{\times} \cong \ker \val_{\rv}$ is definable in $\RV_{\ast}$, thus the ordered abelian group $\Gamma$ is interpretable in $\RV_{\ast}$, so using $\oplus$ the field $k$ is also interpretable in $\RV_{\ast}$.
\end{enumerate}
\end{rem}

We will use the following Ax-Kochen-Ershov style relative quantifier elimination result: 
\begin{fact}
\label{fac: Flenner's cell decomposition}\cite[Propositions 4.3 and 5.1]{flenner2011relative}
Let $K$ be a henselian valued field with $\ch\left(k\right)=0$ and $M = (K, \RV, \rv)$.
\begin{enumerate}
\item Suppose $S\subseteq K$ is $A$-definable in $M$, for some set of parameters $A$ in $M$. Then there are $n \in \mathbb{N}$, $\alpha_{1},\ldots,\alpha_{n}  \in K \cap \acl(A)$ and an $\acl(A)$-definable subset $D\subseteq\RV^{n}$ such that 
\[
S=\left\{ x\in K:\left(\rv\left(x-\alpha_{1}\right),\ldots,\rv\left(x-\alpha_{n}\right)\right)\in D\right\} \mbox{.}
\]

\item $\RV_{\ast}$ is fully stably embedded (i.e.~the structure on $\RV_{\ast}$ 
induced from $M$, with parameters, is precisely the one described above).
\end{enumerate}
Moreover, this holds with arbitrary additional structure on $\RV_{\ast}$ (see the discussion before \cite[Propositions 4.3]{flenner2011relative}.
\end{fact}

\begin{rem}\label{rem: add unif in Flenner}
	\begin{enumerate}
		\item By compactness we get that the stable embeddedness in Fact \ref{fac: Flenner's cell decomposition}(2) is \emph{uniform}. I.e.~for every $\varphi(x,y) \in L$ with $x$ a tuple of sort $\RV$ and $y$ a tuple of arbitrary sorts there is some $\psi(x,z) \in L_{\RV}$ satisfying the following: if $M = (K,\RV,\rv)$ is a  henselian valued field of equicharacteristic $0$ and $b$ is a tuple in $M$ corresponding to $y$, then there is some $c$ a tuple in $\RV(M)$ corresponding to $z$ so that for all tuples $a$ in $\RV(M)$ corresponding $x$ we have 
		\begin{gather*}
		M \models \varphi(a,b) \Leftrightarrow \RV_{\ast}(M) \models \psi(a; c).
	\end{gather*} 
	
	\item Using (1), Fact \ref{fac: Flenner's cell decomposition}(1) and compactness, we have the following additional uniformity in Fact \ref{fac: Flenner's cell decomposition}(1). For every $\varphi(x;y) \in L$, with $x$ of sort $K$, $|x|=1$, and $y$ any finite tuple of variables of any sorts, we can choose $\psi(x_1, \ldots, x_n; z) \in L_{\RV}$ depending only on $\varphi$, with $|x_i|=1$ of sort $\RV$ and $z$ a finite tuple of variables of sort $\RV$, 
	so that: if $M = (K,\RV,\rv)$ is a  henselian valued field of equicharacteristic $0$ and $b$ is a tuple in $M$ corresponding to $y$, then there is some $c$ a tuple in $\RV(M)$ corresponding to $z$ and some $\alpha_1, \ldots, \alpha_n \in K(M)$ so that for any $a \in K(M)$,
	\begin{gather*}
		M \models \varphi(a,b) \Leftrightarrow \RV_{\ast}(M) \models \psi(\rv(a - \alpha_1), \ldots, \rv(a - \alpha_n); c).
	\end{gather*} 

	\end{enumerate}
\end{rem}

\subsection{Ax-Kochen-Ershov for FHP}
In this section we prove the following theorem:
\begin{thm}\label{thm: AKE for FHP}
	Let $M=\left(K,\RV,\rv\right)$ be an equi-characteristic $0$ henselian valued field (viewed as an $L$-structure, see Section \ref{sec:QE hens val}). Then $M$ satisfies $\FHP$ if and only if both the residue field $k$ and the (ordered) value group $\Gamma$ satisfy $\FHP$.
\end{thm}

In the rest of the section we prove this theorem (we do not try to optimize the bounds). Throughout, we let $M=\left(K,\RV,\rv\right)$ be an equi-characteristic $0$ henselian valued field, viewed as an $L$-structure, and $\M \succ M$ a monster model.  By $\aleph_1$-saturation, we may always
assume that $\M$ admits a cross-section map $\ac: K(\M) \to k^{\times}(\M)$,
so we can view $\M$ also as a structure in the language
$L_{\ac}$ with the angular component map added to the language.

The distinction of the cases below is inspired by the analysis in \cite[Section 7.2]{chernikov2014theories}, \cite{chernikov2016henselian} and \cite{zbMATH07542897}. The following two lemmas are easy to verify by the basic properties of valuations (see the proofs of \cite[Lemma 6]{chernikov2010indiscernible} and \cite[Lemma 7]{chernikov2010indiscernible}, replacing $c_{-\infty}, c_{+ \infty}$ by $c_0, c_n$ and restricting to $1 \leq i<j \leq n-1$).
\begin{lem}
\label{lem: val on indisc seq}There is a \emph{finite} set of formulas
$\Delta\subseteq L$ such that the following holds for any $n \geq 3$. Let $\left(b_{i}: 0 \leq i \leq n \right)$ be a $\Delta$-indiscernible sequence
of singletons in $\mathbb{K}$, and consider the function $\left(i,j\right)\mapsto\val\left(b_{j}-b_{i}\right)$
for $1 \leq i<j \leq n-1$ (so we ignore the first and the last elements of the sequence). Then one of the following cases occurs:

\begin{enumerate}
\item it is strictly increasing depending only on $i$ (in which case we
say that the sequence $(b_i : 1 \leq i \leq n-1)$ is \emph{pseudo-convergent}),
\item it is strictly decreasing depending only on $j$ (so the sequence $(b_i : 1 \leq i \leq n-1)$ taken in the reverse direction is pseudo-convergent),
\item it is constant (we refer to such a sequence $(b_i : 1 \leq i \leq n-1)$ as a ``fan'').
\end{enumerate}
\end{lem}

\begin{lem}
\label{lem: val phase change} For any $n \geq 3$, let $\left(b_{i}:0 \leq i \leq n\right)$ be a pseudo-convergent sequence of singletons from $\mathbb{K}$. Then for any $d\in\mathbb{K}$ there
is some $0 \leq j \leq n$ such that:

\begin{enumerate}
\item for all $0 \leq i<j$: $\val\left(b_{n}-b_{i}\right)<\val\left(d-b_{n}\right)$,
which implies $\val\left(d-b_{i}\right)=\val\left(b_{n}-b_{i}\right)$
and $\ac\left(d-b_{i}\right)=\ac\left(b_{n}-b_{i}\right)$,
\item for all $n-1 \geq i>j$: $\val\left(b_{n}-b_{i}\right)>\val\left(d-b_{n}\right)$,
which implies $\val\left(d-b_{i}\right)=\val\left(d-b_{n}\right)$
and $\ac\left(d-b_{i}\right)=\ac\left(d-b_{n}\right)$.
\end{enumerate}
\end{lem}

First we consider the key special case:
\begin{lem}
\label{lem: reduction to RV, key lemma} Assume $\RV_{\ast}$ has $\FHP$. Then every   partitioned formula $\psi\left(x;y,z\right)\in L$
of the form $\varphi\left(\rv\left(x-y\right),z\right)$, with $\varphi\left(x',z\right)$ an $L_{\RV}$-formula and $x,y$ singleton variables of sort $K$ and $z$ an arbitrary finite tuple of variables of sort $\RV$, has $\FHP$.
\end{lem}

\begin{proof}
Let $k = k(\varphi) \in\mathbb{N}$ be sufficiently large, to be determined in the proof. Let $\alpha>0$ be arbitrary, and we find $\beta=\beta\left(\varphi,\alpha,k\right)>0$
such that the following holds. 
For any $n \in \mathbb{N}$, let $\left(b_{i}:i \in [n]\right)$ with $b_i \in K(\M)$ and $\left(c_{i} : i \in [n]\right)$ with $c_i \in \RV(\M)^{z}$ be arbitrary. Let $\mathcal{F} := \left\{ \psi\left(\M; b_{i},c_{i}\right):i \in [n]\right\} $.
Let $\mathcal{C}_{0} := \Cons_{k}\left(\mathcal{F}\right)$ (we are using the notation from Section \ref{sec: basic FHP and Shelah classification}). Assume that
$\left|\mathcal{C}_{0}\right|\geq\alpha{n \choose k}$ (and without loss of generality $n$ is arbitrarily large with respect to $k$ and $\alpha$), then there
is some $I\subseteq\left[n\right]$ such that $\left|I\right|\geq\beta n$
and $\left\{ \psi\left(M,b_{i}c_{i}\right):i\in I\right\} $ is consistent.

Let finite $\Delta\subseteq L$ be as given by Lemma \ref{lem: val on indisc seq}. Let $k_{1}\in\mathbb{N}$ be arbitrary. Assuming that $k = k(k_1, \Delta)$ is sufficiently large, by Ramsey's theorem every sequence of elements from $K(\M)$ of length $k$ contains a $\Delta$-indiscernible subsequence of length
$k_{1}$. Let $\mathcal{C}_{1} := \left\{ S\in\Cons_{k_{1}}\left(\mathcal{F}\right):\left(b_{i}:i\in S\right)\mbox{ is }\Delta\mbox{-indiscernible}\right\} $, where indiscernibility is with respect to the  order induced from the natural order on $\left[n\right]$. 
Then for every $S\in\mathcal{C}_{0}$ there is some $S'\subseteq S$
such that $S'\in\mathcal{C}_{1}$. On the other hand, each $S'\in\mathcal{C}_{1}$
can be contained in at most ${n-k_{1} \choose k-k_{1}}$ sets $S\in\mathcal{C}_{0}$.
Hence 
\[
\left|\mathcal{C}_{1}\right|\geq\frac{\alpha{n \choose k}}{{n-k_{1} \choose k-k_{1}}}\geq\alpha_{1}{n \choose k_{1}}\mbox{,}
\]
for some $\alpha_{1}=\alpha_{1}\left(\alpha,k,k_{1}\right)>0$.

Now let 
\begin{gather*}
	\mathcal{D}_{1} := \left\{ S\in\mathcal{C}_{1}:\left(b_{i}:i\in  S^{-}\right)\mbox{ is pseudo-convergent}\right\} \mbox{,}\\
	\mathcal{D}_{2} := \left\{ S\in\mathcal{C}_{1}:\mbox{ the reverse of }\left(b_{i}:i\in S^{-}\right)\mbox{ is pseudo-convergent}\right\} \mbox{,}\\
	\mbox{and }\mathcal{D}_{3}=\left\{ S\in\mathcal{C}_{1}:\left(b_{i}:i\in S^{-}\right)\mbox{ is a fan}\right\},
\end{gather*}
where $S^{-}$ is the set obtained from $S$ by removing the minimal and the maximal elements (with respect to the order induced from $[n]$). By Lemma \ref{lem: val on indisc seq}, $\mathcal{C}_{1}=\mathcal{D}_{1}\cup\mathcal{D}_{2}\cup\mathcal{D}_{3}$,
hence $\left|\mathcal{D}_{i}\right|\geq\frac{\alpha_{1}}{3}{n \choose k_{1}}$
must hold for at least one $i\in\left\{ 1,2,3\right\} $.

~

\noindent \textbf{Case 1}. $\left|\mathcal{D}_{1}\right|\geq\frac{\alpha_{1}}{3}{n \choose k_{1}}$.

For each $S\in\mathcal{D}_{1}$, fix some $a_{S}\in K(\M)$ such
that $a_{S}\models\left\{ \psi\left(x;b_{i},c_{i}\right):i\in S\right\} $.
Let $j_{S}\in S^{-}$ be as given by Lemma \ref{lem: val phase change}
for $a_{S}$ and the sequence $\left(b_{i}:i\in S^{-}\right)$. By pigeon-hole,
for every $S\in\mathcal{D}_{1}$ at least one of the sets $L_{S} := \left\{ i\in (S^{-})^{-}:i<j_{S}\right\}$,
$R_{S} := \left\{ i\in (S^{-})^{-}:j_{S}<i\right\}$ must have size
$\geq k_2 := \lfloor \frac{(k_{1}-4)}{2} \rfloor$. Let 
\begin{gather*}
	\mathcal{E}_{1} := \left\{ S\in\Cons_{k_{2}}\left(\mathcal{F}\right):\exists S'\in\mathcal{D}_{1}\left(S\subseteq L_{S'}\right)\right\} \mbox{,}\\
	\mathcal{E}_{2} := \left\{ S\in\Cons_{k_{2}}\left(\mathcal{F}\right):\exists S'\in\mathcal{D}_{1}\left(S\subseteq R_{S'}\right)\right\} \mbox{,}\\
	\mbox{and }\mathcal{E}=\mathcal{E}_{1}\cup\mathcal{E}_{2}\mbox{.}
\end{gather*}

As every $S\in\mathcal{D}_{1}$ contains some $S'\in\mathcal{E}$,
by double counting as above we get 
\[
\left|\mathcal{E}\right|\geq\frac{\frac{\alpha_{1}}{3}{n \choose k_{1}}}{{n-k_{2} \choose k_{1}-k_{2}}}\geq\alpha_{2}{n \choose k_{2}}
\]
 for some $\alpha_{2}=\alpha_{2}\left(\alpha_{1},k_{1}\right)>0$.
Hence $\left|\mathcal{E}_{i}\right|\geq\frac{\alpha_{2}}{2}{n \choose k_{2}}$
for at least one $i\in\left\{ 1,2\right\} $.
\medskip

\noindent \textbf{Case 1.1. }$\left|\mathcal{E}_{1}\right|\geq\frac{\alpha_{2}}{2}{n \choose k_{2}}$.

Unwinding the definition, for every $S\in\mathcal{E}_{1}$ there is
some $a_{S}\in K(\M)$ such that $a_{S}\models\left\{ \psi\left(x; b_{i},c_{i}\right):i\in S\right\} $
and, denoting by $i_{S}$ the last element of $S$ in the order induced
from $[n]$, $\val\left(b_{i_{S}}-b_{i}\right)<\val\left(a_{S}-b_{i_{S}}\right)$, and hence $\rv\left(a_{S}-b_{i}\right)=\rv\left(b_{i_{S}}-b_{i}\right)$, 
for all $i\in S,i<i_{S}$ (see Lemma \ref{lem: val phase change}).

As there are at most $n$ choices for $i_{S}$ as $S$ varies over
$\mathcal{E}_{1}$, it follows that for some $i^{*} \in [n]$, the set 
$$\mathcal{G}_1 := \left\{ S\in\Cons_{k_{2}-1}\left(\mathcal{F}\right):\exists S'\in\mathcal{E}_{1}\left(S=S'\setminus\left\{ i_{S'}\right\} \land i_{S'}=i^{*}\right)\right\} $$
has size at least $\frac{\left|\mathcal{E}_{1}\right|}{n}\geq\alpha_{3}{n \choose k_{2}-1}$
for some $\alpha_{3}=\alpha_{3}\left(\alpha_{2},k_{2}\right)>0$.
Note that, assuming $k_{2}-1\geq 1$, for $\mathcal{H} := \bigcup\mathcal{G}_1$
we have $\left|\mathcal{H}\right|\geq\left(\alpha_{3}{n \choose k_{2}-1}\right)^{\frac{1}{k_{2}-1}}\geq\alpha_{4}n$
for some $\alpha_{4}=\alpha_{4}\left(\alpha_{3},k_{2}\right)>0$.
Note that for every $i\in\mathcal{H}$ there is some $a_{S}\in K(\mathbb{M})$
such that $\models\psi\left(a_{S};b_{i},c_{i}\right)$, that is $\models\varphi\left(\rv\left(a_{S}-b_{i}\right),c_{i}\right)$, 
and $\rv\left(a_{S}-b_{i}\right)=\rv\left(b_{i_{*}}-b_{i}\right)$.
Combining, we get that $b_{i^{*}}\models\left\{ \psi\left(x; b_{i}, c_{i}\right):i\in\mathcal{H}\right\} $.
Hence we have found a consistent subfamily $\mathcal{F}_{0}$ of $\mathcal{F}$ 
of size $\geq\beta_{1}n$ for $\beta_{1} := \alpha_{4}>0$.
\medskip

\noindent\textbf{Case 1.2.} $\left|\mathcal{E}_{2}\right|\geq\frac{\alpha_{2}}{2}{n \choose k_{2}}$.

Unwinding the definition, for every $S\in\mathcal{E}_{2}$ there is
some $a_{S}\in K(\mathbb{M})$ such that $a_{S}\models\left\{ \psi\left(x; b_{i}, c_{i}\right):i\in S\right\} $
and, denoting by $i_{S}$ the last element of $S$ in the order induced
from $[n]$, $\val\left(a_{S}-b_{i}\right)<\val\left(b_{i_{S}}-b_{i}\right)$, and hence $\rv\left(a_{S}-b_{i}\right)=\rv\left(a_{S}-b_{i_{S}}\right)$, 
for all $i\in S$, $i<i_{S}$ (see Lemma \ref{lem: val phase change}).

Again, as there are at most $n$ choices for $i_{S}$ as $S$ varies
over $\mathcal{E}_{2}$, it follows that for some $i^{*} \in [n]$, the
set 
$$\mathcal{G}_2 := \left\{ S\in\Cons_{k_{2}-1}\left(\mathcal{F}\right):\exists S'\in\mathcal{E}_{2}\left(S=S'\setminus\left\{ i_{S'}\right\} \land i_{S'}=i^{*}\right)\right\} $$
has size at least $\frac{\left|\mathcal{E}_{2}\right|}{n}\geq\alpha_{3}{n \choose k_{2}-1}$
for some $\alpha_{3}=\alpha_{3}\left(\alpha_{2},k_{2}\right)>0$.

Let $\psi'\left(x';z,z'\right)$ be the $L_{\RV}$-formula $\varphi\left(x',z\right)\land\left(\val_{\rv}(x')<\val_{\rv} \left(z'\right)\right)$.
Consider the family 
$$\mathcal{F}' := \left\{ \varphi\left(x',c_{i}\right)\land\left(\val_{\rv}(x')<\val_{\rv}\left(b_{i^{*}}-b_{i}\right)\right):i \in [n]\right\} =\left\{ \psi'\left(\M;c_{i},d_{i}\right):i \in [n] \right\} $$
of subsets of $\RV(\M)$, where $d_{i} := \rv\left(b_{i^{*}}-b_{i}\right)$.
Note that $\mathcal{G}_2 \subseteq \Cons_{k_{2}-1}\left(\mathcal{F}'\right)$,
hence $\left|\Cons_{k_{2}-1}\left(\mathcal{F}'\right)\right|\geq\alpha_{3}{n  \choose k_{2}-1}$.

As $\psi'\left(x';z,z'\right)$ is an $L_{\RV}$-formula, by assumption (and Remark \ref{rem: add unif in Flenner}(1)) 
it has FHP. Assuming $k_{2}-1$ is larger than the fractional Helly number
of $\psi'$, there is some $\beta_{2}=\beta_{2}\left(\psi',\alpha_{3}\right)>0$
and $I\subseteq\left[n\right]$ such that $\left|I\right|\geq\beta_{2}n$
and 
$$\left\{ \varphi\left(x',c_{i}\right)\land\left(\val_{\rv}\left(x'\right)< \val \left(b_{i^{*}}-b_{i}\right)\right):i\in I\right\} $$
is consistent, say realized by $e\in\RV(\M)$ (where $\val \left(b_{i^{*}}-b_{i}\right)$
is a shortcut for $\val_{\rv}\left(\rv\left(b_{i^{*}}-b_{i}\right)\right)$
in our language). Let $a\in K(\M)$ be such that $\rv\left(a-b_{i^{*}}\right)=e$.
Then $\val\left(a-b_{i^{*}}\right)=\val_{\RV}\left(e\right)<\val\left(b_{i^{*}}-b_{i}\right)$
for all $i\in I$, hence $\rv\left(a-b_{i}\right)=\rv\left(a-b_{i^{*}}\right)=e$
for all $i\in I$. And so $a\models\left\{ \varphi\left(\rv\left(x-b_{i}\right),c_{i}\right):i\in I\right\} $.
Hence we have found a consistent subfamily $\mathcal{F}_{0}$ of $\mathcal{F}$
of size $\geq\beta_{2}n$.

~

\noindent \textbf{Case 2}. $\left|\mathcal{D}_{2}\right|\geq\frac{\alpha_{1}}{3}{n \choose k_{1}}$.

Reduces to Case 1 by working with the reverse order on $[n]$.

~

\noindent\textbf{Case 3.} $\left|\mathcal{D}_{3}\right|\geq\frac{\alpha_{1}}{3}{n \choose k_{1}}$.

For each $S\in\mathcal{D}_{3}$, the $\Delta$-indiscernible sequence
$\left(b_{i}:i\in S^{-} \right)$ is a fan, that is there is some $\gamma_{S}\in\Gamma$
such that $\val\left(b_{j}-b_{i}\right)=\gamma_{S}$ for all $i<j\in S^{-}$.
We fix some $a_{S}\in K(\M)$ such that $a_{S}\models\left\{ \psi\left(x;b_{i},c_{i}\right):i\in S\right\} $.

For each $S\in\mathcal{D}_{3}$, let 
\begin{gather*}
	S_{<} := \left\{ i\in S^-:\val\left(a_{S}-b_{i}\right)<\gamma_{S}\right\}, S_{>} := \left\{ i\in S^-:\val\left(a_{S}-b_{i}\right)>\gamma_{S}\right\} \\
	S_{=} := \left\{ i\in S^- : \val\left(a_{S}-b_{i}\right)=\gamma_{S}\right\}.
\end{gather*}
Then $S^-$ is a disjoint union of $S_{<},S_{>},S_{=}$. Let $k_{2} := \left\lfloor \frac{k_{1}-2}{3}\right\rfloor $.
Let 
\begin{gather*}
	\mathcal{K}_{1} := \left\{ S'\in\Cons_{k_{2}}\left(\mathcal{F}\right):\exists S\in\mathcal{D}_{3}\left(S'\subseteq S_{<}\right)\right\} \mbox{,}\\
	\mathcal{K}_{2} := \left\{ S'\in\Cons_{k_{2}}\left(\mathcal{F}\right):\exists S\in\mathcal{D}_{3}\left(S'\subseteq S_{>}\right)\right\} \mbox{,}\\
	\mathcal{K}_{3} := \left\{ S'\in\Cons_{k_{3}}\left(\mathcal{F}\right):\exists S\in\mathcal{D}_{3}\left(S'\subseteq S_{=}\right)\right\} \mbox{,}\\
	\mathcal{K} := \mathcal{K}_{1}\cup\mathcal{K}_{2}\cup\mathcal{K}_{3}\mbox{.}
\end{gather*}

As every $S\in\mathcal{D}_{3}$ contains some $S'\in\mathcal{K}$,
by double counting we get 
\[
\left|\mathcal{K}\right|\geq\frac{\frac{\alpha_{1}}{3}{n \choose k_{1}}}{{n-k_{2} \choose k_{1}-k_{2}}}\geq\alpha_{2}{n \choose k_{2}}
\]
 for some $\alpha_{2}=\alpha_{2}\left(\alpha_{1},k_{1}\right)>0$.
Hence $\left|\mathcal{K}_{i}\right|\geq\frac{\alpha_{2}}{3}{n \choose k_{2}}$
for some $i\in\left\{ 1,2,3\right\} $.
\medskip

\noindent \textbf{Case 3.1. }$\left|\mathcal{K}_{1}\right|\geq\frac{\alpha_{2}}{3}{n \choose k_{2}}$.

If $S\in\mathcal{K}_{1}$, then in particular $\val\left(a_{S}-b_{i}\right)< \gamma_S = \val\left(b_{i_{S}}-b_{i}\right)$
for all $i\in S$, where $i_{S}$ is the last element of $S$ in the
 order induced from $[n]$. Hence $\rv\left(a_{S}-b_{i}\right)=\rv\left(a-b_{i_{S}}\right)$
for all $i\in S,i<i_{S}$. Then we can conclude exactly as in Case
1.2. 
\medskip

\noindent \textbf{Case 3.2.} $\left|\mathcal{K}_{2}\right|\geq\frac{\alpha_{2}}{3}{n \choose k_{2}}$.

Note that if $S\in\mathcal{K}_{2}$, then $\left|S\right|\leq1$.
Indeed, if $\val\left(a_{S}-b_{i}\right)> \gamma_S = \val\left(b_{i}-b_{j}\right)$
then $\val\left(b_{i}-b_{j}\right)=\val\left(a_{S}-b_{j}\right) > \gamma_S$, a contradiction.
Hence, assuming $k_{2}\geq2$, this case does not occur.
\medskip

\noindent \textbf{Case 3.3.} $\left|\mathcal{K}_{3}\right|\geq\frac{\alpha_{2}}{3}{n \choose k_{2}}$.

In this case for every $S\in\mathcal{K}_{3}$ we have $\val\left(a_{S}-b_{i}\right)=\val\left(b_{j}-b_{i}\right)=\gamma_{S}$
for all $i<j\in S$. Let $i_{S}$ denote the last element of $S$
in the order induced from $[n]$. It follows that $\rv\left(a_{S}-b_{i}\right)=\rv\left(a_{S}-b_{i_{S}}\right)+\rv\left(b_{i_{S}}-b_{i}\right)$
for all $i\in S,i<i_{S}$.

There are at most $n$ choices for $i_{S}$ when $S$ varies over $\mathcal{K}_{3}$. 
It follows that for some $i^{*} \in [n]$, the set 
$$\mathcal{L} := \left\{ S\in\Cons_{k_{2}-1}\left(\mathcal{F}\right):\exists S'\in\mathcal{K}_{3}\left(S=S'\setminus\left\{ i_{S'}\right\} \land i_{S'}=i^{*}\right)\right\} $$
has size at least $\frac{\left|\mathcal{K}_{3}\right|}{n}\geq\alpha_{3}{n \choose k_{2}-1}$
for some $\alpha_{3}=\alpha_{3}\left(\alpha_{2},k_{2}\right)>0$.

Let $\psi'\left(x';z, z'\right)$ be the $L_{\RV}$-formula 
$$\varphi\left(x'+z',z\right)\land\left(\val_{\rv}\left(x'\right)=\val_{\rv}\left(z'\right)\right)\land\WD\left(x',z'\right)$$
(see Remark \ref{rem: WD}), and let $\mathcal{F}'$ be the family 
\begin{gather*}
	 \left\{ \varphi\left(x'+\rv\left(b_{i^{*}}-b_{i}\right),c_{i}\right)\land\left(\val_{\rv}\left(x'\right)=\val\left(b_{i^{*}}-b_{i}\right)\right)\land\mbox{\ensuremath{\WD}}\left(x',\rv\left(b_{i^{*}}-b_{i}\right)\right):i \in [n]\right\} \\
	=\left\{ \psi'\left(x'; c_{i},d_{i}\right):i \in [n]\right\} 
\end{gather*}
 of definable subsets of $\RV(\M)$, with $d_{i} := \rv\left(b_{i^{*}}-b_{i}\right)$.
Note that $\mathcal{L}\subseteq\Cons_{k_{2}-1}\left(\mathcal{F}'\right)$,
hence $\left|\Cons_{k_{2}-1}\left(\mathcal{F}'\right)\right|\geq\alpha_{3}{n  \choose k_{2}-1}$.

As $\psi'\left(x';z,z'\right)$ is an $L_{\RV}$-formula, by assumption
it has FHP. If $k_{2}-1$ is larger than the fractional Helly number
of $\psi'$, then there is some $\beta_{3}=\beta_{3}\left(\psi',\alpha_{3}\right)>0$
and $I\subseteq\left[n\right]$ such that $\left|I\right|\geq\beta_{3}n$
and $\left\{ \psi'\left(x';c_{i},d_{i}\right):i\in I\right\} $ is
consistent, realized by some $e\in\RV(\M)$.

Let $a\in K(\M)$ be such that $\rv\left(a-b_{i^{*}}\right)=e$.
Then in particular $\val\left(b_{i^{*}}-b_{i}\right)=\val_{\rv}\left(e\right)=\val\left(a-b_{i^{*}}\right)$
for all $i\in I$ and, as $\WD\left(e,\rv\left(b_{i^{*}}-b_{i}\right)\right)$
holds, we have $\rv\left(a-b_{i}\right)=\rv\left(a-b_{i^{*}}\right)+\rv\left(b_{i^{*}}-b_{i}\right)=e+\rv\left(b_{i^{*}}-b_{i}\right)$
for all $i\in I$. Hence $a\models\left\{ \varphi\left(\rv\left(x-b_{i}\right),c_{i}\right):i\in I\right\} $,
and we have found a consistent subfamily $\mathcal{F}_{0}$ of $\mathcal{F}$
of size $\geq\beta_{3}n$.

~

Finally, choosing $k$ sufficiently large --- satisfying all of the assumptions
in each of the cases, which can be done depending only on the formula $\varphi$ --- and taking $\beta := \min\left\{ \beta_{i}\right\}$ with $\beta_i$ appearing in each of the cases, we have demonstrated the claim.
\end{proof}
\begin{prop}\label{prop: FHP for K from RV}
Assume $\RV_{\ast}$ satisfies $\FHP$. Then every partitioned formula $\varphi(x,y) \in L$ satisfies $\FHP$.
\end{prop}

\begin{proof}
By Fact \ref{lem: Basic operations preserving FH}(3) is suffices to show $\FHP$ for every partitioned formula $\varphi(x,y) \in L$ with $|x|=1$. If $x$ is a singleton of sort $\RV$, this holds by Remark \ref{rem: add unif in Flenner}(1) and the assumption that $\RV_{\ast}$ satisfies $\FHP$. So assume that $x$ is of sort $K$. By Fact \ref{fac: Flenner's cell decomposition} and Remark \ref{rem: add unif in Flenner}(2), it suffices to show that $\FHP$ holds  for every partitioned formula of the form 
$$\psi\left(x;y_{0}, \ldots, y_{m-1}, z\right)=\varphi\left(\rv\left(x-y_{0}\right),\ldots,\rv\left(x-y_{m-1}\right),z\right),$$
where $x$ is a singleton of sort $K$ and $\varphi\left(x_{0},\ldots,x_{m-1},z\right)$ is an $L_{\RV}$-formula (with $x_i$ singletons and $z$ an arbitrary finite tuple). We prove this by induction on $m$, the
base case $m=1$ given by Lemma \ref{lem: reduction to RV, key lemma}. 

Let $k = k(\varphi)$ be sufficiently large (to be determined in the proof) and let $\alpha >0$
be arbitrary. Let  $n \in \mathbb{N}$, $\left(b_{i, t}:i \in [n], t < m \right)$ with $b_{i,t} \in K(\M)$ and $\left(c_{i} : i \in [n]\right)$ with $c_i \in \RV(\M)^{z}$ be arbitrary (without loss of generality $n$ is arbitrarily large with respect to $k$ and $\alpha$), and let $\mathcal{F} := \left\{ \psi\left(x;b_{i,0}, \ldots, b_{i,m-1}, c_{i}\right):i<n\right\} $
be such that $\left|\Cons_{k}\left(\mathcal{F}\right)\right|\geq\alpha{n \choose k}$.
Let $\mathcal{C} := \Cons_{k}\left(\mathcal{F}\right)$. For each $S\in\mathcal{C}$,
fix some $a_{S}\in K(\M)$ realizing $\left\{ \psi\left(x;b_{i}, c_{i}\right):i\in S\right\} $, where $b_{i} := (b_{i,0}, \ldots, b_{i,m-1})$. Let 

\begin{gather*}
	S_{1} := \left\{ i\in S:\val\left(a_{S}-b_{i,0}\right)<\val\left(b_{i,m-1}-b_{i,0}\right)\right\} = \\ \left\{ i\in S:\val\left(a_{S}-b_{i,m-1}\right)<\val\left(b_{i,m-1}-b_{i,0}\right)\right\} \mbox{,}\\
	S_{2} := \left\{ i\in S:\val\left(a_{S}-b_{i,0}\right)>\val\left(b_{i,m-1}-b_{i,0}\right)\right\}, \\
	S_{3} := \left\{ i\in S:\val\left(a_{S}-b_{i,m-1}\right)>\val\left(b_{i,m-1}-b_{i,0}\right)\right\} \textrm{ and}\\
	S_{4} := \left\{ i\in S:\val\left(a_{S}-b_{i,0}\right)=\val\left(a_{S}-b_{i,m-1}\right)=\val\left(b_{i,m-1}-b_{i,0}\right)\right\} \mbox{.}
\end{gather*}

Then $S$ is the disjoint union of $S_{1},S_{2},S_{3},S_{4}$. Let $k_{2}:= \left\lfloor \frac{k}{4}\right\rfloor $, 
\[
\mathcal{D}_{i} := \left\{ S'\in\Cons_{k_{2}}\left(\mathcal{F}\right):\exists S\in\mathcal{C}\left(S'\subseteq S_{i}\right)\right\} 
\]
 for $i\in\left\{ 1,2,3,4\right\} $, and $\mathcal{D} := \bigcup_{i=1}^{4}\mathcal{D}_{i}$.
As every $S\in\mathcal{C}$ contains some $S'\in\mathcal{D}$, by
double counting we get 
\[
\left|\mathcal{D}\right|\geq\frac{\alpha{n \choose k}}{{n-k_{2} \choose k-k_{2}}}\geq\alpha_{2}{n \choose k_{2}}
\]
for some $\alpha_{2}=\alpha_{2}\left(\alpha,k\right)>0$. Hence $\left|\mathcal{D}_{i}\right|\geq\frac{\alpha_{2}}{4}{n \choose k_{2}}$
for some $i\in\left\{ 1,2,3,4\right\} $.

\noindent \textbf{Case 1.} $\left|\mathcal{D}_{1}\right|\geq\frac{\alpha_{2}}{4}{n \choose k_{2}}$.

For every $S\in\mathcal{D}_{1}$, we have $\val\left(a_{S}-b_{i,0}\right)<\val\left(b_{i,m-1}-b_{i,0}\right)$, which implies $\rv\left(a_{S}-b_{i,0}\right)=\rv\left(a_{S}-b_{i,m-1}\right)$, 
for all $i\in S$. We let 
\begin{gather*}
	\psi'\left(x; y_{0}, \ldots, y_{m-2}, z, z'\right) := \\
	\varphi\left(\rv\left(x-y_{0}\right),\ldots,\rv\left(x-y_{m-2}\right),\rv\left(x-y_{0}\right),z\right)\land\left(\val\left(x-y_{0}\right)<\val_{\rv}\left(z'\right)\right).
\end{gather*} 

\noindent \textbf{Case 2. }$\left|\mathcal{D}_{2}\right|\geq\frac{\alpha_{2}}{4}{n \choose k_{2}}$.

For every $S\in\mathcal{D}_{2}$ we have $\val\left(a_{S}-b_{i,0}\right)>\val\left(b_{i,m-1}-b_{i,0}\right)$, so  $\rv\left(a_{S}-b_{i,m-1}\right)=\rv\left(b_{i,m-1}-b_{i,0}\right)$, for all $i\in S$. 
We define 
\begin{gather*}
	\psi'\left(x; y_{0}, \ldots, y_{m-2}, z, z'\right) := \\
	\varphi\left(\rv\left(x-y_{0}\right),\ldots,\rv\left(x-y_{m-2}\right),z',z\right)\land\left(\val\left(x-y_{0}\right)>\val_{\rv}\left(z'\right)\right).
\end{gather*}

\noindent \textbf{Case 3.} $\left|\mathcal{D}_{3}\right|\geq\frac{\alpha_{2}}{4}{n \choose k_{2}}$.

Symmetric to Case 2. We define 
\begin{gather*}
	\psi'\left(x; y_{1}, \ldots, y_{m-1}, z, z'\right) := \\
	\varphi\left(z',\rv\left(x-y_{1}\right),\ldots,\rv\left(x-y_{m-1}\right),z\right)\land\left(\val\left(x-y_{m-1}\right)>\val_{\rv}\left(z'\right)\right).
\end{gather*}

\noindent \textbf{Case 4.} $\left|\mathcal{D}_{4}\right|\geq\frac{\alpha_{2}}{4}{n \choose k_{2}}$.

For every $S\in\mathcal{D}_{4}$ and $i \in S$ we have 
\begin{gather*}
	\val\left(a_{S}-b_{i,0}\right)=\val\left(a_{S}-b_{i,m-1}\right)=\val\left(b_{i,m-1}-b_{i,0}\right), \textrm{ hence}\\
	\rv\left(a_{S}-b_{i,0}\right)=\rv\left(a_{S}-b_{i,m-1}\right)+\rv\left(b_{i,m-1}-b_{i,0}\right).
\end{gather*}
We define  (see Remark \ref{rem: WD}) 
\begin{gather*}
	\psi'\left(x; y_{1}, \ldots, y_{m-1}, z, z'\right) := \\
	\varphi\left(\rv\left(x-y_{m-1}\right)+z',\rv\left(x-y_{1}\right),\ldots,\rv\left(x-y_{m-1}\right),z\right)\land\WD\left(\rv\left(x-y_{m-1}\right),z'\right).
\end{gather*}

~

We let $d_{i} := \rv\left(b_{i,m-1}-b_{i,0}\right)$ for $i \in [n]$. In each of the cases, we consider the family $\mathcal{F}' := \left\{ \psi'\left(x; b_{i}, c_{i}, d_{i}\right):i \in [n]\right\} $. By the inductive assumption for $m-1$, the formula $\psi'$ has FHP.  Assuming that $k_{2}$
is larger than the fractional Helly number of $\psi'$, it follows
that there is some $\beta=\beta\left(\frac{\alpha_{2}}{4},\psi',k_{2}\right)>0$
and $I\subseteq\left[n\right]$ such that $\left|I\right|\geq\beta n$
and $\left\{ \psi'\left(x; b_{i}, c_{i}, d_{i} \right):i\in I\right\} $
is consistent, realized by some $a\in K(\M)$. In each of the
four cases, it is easy to see that then also $a\models\left\{ \psi\left(x; b_{i,0}, \ldots, b_{i,m-1}, c_{i}\right):i\in I\right\} $.
For example, this holds in Case 4 as by the choice of $a$ we have that $\WD\left(\rv\left(a-b_{i,m-1}\right),\rv\left(b_{i,m-1}-b_{i,0}\right)\right)$
holds for all $i\in I$, and so $\rv\left(a-b_{i,0}\right)=\rv\left(a-b_{i,m-1}\right)+\rv\left(b_{i,m-1}-b_{i,0}\right)$
for all $i\in I$.

Finally, choosing $k$ sufficiently large --- satisfying all of the assumptions
in each of the cases, which can be done depending only on the formula $\varphi$ --- and taking the minimal $\beta$ appearing in each of the cases, we have demonstrated the claim.
\end{proof}

\begin{prop}\label{prop: FHP for RV from k and Gamma}
	Assume $k$ and $\Gamma$ satisfy $\FHP$. Then $\RV_{\ast}$ also satisfies $\FHP$.
\end{prop}
\begin{proof}
 Working in $\M$, which is in particular $\aleph_1$-saturated, we have that the short exact sequence $1\to k^{\times}\to\RV \overset{\val_{\rv}}{\to}\Gamma\to 0$ splits (see e.g.~\cite[Corollary 3.3.38]{aschenbrenner2017asymptotic}), so $\RV$  is the direct sum of $k^{\ast}$ and $\Gamma$.
We have that $\RV_{\ast}$ is bi-interpretable with this short exact sequence (see \cite[Lemma 5.17]{zbMATH07542897}), and this short exact sequence expanded by a function symbol for a right inverse map for $\val_{\rv}$ is bi-interpretable with a structure consisting of two disjoint sorts given by $k$ and $\Gamma$ with their induced structure (and no additional structure).
It is then clear, using Fact \ref{lem: Basic operations preserving FH}(3), that if both $k$ and $\Gamma$ are $\FHP$, then such a structure is also $\FHP$. 
 By Lemma \ref{lem: Basic operations preserving FH}, any theory interpretable in an $\FHP$ theory is also $\FHP$, so we conclude that $\RV_{\ast}$ is $\FHP$.
\end{proof}

Combining Proposition \ref{prop: FHP for RV from k and Gamma} and Proposition \ref{prop: FHP for K from RV} finishes the proof of Theorem \ref{thm: AKE for FHP}.

\subsection{FHP in ultraproducts of $\mathbb{Q}_p$ and explicit bounds}

Burden in henselian valued fields is studied extensively  \cite{chernikov2010indiscernible, chernikov2014theories, chernikov2014valued, chernikov2016henselian, sinclair2022burden, touchard2023burden}. In particular we have:
\begin{fact}\label{fac: bdn in hens val fields}
	\cite{chernikov2016henselian} If $M = (K, \RV, \rv)$ is a henselian valued field of equicharacteristic $0$, $\bdn(k) = \bdn(\Gamma) \leq 1$ and $k^{\times} / (k^{\times})^p$ is finite for all prime $p$, then $\bdn(M) = 1$. In particular, $\bdn(M) = 1$ for any ultraproduct $K := \prod_{i \in \mathbb{N}} \mathbb{Q}_{p_i}/\mathcal{U}$ with $p_i$ prime.
 \end{fact}
\begin{rem}
	If $k$ is inp-minimal, then, by the proof of \cite[Corollary 4.6]{chernikov2015groups}, $k^{\times} / (k^{\times})^p$ is finite for all but at most one prime $p$. It is open if it has to be finite for all $p$ \cite[Problem 25]{chernikov2016henselian} (true when $k$ is also NIP by \cite{johnson2018canonical}).
\end{rem}

\begin{cor}\label{cor: FHP in ultraprod Qp}
Let $K$ be $\prod_{i \in \mathbb{N}} \mathbb{Q}_{p_i} / \mathcal{U}$ or $\prod_{i \in \mathbb{N}} \mathbb{F}_{p_i}((t)) / \mathcal{U}$   for some prime $p_i$ and a non-principal ultrafilter $\mathcal{U}$ on $\mathbb{N}$. Let  $M = (K, \RV, \rv)$. Then every partitioned formula $\varphi(x,y) \in L$  with $|x| \leq d$ satisfies $\FHP_{2^d}$.
\end{cor}
\begin{proof}

	The residue field $k = \prod_{i \in \mathbb{N}} \mathbb{F}_{p_i} / \mathcal{U} $ is pseudo-finite, hence satisfies $\FHP$ by Theorem \ref{thm: FHP in MS meas} (and Fact \ref{fac: psf are MS-meas}). The (ordered abelian) value group $\Gamma$ is NIP by \cite{gurevich1984theory}, hence satisfies $\FHP$ by Fact \ref{fac: NIP implies FHP}. Hence $M = (K, \RV, \rv)$ satisfies $
	\FHP$ by Theorem \ref{thm: AKE for FHP}.

	By Fact \ref{fac: bdn in hens val fields} we have $\bdn(M) = 1$. By Corollary \ref{cor: FHPk bounded by burden} we know that, assuming $M$ satisfies $\FHP$, any formula $\varphi(x,y) \in L$ satisfies $\FHP_{\bdn(M^x) + 1}$. By Fact \ref{fac: submult of burden} this implies that any formula $\varphi(x,y) \in L$ satisfies $\FHP_{(\bdn(M)+1)^{|x|}}$ (and if $M$ is NIP, satisfies $\FHP_{|x| \cdot \bdn(M) + 1}$). Hence any formula $\varphi(x,y) \in L$ satisfies $\FHP_{2^{|x|}}$.
\end{proof}

\noindent We note that each individual field $\mathbb{F}_{p}\left(\left(t\right)\right)$ has $\TP_{2}$ by \cite[Corollary 3.3]{chernikov2015groups}, hence cannot satisfy $\FHP$.

\begin{conjecture}\label{conj: FHP in ultra p-adics}
	The bound $2^d$ in Corollary \ref{cor: FHP in ultraprod Qp} can be improved to $d+1$.
\end{conjecture}


\begin{problem}\label{prob: FHP in valued fields}
\begin{enumerate}
\item Which (ordered, valued, etc.) fields satisfy $\FHP$? 
\item Do bounded PAC, pRC, pPC fields satisfy $\FHP$? They are known to be $\NTP_2$ (see \cite{montenegro2025pseudo}).
	\item Does $\VFA_0$ satisfy $\FHP$? Again, it is known to be $\NTP_2$ by \cite{chernikov2014valued}. In particular, does $\ACFA$ satisfy $\FHP$?
\end{enumerate}
	
\end{problem}

\section{ULCFS and counting partial types over finite sets}\label{sec: ULCFS and counting over finite sets}

\subsection{ULCFS}

Let $T$ be a complete first-order theory with infinite models in a language $L$. First we consider finitary counterparts of some notions, properties and facts around forking and dividing. We assume some familiarity with the basic definitions and properties of dividing (see e.g.~\cite{chernikov2012forking}).
\begin{defn}
\label{def: (Delta,n,k)-dividing of a formula}Let $\Delta$ be a
set of formulas and $n,k\in\mathbb{N}$. Given a formula $\varphi(x,b)$ with $\varphi(x,y) \in L$ and $b$ a tuple in $\M$, we say that it 
\emph{$\left(\Delta,n,k\right)$-divides} over a set $C \subseteq \M$ if there
is a $\Delta\left(C\right)$-indiscernible sequence $\left(b_{i}:i<n\right)$ in $\M$ 
so that $b_{0}=b$ and the set of formulas $\left\{ \varphi\left(x,b_{i}\right):i<n\right\} $
is $k$-inconsistent.
\end{defn}

Note that, by compactness, $\varphi\left(x,b\right)$ is $k$-dividing over
$C$ if and only if it $\left(\Delta,n,k\right)$-dividing over $C$
for all finite sets of formulas $\Delta\subseteq L$ and all $n\in\omega$.
\begin{defn}
Let $\pi\left(x\right)$ be a partial type over a set of parameters
$B$, and let $C\subseteq B$.

\begin{enumerate}
\item For a formula $\varphi\left(x,y\right)\in L$, we say that $\pi$ is \emph{$\left(\varphi,\Delta,n,k\right)$-dividing}
over $C$ if there is an instance $\varphi\left(x,b\right)\in\pi$ such
that $\varphi\left(x,b\right)$ is $\left(\Delta,n,k\right)$-dividing over
$C$.
\item We say that $\pi$ \emph{internally $\left(\varphi,\Delta,n,k\right)$-divides}
over $C$ if it $\left(\varphi,\Delta,n,k\right)$-divides over $C$,
and moreover we can choose a $\Delta\left(C\right)$-indiscernible
sequence $\left(b_{i}:i<n\right)$ witnessing $\left(\Delta,n,k\right)$-dividing over $C$ as in Definition
\ref{def: (Delta,n,k)-dividing of a formula} \emph{inside $B=\dom\left(\pi\right)$}.
\end{enumerate}
\end{defn}

\begin{rem}
Assume that $p\in S_x\left(M\right)$ is a type over a model $M$ and
$C$ is a subset of $M$. If $M$ is $|C|^+$-saturated,
then dividing over $C$ is equivalent to internal dividing over $C$
for $p$ (given an arbitrary sequence witnessing dividing, by saturation
we can find a copy of it inside $M$). For example, if $C \subseteq M$ is finite,
then for any finite $\Delta,n,k$ we have $p\in S\left(M\right)$ is 
$\left(\Delta,n,k\right)$-dividing over $C$ if and only if it internally
$\left(\Delta,n,k\right)$-dividing over $C$. We will be mostly interested
in types over finite sets here, where these two notions may  be
different.
\end{rem}

\begin{defn}
\label{def: ULCFS}

\begin{enumerate}
\item We say that dividing in $T$ satisfies the \emph{Uniform Local Character
over Finite Sets} (or that \emph{$T$ satisfies $\ULCFS$}) if every partitioned 
formula $\varphi\left(x,y\right)\in L$ satisfies ULCFS, i.e.~for every $k\in\omega$ there is some
finite set of formulas $\Delta \subseteq L$ and some $n\in\omega$ so that: 
for every \emph{finite} set $B$ and every $p\left(x\right)\in S_{\varphi}\left(B\right)$
there is some $C\subseteq B$ with $\left|C\right|\leq n$ such that
$p$ does not \emph{internally} $\left(\varphi,\Delta,n,k\right)$-divide
over $C$.
\item We say that $T$ satisfies \emph{strong $\ULCFS$} if in (1) instead
of internal non-dividing we require non-dividing.
\end{enumerate}
\end{defn}

Obviously, strong ULCFS implies ULCFS.

\begin{problem}
	Is it true that (strong) $\ULCFS$ holds in $T$ only assuming that Definition \ref{def: ULCFS} is satisfied for $k=2$? Only assuming that  Definition \ref{def: ULCFS} holds for all formulas $\varphi(x,y)$ with $|x|=1$?
\end{problem}

\subsection{NIP implies strong ULCFS}\label{sec: UCLFS in NIP}

The following is the so-called \emph{UDTFS} property (i.e.~\emph{Uniform Definability of Types over Finite Sets}). It was conjectured by Laskowski, and, after some partial results \cite{johnson2010compression, guingona2012uniform}, established in \cite{chernikov2013externally, chernikov2015externally} for formulas in NIP theories (and more recently for NIP formulas in arbitrary theories \cite{eshel2021uniform}):

\begin{fact}\label{fac: UDTFS}
	Let $\varphi(x,y) \in \mathcal{L}$ be a partitioned NIP formula. Then there exists a formula $\theta(y;\bar{y}) \in \mathcal{L}$, with $\bar{y} = (y_1, \ldots, y_d)$ for some $d \in \mathbb{N}$, satisfying the following: for every finite $A \subseteq \mathbb{M}^y$ with $|A| \geq 2$ and every $b \in \mathbb{M}^x$, there exists some tuple $c  \in A^d$ such that for every $a \in A$ we have $\models \varphi(b,a) \iff \models \theta(a, c)$.
\end{fact}

It is immediate that if $\varphi(x,y)$ satisfies UDTFS then it also satisfies ULCFS. Indeed, given $\theta(y, \bar{y})$ as in Fact \ref{fac: UDTFS} and finite set $A$, every $p(x) \in S_{\varphi}(A)$ clearly does not internally $(\varphi, \theta, n, k)$-divide over a subset of $A$ size at most $d$ for any $n,k \in \mathbb{N}$, namely over $c$ where $c$ is so that $\theta(y,c)$ is a definition for $p$ ($\varphi(x,a) \in p \Rightarrow \models \theta(a,c)$, so if $a_i \in A$ for $i < n$ with $a = a_0$ and $(a_i : i < n)$ is $\theta$-indiscernible over $c$, then also $\models \theta(a_i,c)$ for all $i<n$, hence $\varphi(x,a_i) \in p$).

We show that strong ULCFS follows from the more explicit form of UDTFS established for NIP formulas in \cite{eshel2021uniform, bays2024density}:

\begin{prop}\label{prop: NIP implies ULCFS}
If a partitioned formula $\varphi(x,y)$ is NIP then it satisfies strong $\ULCFS$.
\end{prop}

\begin{proof}

Assume that $\varphi\left(x,y\right)$ is NIP. By \cite{eshel2021uniform, bays2024density} there exists $k,m \in \mathbb{N}$ depending only on $\varphi$ (more precisely, only on the VC-dimension $\VC\left(\varphi\right)$) 
satisfying the following: for every finite set $A\subseteq\mathbb{M}^{y}$
and every $p\in S_{\varphi}\left(A\right)$, there exist some types
$p_{1},\ldots,p_{m}\in S_{\varphi}\left(A\right)$ so that:
\begin{enumerate}
\item each $p_{i}$ is \emph{$k$-compressible}, i.e.~for some $A_{i}\subseteq A$
with $\left|A_{i}\right|\leq k$ we have $p_{i}\restriction_{A_{i}}\vdash p_{i}$,
\item and $p$ is the rounded average of the $p_{i}$'s, i.e.~for any $a\in A$
and $t\in\left\{ 0,1\right\} $, $\varphi^{t}\left(x,a\right)\in p\iff\left|\left\{ i\in\left[m\right]:\varphi^{t}\left(x,a\right)\in p_{i}\right\} >\frac{m}{2}\right|$.
\end{enumerate}
Let $\Delta\left(y,\left(y_{j}:j\in\left[k\right]\right)\right)$
be the finite set of formulas 
\[
\left\{ \forall x\left(\left(\bigwedge_{j\in\left[k\right]}\varphi^{t_{j}}\left(x,y_{j}\right)\right)\rightarrow\varphi^{t}\left(x,y\right)\right):\left(t_{j}:j\in\left[k\right]\right)\in\left\{ 0,1\right\} ^{\left[k\right]},t\in\left\{ 0,1\right\} \right\} .
\]

Now let $A\subseteq\mathbb{\mathbb{M}}^{y}$ be an arbitrary finite
set and $b\in\mathbb{M}^{x}$, and let $p:=\tp_{\varphi}\left(b/A\right)$.
Let $A_{i}\subseteq A$ and $p_{i}\in S_{\varphi}\left(A\right)$
for $i\in\left[m\right]$ satisfy (1) and (2) above with respect to
$p$, and let $A':=\bigcup_{i\in\left[m\right]}A_{i}$. Then $\left|A'\right|\leq km$.
We claim that $p$ does not $\left(\varphi,\Delta,n,k\right)$-divide
over $A'$ for any $n, k\geq2$.

For each $i\in\left[m\right]$, fix an arbitrary enumeration $A_{i}=\left(a_{i,j}:j\in\left[k\right]\right)$.
Let $a\in A$ be arbitrary, and let $t\in\left\{ 0,1\right\} $ be
such that $\varphi^{t}\left(x,a\right)\in p$. Then by (2) there exists
some $s\subseteq\left[m\right]$ with $\left|s\right|>\frac{m}{2}$
so that $\varphi^{t}\left(x,a\right)\in p_{i}$ for each $i\in s$.
Fix an arbitrary $i\in s$, and for $j\in\left[k\right]$ let $t_{j}$
be such that $\varphi^{t_{j}}\left(x,a_{i,j}\right)\in p_{i}$. Then
by (1) we have 
\[
\models\forall x\left(\left(\bigwedge_{j\in\left[k\right]}\varphi^{t_{j}}\left(x,a_{i,j}\right)\right)\rightarrow\varphi^{t}\left(x,a\right)\right).
\]

Let $\left(a^{\ell}:\ell\in\left[n\right]\right)$ be an arbitrary
$\Delta\left(A'\right)$-indiscernible sequence in $\mathbb{M}^{y}$
with $a^{1}=a$. In particular $\left(a^{\ell}:\ell\in\left[n\right]\right)$
is $\Delta\left(a_{i,j}:j\in\left[k\right]\right)$-indiscernible,
hence by the choice of $\Delta$ for any $\ell\in\left[n\right]$
we also have 
\[
\models\forall x\left(\left(\bigwedge_{j\in\left[k\right]}\varphi^{t_{j}}\left(x,a_{i,j}\right)\right)\rightarrow\varphi^{t}\left(x,a^{\ell}\right)\right).
\]

Let $b_{i}\in\mathbb{M}^{x}$ be an arbitrary realization of $p_{i}$,
in particular $b_{i}\models\bigwedge_{j\in\left[k\right]}\varphi^{t_{j}}\left(x,a_{i,j}\right)$,
but then also $b_{i}\models\bigwedge_{\ell\in\left[n\right]}\varphi^{t}\left(x,a^{\ell}\right)$.
\end{proof}

\subsection{Simple theories satisfy strong ULCFS}\label{sec: simple implies ULCFS}

From the definition of $\NTP$ (Definition \ref{def: NTP2}) and compactness we have:
\begin{lem}
\label{lem: bounded rank in a simple theory}Assume that $\varphi\left(x,y\right)$
is $\NTP$  and $k \in \mathbb{N}$. Then there is some $K\in\omega$ such that there
do not exist $\left(a_{\eta}:\eta\in K^{K}\right)$ in $\M^y$ such that $\left\{ \varphi\left(x,a_{\eta\restriction i}\right):i<K\right\} $
is consistent for every $\eta\in K^{K}$ and $\left\{ \varphi\left(x,a_{\eta i}\right):i<K\right\} $
is $k$-inconsistent for every $\eta\in K^{<K}$.
\end{lem}

\begin{prop}
\label{lem: every type does not fork over a small set} If a partitioned formula $\varphi\left(x,y\right)$ is $\NTP$ then it satisfies strong ULCFS.
\end{prop}

\begin{proof}

This follows by compactness from the local character of dividing in simple theories, but we give an explicit finitization of the argument.

Let $\varphi(x,y)$ and $k \in \omega$ be given.

For any $\psi\left(x,y\right) \in L$, finite $\Delta \subseteq L$, $n \in \omega$ and tuple of variables $z$, the set 
$$\Xi_{\psi,\Delta,n}\left(y;z\right) := \left\{ \left(y,z\right):\psi\left(x,y\right)\mbox{is \ensuremath{\left(\Delta,n, k\right)}-dividing over }z\right\} $$
is $\emptyset$-definable. Now let $K$ be as given by Lemma \ref{lem: bounded rank in a simple theory} for $\varphi(x,y)$.

Let $\Delta_{K} := \emptyset$, and by reverse induction on $i = K-1, \ldots, 0$ we define 
\begin{gather*}
	\Delta_i(y_0, \ldots, y_{i-1}; y_i) := \\
	\exists y_{i+1} \ldots y_{K-1} \left( \exists x \bigwedge_{j=0}^{K-1} \varphi(x,y_j) \  \land  \ \bigwedge_{j = i+1}^{K-1} \Xi_{\varphi,\Delta_j,K}\left(y_j; y_0, \ldots, y_{j-1} \right)\right).
\end{gather*}
%
%
%
%
%
% ...
and let $\Delta := \bigcup_{i<K}\Delta_{i}$.

Let $A$ be a set and $p\left(x\right) \in S_{\varphi}(A)$. We try to choose
by induction on $i<K$ elements $a_{i}$ from $A$ such that $\varphi\left(x,a_{i+1}\right)\in p$ and $\varphi\left(x,a_{i+1}\right)$
is $\left(\Delta,K, k\right)$-dividing over $A_{i} := \left(a_{j}:j<i\right)$.

Suppose that we succeded. Now we build a $K^{<K}$-tree for $\varphi\left(x,y\right)$ as follows.
As $\Delta_i \subseteq \Delta$, in particular $\varphi\left(x,a_{i+1}\right)$ is $\left(\Delta_{i},K, k\right)$-dividing 
over $\left(a_{j}:j<i\right)$ for all $i<K$. Let $b_{\emptyset} := a_{0}$.
Then let $b_{0} := a_{1}$. By assumption, $\varphi\left(x,a_{1}\right)$ is 
$\left(\Delta,K, k\right)$-dividing over $a_{0}$,  let $\left(b_{j}:j<K\right)$
be a $\Delta_1(a_{0})$-indiscernible sequence  witnessing this.
By the choice of $\Delta_1$ it follows that for each $j<K$ there are 
some $\left(b_{j 0^{i}}\right)_{i<(K-1)}$ such that each $\varphi\left(x,b_{j0^{i}}\right)$ is 
$\left(\Delta_{i},K, k\right)$-dividing over $\left(b_{j0^{h}}:h<i\right)$
and such that $\left\{ \varphi\left(x,b_{j0^{i}}\right):i<K-1\right\} $
is consistent. This gives the first level of the tree. Continuing in the same way we end up with a tree $\left(b_{\eta}:\eta\in K^{<K}\right)$
contradicting the choice of $K$. 

Thus we had to get stuck at some $i<K$, which means that $p$ does
not $\left(\varphi, \Delta,K, k\right)$-divide over $\left(a_{j}:j<i\right)$, 
as required.
\end{proof}

\subsection{ULCFS, resilience, $\NTP_2$}\label{sec: resilience}

\begin{defn}\label{def: resilience}
\cite[Definition 4.8]{yaacov2014independence} A theory  $T$ is \emph{resilient} if every formula $\varphi\left(x,y\right)\in L$ is \emph{resilient}, i.e.~there do not exist an indiscernible sequence $\left(a_{i}:i\in\mathbb{Z}\right)$ and an $\left\{ a_{i}:i\in\mathbb{Z}\setminus\left\{ 0\right\} \right\} $-indiscernible
sequence $\left(b_{i}:i\in\omega\right)$ so that $a_{0}=b_{0}$,
$\left\{ \varphi\left(x,a_{i}\right):i\in\mathbb{Z}\right\} $ is consistent
and $\left\{ \varphi\left(x,b_{i}\right):i\in\omega\right\} $ is inconsistent (hence $k$-inconsistent for some $k \in \omega$ by compactness and indiscernibility).
\end{defn}

\begin{fact}
	If $T$ is either simple or NIP, then it is resilient. And if $T$ is resilient, then it is $\NTP_2$ \cite[Proposition 4.11]{yaacov2014independence}. It is open if $\NTP_2$ is equivalent to  resilience \cite[Question 4.14]{yaacov2014independence}. 
\end{fact}

\begin{problem}
Assume that $T$ is not resilient. Is it possible to find a witness to it as in Definition \ref{def: resilience}, such that moreover $\left\{ \varphi\left(x,b_{i}\right):i\in\omega\right\} $
is $2$-inconsistent?
\end{problem}

This following lemma and corollary appeared first in a preliminary version of \cite{chernikov2014valued}, but were removed in the final version due to an error in its intended application.

\begin{lem}
\label{lem: resilience equivalent to rotation}The following are equivalent for a theory $T$.
\begin{enumerate}
\item Every formula $\varphi\left(x,y\right)$ with $\left|x\right|\leq n$
is resilient.
\item For every small set $D$, every $c$ with $\left|c\right|\leq n$ and every
sequence $\bar{a}=\left(a_{i}\right)_{i\in\mathbb{Q}}$ indiscernible
over $cD$, if $\bar{b}=\left(b_{i}\right)_{i\in\omega}$ is indiscernible
over $a_{\neq0}D$ and $b_{0}=a_{0}$, then there exists $\bar{b}'$ indiscernible
over $ca_{\neq0}D$ such that $\bar{b}'\equiv_{\bar{a}D}\bar{b}$.
\item For every small set $D$, every $c$ with $\left|c\right|\leq n$ and every
sequence $\bar{a}=\left(a_{i}\right)_{i\in\mathbb{Q}}$ indiscernible
over $cD$, if $\bar{b}=\left(b_{i}\right)_{i\in\omega}$ is indiscernible
over $a_{\neq0}D$ and $b_{0}=a_{0}$, then there exists $\bar{b}'$ indiscernible
over $ca_{\neq0}D$ such that $\bar{b}'\equiv_{a_{0}D}\bar{b}$.
\end{enumerate}
\end{lem}
\begin{proof}
(1) implies (2): Let $\varepsilon>0$ from $\mathbb{Q}$ be arbitrary. Consider the type $p_{\varepsilon}\left(x,a_{0}\right)=\tp\left(c/a_{0}a_{\notin\left(-\varepsilon,\varepsilon\right)}D\right)$.
By indiscernibility of $\bar{a}$ over $cD$ it follows that $\bigcup_{i\in\left(-\varepsilon,\varepsilon\right)}p_{\varepsilon}\left(x,a_{i}\right)$
is consistent (realised by $c$). We claim that $\bigcup_{i\in\mathbb{\omega}}p_{\varepsilon}\left(x,b_{i}\right)$
is consistent. Otherwise there is some $\varphi\left(x,a_{0}a'd\right)\in p_{\varepsilon}\left(x,a_{0}\right)$
with $a'\subseteq a_{\notin\left(-\varepsilon,\varepsilon\right)}$
and $d\subseteq D$ such that $\left\{ \varphi\left(x,b_{i}a'd\right)\right\} _{i\in\omega}$
is inconsistent. Let $a_{i}'=a_{i}a'd$ and $b_{i}'=b_{i}a'd$. Then
$\bar{a}'=\left(a_{i}'\right)_{i\in\left(-\varepsilon,\varepsilon\right)}$
is indiscernible, $\bar{b}'=\left(b_{i}'\right)_{i\in\omega}$ is
indiscernible over $a_{\neq0}'$, $a_{0}'=b_{0}'$ and they witness
that $\varphi\left(x,y\right)$ is not resilient --- a contradiction.

As $\varepsilon>0$ was arbitrary, it follows by compactness that,
letting $p\left(x,a_{0}\right)=\tp\left(c/a_{0}a_{\neq0}D\right)=\bigcup_{\varepsilon>0}p_{\varepsilon}\left(x,a_{0}\right)$,
we have that $\bigcup_{i\in\mathbb{\omega}}p\left(x,b_{i}\right)$
is consistent. But this means that there is $\bar{b}'\equiv_{\bar{a}D}\bar{b}$
which is indiscernible over $ca_{\neq0}D$ (see \cite[Lemma 3.5(1)]{chernikov2014valued}).

(2) implies (3) is obvious.

(3) implies (1): Assume that $\varphi\left(x,y\right)$ is not resilient, where $\left|x\right|\leq n$, 
and let $\bar{a}=\left(a_{i}\right)_{i\in\mathbb{Q}},\bar{b}=\left(b_{i}\right)_{i\in\omega}$
witness this. Then by Ramsey, indiscernibility and compactness we may find
$c\models\left\{ \varphi\left(x,a_{i}\right)\right\} _{i\in\mathbb{Q}}$
such that in addition $\bar{a}$ is indiscernible over $c$. But as
$c\models\varphi\left(x,b_{0}\right)$ and $\left\{ \varphi\left(x,b_{i}\right)\right\} _{i\in\omega}$
is inconsistent, there is no way to make it $c$-indiscernible keeping
the type of $\bar{b}$ even just over the first element.\end{proof}

\begin{cor}\label{C:Resilience-One-Variable}
If $T$ is not resilient, then there is some $\varphi\left(x,y\right)$
with $\left|x\right|=1$ which is not resilient.\end{cor}
\begin{proof}
We show by induction that if $T$ satisfies \ref{lem: resilience equivalent to rotation}(2)
for $n-1\geq1$, then it satisfies it for $n$. So let $c=\left(c_{i}\right)_{i<n}$,
$\bar{a}$ and $\bar{b}$ be given, with $\bar{a}$ indiscernible
over $c$ and $\bar{b}$ indiscernible over $a_{\neq0}$. Then by
the induction hypothesis and Lemma \ref{lem: resilience equivalent to rotation}
with $D=\emptyset$ we find $\bar{b}'\equiv_{\bar{a}}\bar{b}$ and
indiscernible over $c_{<n-1}a_{\neq0}$. Again by the inductive assumption
and Lemma \ref{lem: resilience equivalent to rotation} with $D=c_{<n-1}$
we find a $c_{n-1}c_{<n-1}a_{\neq0}$-indiscernible $\bar{b}''\equiv_{\bar{a}}\bar{b}'\equiv_{\bar{a}}\bar{b}$
--- as wanted.
\end{proof}

\begin{prop}\label{prop: str ULCFS implies resilient}
\begin{enumerate}
\item If $\varphi(x,y)$ satisfies ULCFS, then $\varphi(x,y)$ is $\NTP_2$. In particular, if $T$ satisfies ULCFS for $1$-types then $T$ is $\NTP_2$.
	\item If $\varphi(x,y)$ satisfies strong ULCFS then $\varphi(x,y)$ is resilient. In particular, if $T$ satisfies strong UCLFS for $1$-types then $T$ is resilient.
\end{enumerate}
	
\end{prop}
\begin{proof}

(1) Assume $\varphi(x,y)$ has $k$-$\TP_2$ for some $k \in \omega$, witnessed by a strongly indiscernible array $(a_{i,j})_{i,j \in \omega}$ (i.e.~the rows $\bar{a}_i := (a_{i,j} : j \in \omega), i \in \omega$ are mutually indiscernible, and the sequence of rows $(\bar{a}_i : i \in \omega)$ is indiscernible, see \cite[Definition 1.1]{chernikov2014theories}). Let $\Delta, n$ be arbitrary. Let $A := \{a_{i,j} : i,j < n+1\}$, $b \models \{ \varphi(x,a_{0,j}) : j < n+1\}$ and $p := \tp_{\varphi}(b/A)$. Let $C \subseteq A$ be arbitrary with $|C| \leq n$. Then there is some $i < n+1$ so that $\{a_{i,j} : j < n+1\} \cap C = \emptyset$, $(a_{i,j} : j < n+1)$ is indiscernible over $C$ (by strong indiscernibility of the array) and $\{\varphi(x,a_{i,j}) : j < n+1\} $ is $k$-inconsistent. This shows that $p$ is internally $\left(\varphi, \Delta, n, k  \right)$-dividing over $C$. For the ``in particular'' part, if $T$ is not $\NTP_2$, there is a formula $\varphi(x,y)$ with $|x|=1$  with $\TP_2$ (by Fact \ref{fac: 1-var 2-incons for TP_2}(2)).

	(2) Assume $\varphi(x,y)$ is not resilient (if $T$ is not resilient, by Corollary \ref{C:Resilience-One-Variable}  there is some non-resilient $\varphi\left(x,y\right)\in L$ with $|x|=1$), then there exist some indiscernible sequence $\bar{b}=\left(b_{i}:i\in\mathbb{Z}\right)$
and some $\left(b_{i}:i\in\mathbb{Z}\setminus\left\{ 0\right\} \right)$-indiscernible
sequence $\bar{b}'=\left(b'_{i}:i\in\omega\right)$ with $b_{0}=b_{0}'$
so that $\left\{ \varphi\left(x,b_{i}\right):i\in\mathbb{Z}\right\} $
is consistent and $\left\{ \varphi\left(x,b_{i}'\right):i\in\omega\right\} $
is $k$-inconsistent for some $k\in\omega$. Let
now $\Delta$ and $n\in\omega$ be arbitrary. Let $B_{n} := \left(b_{i}:0\leq i<n+1\right)$
and let $p_{n}\in S_{\varphi}\left(B_{n}\right)$ be arbitrary so that we have 
$\left\{ \varphi\left(x,b_{i}\right):0\leq i<n+1\right\} \subseteq p_{n}$.
We claim that $p_{n}$ is $\left(\varphi,\Delta,n,k\right)$-dividing over
every subset $C$ of $B_{n}$ with $\left|C\right|\leq n$. Indeed, given such $C$, 
there is some $0\leq j<n+1$ with $b_{j}\in B_{n}\setminus C$. By
indiscernibility of $\bar{b}$, taking an automorphism $\sigma$ of
$\mathbb{M}$ sending $\left(\ldots,b_{-1},b_{0},b_{1},\ldots\right)$
to $\left(\ldots,b_{j-1},b_{j},b_{j+1},\ldots\right)$ we consider
the sequence $\bar{b}'' := \sigma\left(\bar{b}'\right)$. In particular
it is indiscernible over $\left\{ b_{i}:i\in\mathbb{Z}\setminus\left\{ j\right\} \right\} $
and witnesses that $\varphi\left(x,b_{j}\right)$ is $\left(\Delta,n,k\right)$-dividing over $C$.
\end{proof}

\begin{rem}
	We note a weak converse to Proposition \ref{prop: str ULCFS implies resilient}(2): given $\varphi(x,y)$, if there is $k \in \omega$ so that for all $\Delta, n$ there is a finite set $A$ with $|A| = n+1$ and $p \in S_{\varphi}(A)$ so that $p$ is  $(\varphi, \Delta, n, k)$-dividing over every $C \subseteq A$ with $|C| \geq n$, then $\varphi(x,y)$ is not resilient. Indeed, in this case we can find a witness to $\varphi$ being not resilient by Ramsey and compactness.
\end{rem}

The above observations motivate the following conjecture:
\begin{conjecture}\label{conj: NTP2 implies ULCFS}
	If $\varphi(x,y)$ is $\NTP_2$ (or even $T$ is $\NTP_2$), then $\varphi(x,y)$ satisfies ULCFS. If $\varphi(x,y)$ is resilient (or even $T$ is resilient), then $\varphi(x,y)$ satisfies strong ULCFS.
\end{conjecture}

\subsection{Counting partial types over finite sets}\label{sec: fin type count}

In this section we consider some finitary versions of the two-cardinal function $f_{T}(\kappa, \lambda)$ counting partial types from Section \ref{sec: two card count part types}, which also generalizes the shatter function from VC-theory.
In the same way as UDTFS is a strengthening of the Sauer-Shelah lemma, we view ULCFS as a strengthening of this type-counting function being polynomially bounded.

\begin{defn}\label{def: f_phi(k,l) and poly bdd}
Let $\varphi\left(x,y\right)$ be a partitioned formula.

\begin{enumerate}
\item For a set $A \subseteq \M^y$ and $k\in\omega$, we let $S_{\varphi,k}\left(A\right)$
be the set of \emph{positive} partial $\varphi$-types of size $\leq k$ over $A$,
i.e.~consistent sets of formulas of the form $\{ \varphi\left(x,a_{i} \right) : i < k\}$
with $a_{0},\ldots,a_{k-1}\in A$. 
\item For $p,q \in S_{\varphi,k}\left(A\right)$ and $m \in \omega$, we say that $p$ and $q$ are \emph{$m$-inconsistent} if  there are some $p_0 \subseteq p, q_0 \subseteq q$ with $|p_0|,|q_0| \leq m$ so that $p_0 \cup q_0$ is inconsistent.

\item For $m \leq k\leq l < \omega$, we define 
\begin{gather*}
	f_{\varphi}\left(m, k,l\right) := 
\max\Big\{ \left|S\right|: A \subseteq \M^y, S\subseteq S_{\varphi,k}\left(A\right),\left|A\right|=l, \\
p\neq q\in S \Rightarrow p\left(x\right), q\left(x\right)\mbox{ are $m$-inconsistent}\Big\} \mbox{.}
\end{gather*}

\item We let $f_{\varphi}(k,l) := f_{\varphi}(1,k,l)$ and $f'_{\varphi}(k,l) := f_{\varphi}(k,k,l)$.

\item We will say that $f_{\varphi}\left(k,l\right)$ (or $f'_{\varphi}(k,l)$) is \emph{polynomially
bounded} if there is $d$ such that: for every fixed $k$ there is some $C_k \in \omega$ with $f_{\varphi}\left(k,l\right) \leq C_k l^{d}$
for all $l$. 
\item For a formula $\varphi\left(x,y\right)$ and $t \in \omega$, we let 
$$\left( \bigwedge_{t}\varphi \right)(x;y_0, \ldots, y_{t-1}) := \bigwedge_{i < t} \varphi(x,y_i).$$
\end{enumerate}
\end{defn}

\begin{rem}\label{rem: positive types}
	While the restriction to positive types is natural for our considerations here, we note that for any $\varphi(x,y) \in L$, taking $\psi(x;y_1,y_2) := \varphi(x,y_1) \land \neg \varphi(x,y_2) \in L$,  each $\varphi$-type of size $ \leq k$ over a set $A \subseteq \M$ is equivalent to a positive $\psi$-type of size $\leq k$ over $A \times A$ (more generally, see e.g.~\cite[Lemma II.2.1]{MR513226}) --- hence many type counting questions transfer to general types.
\end{rem}

\begin{rem}\label{rem: fin type count easy rem}
	\begin{enumerate}
		\item $f_{\varphi}\left(m,k,l\right)\leq f_{\varphi}\left(m',k',l'\right)$ and $f_{ \varphi}(m, k, l) \leq f_{\bigwedge_2 \varphi}(m, k, l)$
for all $m' \geq m, k'\geq k,l'\geq l$; in particular $f_{\varphi}(k,l) \leq f'_{\varphi}(k,l) \leq \sum_{i \leq k}{l \choose i} \leq \left(\frac{e}{k}\right)^k l^k$ for all $k,l$.
		\item $f_{\bigwedge_2 \varphi}(m, k, l) \leq f_{\varphi}(2m, 2k, 2l)$  and $f_{ \varphi}(m, k, l) \leq f_{\bigwedge_m \varphi}(1, k^m, l^m)$ for all $m,k,l$.
	\end{enumerate}
\end{rem}
\begin{proof}
(1) is immediate from the definition. For the first inequality in (2), let $A \subseteq \M^{y_1,y_2}$ with $|A| = l$ and $S\subseteq S_{\bigwedge_2 \varphi,k}\left(A\right)$ a set of pairwise $m$-inconsistent positive $\bigwedge_2 \varphi$-types. Let $A' \subseteq \M^y$ be the set $\pi_1(A) \cup \pi_2(A)$, where $\pi_i$ is the projection onto $i$th coordinate, then $|A'| \leq 2l$. And for $p \in S$, we let $p' \in S_{\varphi, 2k}(A')$ be defined via:  $\varphi(x,a) \in p'$ if and only if there is some $a'$ so that either $\varphi(x,a) \land \varphi(x,a') \in p$ or $\varphi(x,a') \land \varphi(x,a) \in p$. It follows that for any $p_1 \neq p_2 \in S$, $p'_1$ and $p'_2$ are $2m$-inconsistent. For the second inequality, given a $\varphi$-type $p$ of size $k$ over $A \subseteq \M^y$, consider the $\bigwedge_m\varphi$-type $(\bigwedge_m p)(x)$ of size $k^m$ over $A^m$ given by $\{ \bigwedge_{i < m} \varphi(x,a_i) : (a_1, \ldots, a_m) \in A^m, \varphi(x,a_i) \in p \textrm{ for all } i<m \}$ --- if $p_1, p_2$ are $m$-inconsistent, then $p'_1, p'_2$ are $1$-inconsistent.
\end{proof}

\begin{prop}
\label{prop: f poly bdd implies NTP2 and low}

\begin{enumerate}
\item Assume that $f_{\bigwedge_t \varphi}\left(k,l\right)$ is polynomially bounded for all $t \in \omega$, then
$\varphi\left(x,y\right)$ is $\NTP_{2}$. In particular, if $f_{\varphi}$ is polynomially bounded for all partitioned formulas $\varphi(x,y) \in L$ with $|x|=1$ in $T$, then $T$ is $\NTP_2$.
\item Assume that $f'_{\varphi}\left(k,l\right)$ is polynomially bounded. Then $f'_{\bigwedge_t \varphi}\left(k,l\right)$ is also polynomially bounded for all $t \in \omega$, and 
$\varphi\left(x,y\right)$ is low (see Definition \ref{def: low formula}). 
\end{enumerate}
\end{prop}

\begin{proof}
(1) Assume that $\varphi\left(x,y\right)$ has $\TP_{2}$. Then, by Fact \ref{fac: 1-var 2-incons for TP_2}(2), for some $t = t(\varphi)$ the formula $\psi(x; y_0, \ldots, y_{t-1}) := \bigwedge_{i < t} \varphi(x,y_i)$ has $2$-$\TP_2$, i.e.~there is an array $\left(a_{i,j}\right)_{i,j<\omega}$ such that $\left\{ \psi\left(x,a_{i,j}\right):j<\omega\right\} $
is $2$-inconsistent for every $i<\omega$ and $\left\{ \psi\left(x,a_{i,f\left(i\right)}\right):i<\omega\right\} $
is consistent for every $f:\omega\to\omega$. 

Let $k$ and $l$ be
arbitrary, and let $A := \left\{ a_{i,j}\right\} _{i<k,j<l}$. For $f:k\to l$,
let $p_{f}(x) := \left\{ \psi\left(x,a_{i,f\left(i\right)}\right):i<k\right\} $.
Then any two types in the set $S := \left\{ p_{f}:f\in l^{k}\right\} \subseteq S_{\psi,k}\left(A\right)$ are $1$-inconsistent. It follows that $\left|S\right|\geq l^{k}$
while $\left|A\right|=k\times l$. That is, $\left|S\right|\geq c\left|A\right|^{k}$
with $c=\frac{1}{k^{k}}$. It follows that for $k$ fixed and $l\gg k$, $f_{\psi} \left(k,l\right)$
grows faster than any polynomial of degree $\leq k-1$. Thus $f_{\psi} = f_{\bigwedge_t \varphi}$ 
is not polynomially bounded.

(2) By Remark \ref{rem: fin type count easy rem}, $f'_{\bigwedge_2\varphi}\left(k,l\right) \leq f'_{\varphi}\left(2k,2l\right)$, so if the latter is polynomially bounded by $C_k l^d$ then the former is also polynomially bounded by $2^d C_{2k} l^d$. 

Assume $\varphi\left(x,y\right)$ is not low. Then for an arbitrary
large $k\in\omega$ we can find an indiscernible sequence $\left(a_{i}:i\in\omega\right)$ in $\M^y$
so that $\left\{ \varphi\left(x,a_{i}\right):i\in\omega\right\} $
is $k$-consistent, but $\left(k+1\right)$-inconsistent. But then
for every $k \leq l \in\omega$, 
$$\left\{ \{\varphi(x,a_{i_{t_j}}) : j < k\} : t_0 < \ldots < t_{k-1} < l  \right\}$$
is a family of $\geq {l \choose k}$ pairwise-inconsistent positive $\varphi$-types of size $k$ over $\left(a_{i}:i<l\right)$. Hence $f'_{\varphi}(k,l) \geq \frac{1}{k^k} l^k$, so $f'_{\varphi}$ is not polynomially bounded.
\end{proof}

\begin{rem}
	Since there are non-low (super-)simple theories (see \cite{casanovas1998supersimple}), and simplicity of $T$ implies $f_{\varphi}$ is polynomially bounded for all formulas (by Propositions \ref{lem: every type does not fork over a small set} and \ref{prop: ULCF implies f is poly bdd}), we see that $f'_{\varphi}(k,l)$ is polynomially bounded is a strictly stronger condition then $f_{\varphi}(k,l)$ is polynomially bounded.
\end{rem}

Next we observe that, closing under conjunctions, ULCFS implies polynomial boundedness of $f_{\varphi}$.

\begin{prop}\label{prop: ULCF implies f is poly bdd}
For a partitioned formula $\varphi(x,y)$, if $\bigwedge_{2} \varphi$ satisfies ULCFS, then $f_{\varphi}(k,l)$ 
 is polynomially bounded.
\end{prop}

\begin{proof}
 As $\bigwedge_{2} \varphi$ satisfies ULCFS, there is some
finite set of formulas $\Delta \subseteq L$ and some $d\in\omega$ so that: 
for every \emph{finite} set $B \subseteq \M^{y_1, y_2}$ and every $p'\left(x\right)\in S_{\bigwedge_2 \varphi}\left(B\right)$
there is some $C\subseteq B$ with $\left|C\right|\leq d$ such that
$p'$ does not internally $\left(\bigwedge_2 \varphi,\Delta,d,2\right)$-divide
over $C$.

Let now $A \subseteq \M^y$ be a finite set of size $l$, and $S \subseteq S_{\varphi,k}\left(A\right)$. Given $p \in S$, we consider $\bigwedge_2 p \in S_{\bigwedge_2 \varphi,k^2}\left(A^2\right)$ as in Remark \ref{rem: fin type count easy rem}. By the previous paragraph, there is some $A_p \subseteq A^2$ with $|A_p| \leq d$ so that $\bigwedge_2 p$ does not internally $\left(\bigwedge_2 \varphi,\Delta,d,2\right)$-divide over $A_p$.

\begin{claim}
	There is some $K = K (d, k, \Delta) \in \omega$ satisfying the following. For any $C \subseteq A^2$ with $|C| = d$, if  $(p_i(x) : i < n)$ with $p_i \in S_{\varphi,k}\left(A\right)$ are pairwise $1$-inconsistent and $\bigwedge_2 p_i$ does not internally $\left(\bigwedge_2 \varphi,\Delta,d,2\right)$-divide over $C$ for all $i<n$, then $n \leq K$.
\end{claim}
\begin{proof}
For $i < n$, let $p_i = \{\varphi(x,a_{i,t}) : t < k \}$ for some $a_{i,t} \in A$, and let $\bar{a}_i := (a_{i,t} : t < k)$. By Ramsey's theorem, using pairwise $1$-inconsistency of the $p_i$'s, there is some $K = K (d, k, \Delta)$ satisfying the following. If $n > K$ then there
is a subsequence $\left(\bar{a}_{i_{j}}:j<d\right)$ of $\left(\bar{a}_{i}:i< n \right)$
which is $\Delta$-indiscernible over $C$ and some $\alpha,\beta<k$ so that  for all $j<j'<d$ we have: $\varphi\left(x,a_{i_{j},\alpha}\right)\in p_{i_{j}},\varphi\left(x,a_{i_{j'},\beta}\right)\in p_{i_{j'}}$
and $\left\{\varphi\left(x,a_{i_{j},\alpha}\right),\varphi\left(x,a_{i_{j'},\beta}\right)\right\}$ is inconsistent. This shows that $\varphi(x,a_{i_j, \alpha}) \land \varphi(x,a_{i_j, \beta}) \in \bigwedge_2 \left( p_{i_{0}} \right)$ internally $\left(\Delta,d,2\right)$-divides
over $C$ --- a contradiction.
\end{proof}

As there are at most $(|A^2|)^d \leq  l^{2d}$ choices for $A_p$, it follows from the claim that for any $k$, any family $S \subseteq S_{\varphi,k}\left(A\right)$ of pairwise $1$-inconsistent types has size at most $K l^{2d}$, with $K$ depending only on $\varphi$ and $k$ --- hence $f_{\varphi}(k,l)$ is polynomially bounded.
\end{proof}

\begin{conjecture}\label{conj: NTP2 iff counting types}
	The following are equivalent for a partitioned formula $\varphi(x,y)$:
	\begin{enumerate}
		\item $\bigwedge_t \varphi$ is $\NTP_2$ for all $t \in \omega$;
		\item $f_{\bigwedge_t \varphi}(k,l)$ is polynomially bounded for all $t \in \omega$.
	\end{enumerate}
	Or at least: $f_{\varphi}(k,l)$ is polynomially bounded for every formula $\varphi(x,y)$ in an $\NTP_2$ theory $T$.
\end{conjecture}

We note that it is necessary to consider conjunctions in Conjecture \ref{conj: NTP2 iff counting types}:
\begin{example}
Let $T$ the theory of an infinite triangle-free random graph. Then the edge relation $R\left(x,y\right)$
is $\NTP$ (hence satisfies strong ULCFS by Proposition \ref{lem: every type does not fork over a small set}), while the formula  $ R\left(x,y_{1}\right)\land R\left(x,y_{2}\right)$
has $2$-$\TP_{2}$ (see \cite[Example 3.13]{chernikov2014theories}) --- hence does not satisfy ULCFS by Proposition \ref{prop: str ULCFS implies resilient}, and $f_{\bigwedge_2 R}(k,l)$ is not polynomially bounded by (the proof of) Proposition \ref {prop: f poly bdd implies NTP2 and low}(1).
Note that $f_{R}(k,l) = f'_{R}(k,l)$ (since in a triangle-free random graph, two partial types $\{R(x,a_i) : i < m\}$ and $\{R(x,b_j) : j < n\}$ are inconsistent if and only if $\models R(a_i,b_j)$ for some $i,j$, in which case $\{R(x,a_i), R(x,b_j)\}$ is inconsistent), hence $f_{R}(k,l)$ is also not polynomially bounded by Proposition \ref{prop: f poly bdd implies NTP2 and low}(2).
\end{example}

Next we observe that at least $\TP_2$ is characterized by the maximality of the function $f_{\varphi}(k,l)$ (unlike in the infinitary version of the problem, see Corollary \ref{cor: counterex to Adler conj}). We use a classical result of Erd\H os:
\begin{fact}\cite{erdos1964extremal}
\label{fac: Zarankiewicz bound} Let $G=\left(V,E\right)$ be a uniform
$k$-hypergraph (i.e.~$E \subseteq {V \choose k}$) with $\left|V\right|=l$ such that it does not contain
a complete $k$-partite uniform $k$-hypergraph  $K_{d, \ldots, d}$ with each part of size $d$.
Then the number of hyper-edges for all $l \gg k, d$ satisfies $\left|E\right| \leq l^{k-\frac{1}{d^{k-1}}}$.
\end{fact}

\begin{prop}\label{prop: NTP2 iff power saving}
The following are equivalent.

\begin{enumerate}
\item The partitioned formula $\varphi\left(x,y\right)$ does not have $2$-$\TP_{2}$.
\item  The function $f_{\varphi}(k,l)$ has \emph{power saving}, i.e.~there exists some $\varepsilon = \varepsilon(\varphi) > 0$ and $k_0 \in \mathbb{N}$ so that: for each $k \geq k_0$  there is some $C = C(\varphi, k) \in \mathbb{N}$ so that $f_{\varphi}\left(k,l\right) \leq C l^{k-\varepsilon}$ for all $l$.
\end{enumerate}
\end{prop}

\begin{proof}
(2)$\Rightarrow$(1). The proof of Proposition \ref{prop: f poly bdd implies NTP2 and low}(1) shows
that if $\varphi\left(x,y\right)$ has $2$-$\TP_{2}$ then for any $k\in\omega$ and $l \geq k$, $f_{\varphi}\left(k,l\right)\geq\left(\frac{1}{k}\right)^{k}l^{k}$, which grows  with $l$ faster  than $C l^{k-\varepsilon}$ for any $C$ and $\varepsilon > 0$.

(1)$\Rightarrow$(2): If $\varphi\left(x,y\right)$ does not have $2$-$\TP_2$, 
then by compactness there exist some $1 < d\in\omega$ such that there
are no $(a_{i,j} : 1 \leq i,j  \leq d)$ so that $\{\varphi(x,a_{i,j}) : 1 \leq j \leq  d\}$ is $2$-inconsistent for all  $ 1 \leq i \leq d$ and $\{\varphi(x,a_{i,f(i)}) : 1 \leq i \leq  d\}$ is consistent for every $f:[d] \to [d]$.

Suppose that no $\varepsilon >0$ and $k_0$ satisfy (2). In particular, taking $\varepsilon := \frac{1}{d^{d-1}}$, there exists some $k \geq d$ so that: for any $C \in \mathbb{N}$ there is an arbitrarily large $l \in \mathbb{N}$ with $f_{\varphi}(k,l) > C l^{k - \varepsilon}$.  Let $C := 2$, and, as $\varepsilon<1$, choose $l$ sufficiently large so that $\left(\frac{e}{k-1}\right)^{k-1} l^{k-1} \leq l^{k-\varepsilon}$. 

Let $S\subseteq S_{\varphi,k}\left(A\right)$ with $|S| > C l^{k - \varepsilon}$ and $A \subseteq \M^y, \left|A\right|=l$ be a family of pairwise $1$-inconsistent types witnessing this. Note that the number of partial types $p \in S$ of size $< k$ is at most $\sum_{i \leq k-1}{l \choose i} \leq \left(\frac{e}{k-1}\right)^{k-1} l^{k-1} \leq l^{k-\varepsilon}$ by assumption on $l$.   

Hence there is $S' \subseteq S$ with $|S'| > l^{k - \varepsilon}$ so that all $p \in S'$ are of size \emph{exactly} $k$.

Consider the uniform $k$-hypergraph
$G$ with 
$$V\left(G\right) := A, E\left(G\right) := \left\{ \dom\left(p\right) : p\in S'\right\} \textrm{ (so }|E| \geq l^{k - \varepsilon} \textrm{)}.$$

Assume that for every $\bar{a} = \{a_{d+1}, \ldots, a_{k}\} \in {V \choose k-d}$ we have $|E_{\bar{a}}| \leq l^{d - \varepsilon}$, where $E_{\bar{a}} = \{\{a_1, \ldots, a_k\} \in {V \choose k} : \{a_1, \ldots, a_k\} \cup \bar{a} \in E \}$. As there are ${l \choose k-d} \leq l^{k-d}$ choices for $\bar{a}$, this would imply that $|E| \leq l^{k - \varepsilon}$ --- contradicting the assumption. Hence there is some $\bar{a} \in {V \choose k-d}$ so that the $d$-hypergraph $G_{\bar{a}} := (V, E_{\bar{a}})$ has $> l^{d-\varepsilon}$ hyperedges. By Fact \ref{fac: Zarankiewicz bound}, $G_{\bar{a}}$ contains a complete $d$-partite $d$-hypergraph $H$ 
with each part of size $d$. For $1 \leq i \leq d$, let $\bar{a}_{i}:=\left(a_{i,j}: 1 \leq j \leq d\right)$ with $a_{i,j} \in A$ enumerate
the $i$'s part of the partition of $H$.

As each hyperedge of $E$ corresponds to a domain of some $p\in S'$, it follows that 
$$\left\{ \varphi\left(x,a_{i,f\left(i\right)}\right): 1 \leq i \leq d\right\} \cup \{ \varphi(x, a_{d+1}), \ldots, \varphi(x,a_{k}) \} $$
is consistent for each $f: [d] \to [d]$. On the other hand, let $ 1 \leq i' \leq d$ and $1 \leq j \neq j' \leq d$ be arbitrary. As $H$ is complete, it follows from its
definition that 
\begin{gather*}
	\left\{ \varphi\left(x,a_{i,1}\right):1 \leq i \leq d,i\neq i'\right\} \cup\left\{ \varphi\left(x,a_{i',j}\right)\right\} \cup \{ \varphi(x, a_{d+1}), \ldots, \varphi(x,a_{k}) \}  \in S',\\
	\left\{ \varphi\left(x,a_{i,1}\right):1 \leq i \leq d,i\neq i'\right\} \cup\left\{ \varphi\left(x,a_{i',j'}\right)\right\} \cup \{ \varphi(x, a_{d+1}), \ldots, \varphi(x,a_{k}) \}  \in S'.
\end{gather*}
As any two types in $S'$ are $1$-inconsistent, the only possibility is that 
$$\{ \varphi\left(x,a_{i',j}\right), \varphi\left(x,a_{i',j'}\right)\} $$
 is  inconsistent. It follows that $\left\{ \varphi\left(x,a_{i,j}\right): 1 \leq j \leq d\right\} $
is $2$-inconsistent for every $1 \leq i \leq d$. But then $\left(a_{i,j}: 1 \leq i, j \leq d\right)$ contradicts the choice of $d$
\end{proof}

Hence we have the following weak partial result towards Conjecture \ref{conj: NTP2 iff counting types}: 
\begin{cor}
	The following are equivalent for a partitioned formula $\varphi(x,y)$: 
	\begin{enumerate}
	\item $\bigwedge_t\varphi$ is $\NTP_2$ for all $t \in \mathbb{N}$;
		\item $f_{\bigwedge_t \varphi}(k,l)$ satisfies power saving for all $t \in \mathbb{N}$.
	\end{enumerate}
\end{cor}
\begin{proof}
	By Proposition \ref{prop: NTP2 iff power saving} and Fact \ref{fac: 1-var 2-incons for TP_2}(2).
\end{proof}

\begin{problem}
	\begin{enumerate}
		\item Is there an implication in either direction between ``$\varphi(x,y)$ is FHP'' and ``$f'_{\varphi}(k,l)$ is polynomially bounded''? At least at the level of the theory $T$ as opposed to individual formula?
		Matousek's theorem (Fact \ref{fact: FHP for counting}), combined with Remark \ref{rem: positive types}, shows that if there is some $C,d$ so that $f'_{\varphi}(l,l) \leq C l^d$ for all $l$ then $\varphi(x,y)$ is FHP --- so part of the question is if this assumption could be relaxed to polynomial boundedness of $f'_{\varphi}(k,l)$.
		\item  Does FHP imply that $f_{\varphi}(k,l)$ is polynomially bounded? By Proposition \ref{prop: FHP implies NTP2} this would follow from Conjecture \ref{conj: NTP2 iff counting types}.
	\end{enumerate}
\end{problem}

\begin{problem}
	We can also ask a more general version of Conjecture \ref{conj: NTP2 iff counting types}, corresponding to a finitary version of Conjecture 
	\ref{conj: Adler}: what are the possible functions $f_{\varphi}(k,l)$?
\end{problem}

\section{Counting partial types in the infinitary setting}\label{sec: two card count part types}

\begin{defn}
For a complete first-order theory $T$, $1 \leq n \in \omega$ and an infinite cardinal $\kappa$, we let $f^n_{T}\left(\kappa\right) := \sup \{|S_n(M)| : M \models T, |M| = \kappa\}$ and $f_{T}\left(\kappa\right) := \sup_{n \in \omega} f^n_{T}\left(\kappa\right)$.
\end{defn}

A celebrated result of Keisler \cite{keisler1974number, keisler1978stability} (see also \cite{keisler1976six}), refining earlier work of \cite{shelah1971stability} and Morley \cite{morley1965categoricity}, demonstrates that there are exactly six possibilities for $f_T$ when $T$ is a complete countable theory:

\begin{fact}
\label{fac: classification of usual spectra}
Let $T$ be a complete countable first-order theory with an infinite model (we will assume this throughout the section). Then for all
$\kappa\geq\aleph_{0}$ we have $f_T(\kappa) = f^1_T(\kappa)$ and:

\begin{enumerate}
\item $T$ is $\omega$-stable $\Rightarrow$ $f_{T}\left(\kappa\right)=\kappa$;
\item $T$ is superstable, not $\omega$-stable $\Rightarrow$ $f_{T}\left(\kappa\right)=\kappa+2^{\aleph_{0}}$;
\item $T$ is stable, not superstable $\Rightarrow$ $f_{T}\left(\kappa\right)=\kappa^{\aleph_{0}}$;
\item $T$ is not multiorder, unstable $\Rightarrow$ $f_{T}\left(\kappa\right)=\ded\kappa$;
\item $T$ is $\NIP$, multiorder $\Rightarrow$ $f_{T}\left(\kappa\right)=\left(\ded\kappa\right)^{\aleph_{0}}$;
\item $T$ is not $\NIP$ $\Rightarrow$ $f_{T}\left(\kappa\right)=2^{\kappa}$.
\end{enumerate}
\end{fact}

Here $\ded \kappa$ is the supremum over all cardinals $\lambda$ so that there exists a linear order of size $\kappa$ with $\lambda$ many Dedekind cuts (equivalently, the supremum of the sizes of linear orders with a dense subset of size at most $\kappa$), and for the definition of the multiorder property see Definition \ref{def: multiorder}. We always have $\kappa < \ded \kappa \leq 2^{\kappa}$, hence under the generalized continuum hypothesis (GCH), $\ded \kappa = 2^{\kappa}$ for all infinite $\kappa$, in which case the functions in (4)--(6) are equal. It is also consistent with ZFC that the functions in (5) and (6) are distinct \cite{mitchell1972aronszajn}; it is also consistent that the functions in (4) and (5) are distinct \cite{chernikov2016non}. But it is not known if it is consistent that all three functions in (4), (5), (6) are pairwise distinct simultaneously. Yet, in ZFC, for every infinite $\kappa$, we have $2^{\kappa} \leq \ded \ded \ded \ded \kappa$ \cite{chernikov2016number}. We refer to \cite{chernikov2016non, chernikov2016number} for a further discussion of this function.

Here we consider the following two-cardinal  refinement of $f_T$ introduced in \cite{shelah1980simple} in order to characterize simple theories, and further studied in \cite{casanovas1999number, lessmann2000counting, casanovas2003dividing} (where it is denoted as $\NT(\kappa, \lambda)$).
\begin{defn}\label{def: inf two-card counting function}
For a complete first-order theory $T$ and infinite cardinals $\kappa\leq\lambda$, we let $f^n_{T}\left(\kappa,\lambda\right)$
be the supremum of the cardinalities $\left|P\right|$, where $P$ is a
family of pairwise inconsistent partial $n$-types each of cardinality $\leq\kappa$, all 
over the same fixed set of parameters of size $\lambda$. We let $f_T(\kappa, \lambda) := \sup_{n \in \omega} f^n_{T}(\kappa, \lambda)$.
\end{defn}

The following is immediate from the definitions: 
\begin{rem}
\label{rem: obvious properties of f_T(k,l)}In any theory $T$ we
have:

\begin{enumerate}
\item $f^n_{T}\left(\lambda\right)=f^n_{T}\left(\lambda,\lambda\right)$ and $f^{n}_T(\kappa, \lambda) \leq f_T^{n+1}(\kappa, \lambda)$ for $\kappa \leq \lambda \geq |T|$ and $n \in \omega$;
\item $f_{T}^1\left(1,\lambda\right)\geq\lambda$ for all $\lambda\geq\aleph_{0}$;
\item $f_{T}^n\left(\kappa,\lambda\right)\leq f_{T}^n\left(\kappa',\lambda'\right)$
for any $\lambda'\geq\lambda,\kappa'\geq\kappa$;
\item $f_{T}\left(\kappa,\lambda\right)\leq\lambda^{\kappa}$ for all $\lambda\geq\kappa \geq \aleph_0$.
\end{enumerate}
\end{rem}

This two-cardinal function is known to characterize simplicity and supersimplicity of $T$ as follows (restricting to countable theories): 
\begin{fact}
\label{fac: known results about f_T(kappa,lambda)}

Let $T$ be a complete countable theory with infinite models, and $\kappa \leq \lambda$ infinite cardinals.

\begin{enumerate}
\item The following are equivalent:

\begin{enumerate}
\item $T$ is simple;
\item $f_{T}\left(\kappa,\lambda\right)\leq\lambda^{\aleph_{0}}+2^{\kappa}$
for all $\kappa$, $\lambda$;
\item $f^1_{T}\left(\kappa,\lambda\right)<\lambda^{\kappa}$ for some 
$\lambda=\lambda^{<\kappa}$.
\end{enumerate}
\item The following are equivalent:

\begin{enumerate}
\item $T$ is supersimple;
\item $f_{T}\left(\kappa,\lambda\right)\leq\lambda+2^{\kappa}$ for all
$\kappa,\lambda$;
\item $f_{T}^1 \left(\kappa,\lambda\right)<\lambda^{\aleph_{0}}$ for some
$\kappa,\lambda$. 
\end{enumerate}
%
%
%
\end{enumerate}
\end{fact}

\begin{proof}

(1) The equivalence of (a), (b) and (c) is from  \cite[Theorem 0.2]{shelah1980simple} (see also \cite[Theorem III.7.7]{MR513226} or \cite[Proposition 2.20]{kim1996simple}) and  \cite[Theorem 2.8]{casanovas1999number}. Item (2) is from \cite[Theorem 3.2]{casanovas1999number}. 
\end{proof}

Motivated by Facts \ref{fac: classification of usual spectra} and \ref{fac: known results about f_T(kappa,lambda)}, Adler made the following conjecture:

\begin{conjecture} \cite{AdlerBanff}\label{conj: Adler}
	\begin{enumerate}
		\item  There are only finitely many possible functions $f_{T}(\kappa, \lambda)$ when $T$ is a countable theory $T$.
		\item The property ``$T$ is $\NTP_2$'' can be detected from $f_{T}(\kappa, \lambda)$ (note that if $T$ is $\TP_2$, then $f_{T}(\kappa, \lambda)$ is maximal, i.e.~equal to $\lambda^{\kappa}$ for all $\aleph_0 \leq \kappa \leq  \lambda$).
	\end{enumerate}
\end{conjecture}

In what follows, we will refute Conjecture \ref{conj: Adler}(2) and confirm Conjecture \ref{conj: Adler}(1) assuming GCH. 

\begin{defn}
\begin{enumerate}
	\item As usual, by a tree we mean a partial order $(\mathcal{T}, \leq )$ so that for every $t \in \mathcal{T}$, $\{s \in \mathcal{T} : s \leq t\}$ is linearly well-ordered by $\leq$. By a branch in a tree, we mean a maximal linearly ordered subset of the nodes. The length of a branch is its order type. The level $l(t)$ of $t \in T$ is the order type of $\{s \in \mathcal{T} : s < t\}$. The height of $\mathcal{T}$ is $\sup\{l(t) +1 : t \in \mathcal{T}\}$.
	\item For two infinite cardinals $\kappa, \lambda$, we let $\lambda^{\kappa, \tr}$ denote the supremum of cardinals $\mu$ so that there exists a tree with $\leq \lambda$ many nodes and $\geq \mu$ many branches of length $ \leq \kappa$. 
\end{enumerate}
\end{defn}

\begin{rem}\label{rem: props of tree exp}

\begin{enumerate}
	\item Equivalently, $\lambda^{\kappa, \tr}$ can be defined as the supremum of cardinals $\mu$ so that there exists a tree of height $\leq \kappa$  and size $\leq \lambda$ with $\geq \mu$ branches (by restricting to the  subtree given by the union of all branches of length $\leq \kappa$ with the induced ordering).
	\item Note that $\lambda \leq \lambda^{\kappa, \tr} \leq \lambda^{\kappa}$ and if $\lambda = \lambda^{<\kappa}$, then $\lambda^{\kappa, \tr} = \lambda^{\kappa}$ (witnessed by the tree $\lambda^{< \kappa}$ with the usual tree ordering: for $\eta_i \in \lambda^{< \kappa}$, say $\eta_i \in \lambda^{\alpha_i}$ for ordinals $\alpha_1 \leq \alpha_2 < \kappa$, $\eta_1  \leq  \eta_2$ if $\eta_2 \restriction _{\alpha_1} = \eta_1$). 
	\item We have $\kappa^{\kappa, \tr} = \ded \kappa$, by (1) and \cite[Theorem 2.1(b)]{baumgartner1976almost}	(see also \cite[Section 6.1]{chernikov2016non} for a more detailed discussion). 
\end{enumerate}

\end{rem}

\begin{prop}
\label{prop:tree exponent lemma}Given infinite cardinals $\kappa \leq \lambda$, if $\lambda^{\kappa}>\lambda+2^{\kappa}$
then for some regular $\theta\leq\kappa$ there is a tree with $\leq\lambda$
nodes, $\theta$ levels and $\lambda^{\kappa}$ branches. 

In particular, if $\lambda^{\kappa} > 2^{\kappa}$ then $\lambda^{\kappa, \tr} = \lambda^{\kappa}$.
\end{prop}

\begin{proof}
Let $\mu := \min\left\{ \mu:\mu^{\kappa}\geq\lambda\right\} \leq \lambda$. Note that $\mu > 2^{\kappa}$ (as otherwise $\mu^{\kappa} \leq 2^{\kappa}$, hence $\lambda \leq 2^{\kappa}$, so $\lambda^{\kappa} \leq 2^{\kappa}$ ---   contradicting the assumption). Then 
$\mu^{\kappa}=\left(\mu^{\kappa}\right)^{\kappa}\geq\lambda^{\kappa}>\lambda+2^{\kappa}$. Also for a cardinal $\nu$, $\nu<\mu\Rightarrow \nu^{\kappa}  <\lambda$ by minimality,
hence also $\nu<\mu\Rightarrow \nu^{\kappa}<\mu$ (for finite $\nu$ as $\mu > 2^{\kappa}$, and for infinite $\nu$ as otherwise $\nu^{\kappa} = (\nu^{\kappa})^{\kappa} \geq \mu^{\kappa} > \lambda$). Let
$\theta := \cof\left(\mu\right)$ --- regular. We know that $\mu^{\kappa}>\mu$
(as $\mu\leq\lambda$ and $\mu^{\kappa}>\lambda$). On the other hand,
if $\kappa<\theta = \cof(\mu) \leq \mu$ then using the above observations we would  get $\mu^{\kappa}= \mu$ (see e.g.~\cite[Theorem 5.20(iii)(a)]{kunen2014set}),
a contradiction. So necessarily $\theta\leq\kappa$.

Consider the tree $\mu^{<\theta}$, it has $\theta$ levels and $\mu^{\theta}$-many branches of length $\theta$. It also has $\sum_{\beta < \theta \textrm{ ordinal}} \mu^{|\beta|}$ nodes. Note that for each $\beta < \theta = \cof(\mu)$ we have $\mu^{|\beta|} \leq  \mu \cdot \sup \{ \nu^{|\beta|} : \nu < \mu \textrm{ cardinal}\}$ (see e.g.~\cite[Lemma 1.6.15(c)]{holz2009introduction}), and $\nu^{|\beta|} \leq \nu^{\theta} \leq \nu^{\kappa} < \lambda$ by minimality of $\mu$. Hence $\mu^{|\beta|} \leq \mu \cdot \lambda = \lambda$, so the number of nodes  in the tree $\mu^{<\theta}$ is at most $\theta \cdot \lambda = \lambda$.

 And as $\kappa \geq \theta = \cof(\mu)$ we have (see e.g.~\cite[Lemma 1.6.15(b)]{holz2009introduction}) $\mu^{\kappa}=\left(\sup\left\{ \nu^{\kappa}:\nu<\mu \textrm{ cardinal}\right\} \right)^{\cof\left(\mu\right)=\theta} \leq \mu^{\theta} \leq\lambda^{\theta}\leq\lambda^{\kappa}\leq (\mu^{\kappa})^{\kappa} = \mu^{\kappa}$.
So $\mu^{\theta}=\lambda^{\kappa}$, and we are done. 

For the ``in particular'' part, note that if $\lambda^{\kappa} = \lambda$ then clearly $\lambda^{\kappa} = \lambda^{\kappa, \tr}$; and otherwise $\lambda^{\kappa} > \lambda + 2^{\kappa}$ and the result above applies.
\end{proof}

We recall the multi-order property from \cite{keisler1976six} (which is equivalent to the theory $T$ not admitting \emph{ird-patterns} of infinite height, or $\kappa_{\ird}(T) > \aleph_0$ in the sense of  \cite[Definition III.7.1]{MR513226}, see also \cite[Section 5]{adler2007strong}):

\begin{defn}\label{def: multiorder}
A theory $T$ has the \emph{multi-order property} if there exist formulas $\varphi_i(x,y_i) \in L$ for $i \in \omega$ with $x$ fixed and $y_i$ arbitrary tuples of variables, and tuples $(a_{i,j} : i,j \in \omega)$ so that: for every $f: \omega \to \omega$ there exists some $b_f$ with $\models \varphi(b_f, a_{i,j}) \iff j < f(i)$ for all $i,j \in \omega$.
\end{defn}

\begin{prop}\label{prop: multiorder 1 var}
	The multi-order property is witnessed in one variable, i.e.~if $T$ is multi-order then we can choose the formulas $\varphi_i(x,y_i)$ in Definition \ref{def: multiorder} with $|x|=1$.
\end{prop}
\begin{proof}
	This follows from \cite[Proposition 4.4]{simon2022linear} (see also \cite{guingona2015common} for the analogous sub-additivity statement in the case of finite cardinals). Namely, by Ramsey and compactness it is easy to see that $T$ has the multi-order property witnessed by some formulas $\varphi(x,y_i), i \in \omega$ with $|x|=n$ if and only if it is not the case that $\opD(x=x) < \aleph_0$ in the sense of \cite[Definition 4.1]{simon2022linear}, if and only if there is an $n$-tuple $a$ so that $\opD(a/\emptyset) < \aleph_0$ does not hold. But \cite[Proposition 4.4]{simon2022linear} implies that for any finite tuples $a,b$, if $\opD(a/\emptyset), \opD(b/\emptyset) < \aleph_0$, then $\opD(ab/\emptyset) < \aleph_0$ --- hence the multi-order property reduces to singletons.
\end{proof}

By analogy with the multi-order property, we define the multi-$\TP_1$ property: 
\begin{defn}
We say that $T$ has the \emph{multi-$\TP_1$} property if there exist formulas $\varphi_i(x,y_i) \in L$ with $x$ fixed and $y_i$ arbitrary finite tuples of variables and trees of tuples $(a^i_{\eta})_{\eta \in \omega^{<\omega}}$ for $i \in \omega$ in $\M$ such that:
\begin{enumerate}
\item for any choice of branches $(\eta_i : i \in \omega)$ with $\eta_i \in \omega^{\omega}$, the set of formulas $\bigcup_{i \in \omega} \{\varphi_i(x;a^i_{\eta_i | \alpha}) : \alpha < \omega\}$ is consistent,
\item for all $i \in \omega$ and all $\eta \perp \nu$ in $\omega^{<\omega}$, $\{\varphi_i(x;a^i_{\eta}),\varphi_i(x;a^i_{\nu})\}$ is inconsistent. 
\end{enumerate}

\end{defn}

\begin{prop}\label{prop: NIP multi-order implies multi TP1}
	If $T$ is NIP and multi-order then $T$ is mutli-$\TP_1$, witnessed by formulas $(\varphi_i(x,y_i))_{i \in \omega}$ with $|x|=1$.
\end{prop}
\begin{proof}

By Proposition \ref{prop: multiorder 1 var}, let $\varphi_i(x,y_i) \in L$ for $i \in \omega$ with $|x|=1$ witness the multi-order property. 

We will use an elaboration of the trick in Shelah's proof that NIP plus unstable implies the strict order property SOP. By Ramsey and compactness, we can choose tuples  $(a_{i,j} : i \in \omega, j \in \mathbb{Q})$ so that for every $f: \omega \to \mathbb{Q}$ there exists some $b_f$ with $\models \varphi(b_f, a_{i,j}) \iff j < f(i)$ for all $i  \in \omega, j \in \mathbb{Q}$; and so that the sequences $\{ \bar{a}_i : i \in \omega \}$ with $\bar{a}_i = (a_{i,j} : j \in \mathbb{Q})$ are mutually indiscernible.
		
	By induction on $\alpha < \omega$ we will choose formulas $\varphi'_\alpha(x,y'_\alpha) \in L$ and sequences $\bar{a}'_\alpha = (a'_{\alpha,j} : j \in \mathbb{Q})$ so that:
	\begin{enumerate}
		\item $(\bar{a}'_{i} : i \leq \alpha) \cup (\bar{a}_i : \alpha < i < \omega)$ are mutually indiscernible;
		\item the set of formulas $\pi'_{\leq \alpha}(x) \cup \pi_{>\alpha}(x)$ is consistent, where \begin{gather*}
			\pi'_{\leq \alpha}(x) = \bigcup_{i \leq \alpha} \left(\{ \varphi'_i(x,a'_{i,j}) : j \in \mathbb{Q}, j < 0 \} \cup \{ \neg  \varphi'_i(x,a'_{i,j}) : j \in \mathbb{Q}, j \geq 0 \} \right),\\
		\pi_{>\alpha}(x) = \bigcup_{\alpha < i < \omega} \left(\{  \varphi_i(x,a_{i,j}) : j \in \mathbb{Q}, j < 0 \} \cup \{\neg \varphi_i(x,a_{i,j}) : j \in \mathbb{Q}, j \geq 0 \} \right);
		\end{gather*}
		\item for all $i \leq \alpha$ and $j_2 > j_1 \in \mathbb{Q}$, $\varphi'_i(x,a'_{i, j_2}) \vdash \varphi'_i(x,a'_{i, j_1})$.
	\end{enumerate}

	Fix $\alpha \in \omega$, and assume we have already chosen $(\varphi'_i(x,y'_\alpha) : i < \alpha)$ and $(\bar{a}'_i : i < \alpha)$ satisfying (1)--(3).

	 Since the formula $\varphi_\alpha(x,y_{\alpha})$ is NIP, there is some $k_\alpha \in \omega$ and $\eta: k_\alpha \to \{0,1\}$ so that $\bigwedge_{j \in k_{\alpha}} \varphi_\alpha(x,a_{\alpha,j})^{\eta(j)}$ is inconsistent (where $\varphi^1 = \varphi$ and $\varphi^0 = \neg \varphi$). Starting with this formula, we change one by one instances of $\neg \varphi_{\alpha}(x, a_{\alpha,j}) \land \varphi_{\alpha}(x, a_{\alpha,j+1})$ to $ \varphi_{\alpha}(x, a_{\alpha,j}) \land \neg \varphi_{\alpha}(x, a_{\alpha,j+1})$. After finitely many steps, we arrive to a formula of the form $\bigwedge_{j \in \omega, 0 \leq j < \ell } \varphi_{\alpha} (x, a_{\alpha,j}) \land \bigwedge_{j \in \omega, \ell \leq j < k_{\alpha}}  \neg \varphi_{\alpha} (x, a_{\alpha,j})$ for some $0 \leq \ell < k_{\alpha}$, and by assumption ((2) for $\alpha-1$) and mutual indiscernibility this latter formula is consistent with  $\pi'_{\leq \alpha-1}(x) \cup \pi_{> \alpha}(x)$. Therefore there is a step in the process in which we pass from a formula inconsistent with $\pi'_{\leq \alpha-1}(x) \cup \pi_{> \alpha}(x)$ to a formula consistent with $\pi'_{\leq \alpha-1}(x) \cup \pi_{> \alpha}(x)$. Namely, there is some $j_0 \in k_{\alpha}$ and $\eta_0: k_{\alpha} \to \{0,1\}$ so that
	 \begin{gather*}
	 	\bigwedge_{0 \leq j < j_0 } \varphi_{\alpha}(x, a_{\alpha, j})^{\eta_0(j)} \land \varphi_{\alpha}(x,a_{\alpha, j_0}) \land \neg \varphi_{\alpha}(x,a_{\alpha, j_0+1}) \land \bigwedge_{ j_0+1 < j < k_{\alpha} } \varphi_{\alpha}(x, a_{\alpha, j})^{\eta_0(j)} 
	 \end{gather*} 
 is consistent with $\pi'_{\leq \alpha-1}(x) \cup \pi_{> \alpha}(x)$, but 
	 \begin{gather*}
	 	\bigwedge_{0 \leq j < j_0 } \varphi_{\alpha}(x, a_{\alpha, j})^{\eta_0(j)} \land \neg \varphi_{\alpha}(x,a_{\alpha, j_0}) \land \varphi_{\alpha}(x,a_{\alpha, j_0+1}) \land \bigwedge_{ j_0+1 < j < k_{\alpha} } \varphi_{\alpha}(x, a_{\alpha, j})^{\eta_0(j)} 
	 \end{gather*} 
	 is inconsistent with $\pi'_{\leq \alpha-1}(x) \cup \pi_{> \alpha}(x)$, hence already with some finite subtype.

 By compactness we can extend (without changing the already given elements)  our sequences $\bar{a}'_i = (a'_{i,j} : j \in \mathbb{Q})$  to $(a'_{i,j} : j \in J_{-} + \mathbb{Q} + J_{+})$ for $i < \alpha$ and $\bar{a}_i = (a_{i,j} : j \in \mathbb{Q})$  to $(a_{i,j} : j \in J_{-} + \mathbb{Q} +J_{+})$  for $\alpha \leq i < \omega$, maintaining mutual indiscernibility, where $J_{-} = \omega, J_{+}=\omega^{\ast}$.  
 
Then, by the above and mutual indiscernibility, there are some finite sets $J_{0} \subseteq J_{-}, J_{1} \subseteq J_{+}$ and $I \subseteq \omega, \alpha < I$ so that, taking 
\begin{gather*}
	\widetilde{\pi}'_{\leq \alpha-1}(x; \bar{a}_{\leq \alpha-1, J_0, J_1}) := \{\varphi'_{i}(x,a'_{i,j}) : i \leq  \alpha-1, j \in J_0 \} \cup\\ \{\neg \varphi'_{i}(x,a'_{i,j}) : i \leq  \alpha-1, j \in J_1 \},\\
	\widetilde{\pi}_{> \alpha}(x; \bar{a}_{I, J_0,J_1}) := \{\varphi_{i}(x,a_{i,j}) : i  \in I, j \in J_0 \} \cup \{\neg \varphi_{i}(x,a_{i,j}) : i \in I, j \in J_1 \},
\end{gather*}
the first formula is consistent with $\pi'_{\leq \alpha-1}(x) \cup \pi_{> \alpha}(x) \cup  \widetilde{\pi}'_{\leq \alpha-1}(x) \cup \widetilde{\pi}_{> \alpha}(x)$, while the second formula is inconsistent with $\widetilde{\pi}'_{\leq \alpha-1}(x) \cup \widetilde{\pi}_{> \alpha}(x)$. Fix an order preserving bijection $\gamma: \mathbb{Q} \to (j_0-1, j_0+2) \cap \mathbb{Q}$ with $\gamma(j_0)=j_0, \gamma(j_0 + 1) = j_0 + 1$, and for $j \in \mathbb{Q}$ we consider the finite tuple
\begin{gather*}
	a'_{\alpha,j} := a_{\alpha, \gamma(j)} \cup (a_{\alpha, j'}: j' \in \omega, 0 \leq j' < j_0 \lor j_0+1 < j < k_{\alpha}) \cup \\ (a'_{i,j'}: i \leq \alpha-1, j' \in J_0 \cup J_1) \cup (a_{i,j'}: i \in I, j' \in J_0 \cup J_1) 
\end{gather*}
and formula 
\begin{gather*}
	\varphi'_{\alpha}(x, y'_{\alpha}) :=  \varphi_{\alpha}(x,y_{\alpha, j_0}) \land  \bigwedge_{0 \leq j < j_0 } \varphi_{\alpha}(x, y_{\alpha, j})^{\eta_0(j)} \land \bigwedge_{ j_0+1 < j < k_{\alpha} } \varphi_{\alpha}(x, y_{\alpha, j})^{\eta_0(j)} \land  \\
	\widetilde{\pi}'_{\leq \alpha-1}(x; y_{\leq \alpha-1, J_0, J_1}) \land \widetilde{\pi}_{> \alpha}(x; y_{I, J_0,J_1}).
\end{gather*}
By construction, the sequences $\bar{a}'_{\alpha}, (\bar{a}'_i)_{i < \alpha}, (\bar{a}_i)_{\alpha < i < \omega}$ are still mutually indiscernible (so (1) holds). Using indiscernibility, for all $j_2 > j_1 \in \mathbb{Q}$ we have $\varphi'_{\alpha}(x,a'_{\alpha, j_2}) \vdash \varphi'_{\alpha}(x,a'_{\alpha, j_1})$ (so (3) holds), and as $\varphi'_{\alpha}(x,a'_{\alpha, j_0}) \land \neg \varphi'_{\alpha}(x,a'_{\alpha, j_0+1})$ is consistent with $\pi'_{\leq \alpha-1}(x) \cup \pi_{> \alpha}(x)$ implies that (2) also holds.
		
		We can think of $(\varphi'_i, \bar{a}'_i)_{i \in \omega}$ as witnessing that $T$ satisfies ``multi-SOP''. From this we can produce multi-$\TP_1$ similarly to the standard argument that SOP implies $\TP_1$.
		Namely, by induction on $|\eta|$, we can choose a tree of non-empty closed intervals $[i_{\eta}, i'_{\eta}]$ in $\mathbb{Q}$ with $\eta \in \omega^{< \omega}$ so that:
		\begin{enumerate}[(a)]
			\item for all $\nu \leq \eta$ in $\omega^{<\omega}$, $\emptyset \neq [i_{\eta}, i'_{\eta}] \subseteq  [i_{\nu}, i'_{\nu}]$;
\item for all  $\eta \perp \nu$ in $\omega^{<\omega}$, $[i_{\eta}, i'_{\eta}] \cap [i_{\nu}, i'_{\nu}] = \emptyset $.
		\end{enumerate}
		For $\alpha \in \omega$ and $\eta \in \omega^{< \omega}$, let $b^{\alpha}_{\eta} := (a'_{\alpha, i_{\eta}}, a'_{\alpha, i'_{\eta}})$, and let $\psi_{\alpha}(x, b^{\alpha}_{\eta}) := \varphi'_{\alpha}(x, a'_{\alpha, i_{\eta}}) \land \neg \varphi'_{\alpha}(x, a'_{\alpha, i'_{\eta}})$. It follows that for any  $\alpha \in \omega$ and $\eta \perp \nu$ in $\omega^{<\omega}$, $\{\psi_\alpha(x;b^\alpha_{\eta}),\psi_\alpha(x;b^\alpha_{\nu})\}$ is inconsistent by (3) and (b). And for every choice of branches $(\eta_\alpha : \alpha \in \omega)$ with $\eta_\alpha \in \omega^{\omega}$,  $\bigcup_{\alpha \in \omega} \{\psi_\alpha(x;b^\alpha_{\eta_\alpha | t}) : t < \omega\}$ is consistent by (a), (2)+(1) and compactness.
\end{proof}

The following is a generalization of \cite[Lemma 4, (ii)$\Rightarrow$(iii)]{keisler1978stability}.
\begin{prop}\label{prop: multi TP1 lower bound}
	\begin{enumerate}
		\item If $T$ has $\TP_1$ then $f^1_T(\kappa, \lambda) \geq \lambda^{\kappa,\tr}$ for all $\aleph_0 \leq \kappa \leq \lambda$.
		\item If $T$ has multi-$\TP_1$, witnessed by formulas $(\varphi_\alpha(x,y_{\alpha}))_{\alpha \in \omega}$ with $|x|=1$, then $f^1_T(\kappa, \lambda) \geq \left( \lambda^{\kappa,\tr} \right)^{\aleph_0}$ for all $\aleph_0 \leq \kappa \leq \lambda$.
	\end{enumerate}
\end{prop}
\begin{proof}

(1) By Fact \ref{fac: 1-var 2-incons for TP_2} some formula $\varphi(x,y)$ with $|x|=1$ has $\TP_1$, witnessed by $A := (a_{\eta} : \eta \in \omega^{<\omega})$ (as in Definition \ref{def: NTP2}(2)).
Let $(\mathcal{T}, \leq)$ be an arbitrary tree. By compactness we can then find  $A' := \left(a'_{\eta}\right)_{\eta \in \mathcal{T}}$ with $a'_{\eta}$ tuples in $\M$ such that: $\pi_{B}(x) := \{\varphi(x,a'_{\eta}) : \eta \in B\}$ is consistent for every branch $B \subseteq \mathcal{T}$; and for any incomparable nodes $\eta \perp \nu$ in $\mathcal{T}$, $\{\varphi(x,a'_{\eta}), \varphi(x,a'_{\nu})\}$ is inconsistent (using that every finite induced subtree of $\mathcal{T}$ is isomorphic to an induced subtree of $\omega^{<\omega}$, and that $A$ satisfied these conditions). Note that $\{\pi_{B}(x) : B \subseteq \mathcal{T} \textrm{ a branch}\}$ is a family of pairwise inconsistent partial types each of size at most the height of $\mathcal{T}$, all over the set of parameters $A'$ with $|A'|$ at most the size of $\mathcal{T}$.  Hence $f_{T}\left(\kappa,\lambda\right)\geq\lambda^{\kappa,\tr}$ (using Remark \ref{rem: props of tree exp}(1)).
	
	(2) As in (1), by multi-$\TP_1$ and compactness, given any tree $(\mathcal{T}, \leq)$ we can find trees of parameters $(a^{\alpha}_{\eta} : \eta \in \mathcal{T})$ in $\M$ for $\alpha \in \omega$ so that: 
	\begin{itemize}
	\item for any choice of branches $B_{\alpha} \subseteq \mathcal{T}, \alpha \in \omega$, the partial type	$\bigcup_{\alpha \in \omega}\{\varphi_{\alpha}(x, a^{\alpha}_{\eta}) : \eta \in B_{\alpha} \}$ is consistent;
	\item for every $\alpha \in \omega$ and $\eta \perp \nu$ in $\mathcal{T}$, $\{\varphi_{\alpha}(x,a^{\alpha}_{\eta}), \varphi(x,a^{\alpha}_{\nu})\}$ is inconsistent.
	\end{itemize}
	Hence, if the tree $\mathcal{T}$ has height $\leq \kappa$, size $\leq \lambda$ and $\geq \nu$ branches, this gives a set of parameters $A := \bigcup \{a^{\alpha}_{\eta} : \alpha \in \omega, \eta \in \mathcal{T}\}$ with $|A| \leq \aleph_0 \cdot \lambda = \lambda$ with $\geq \nu^{\aleph_0}$ pairwise-inconsistent partial types over it, each of size $ \leq  \aleph_0 \cdot \kappa  = \kappa$.

	Let $\aleph_0 \leq \kappa \leq \lambda$ be arbitrary. Assume first  that $\cof(\lambda^{\kappa, \tr}) > \aleph_0$. Then $\left(\lambda^{\kappa, \tr} \right)^{\aleph_0} = \lambda^{\kappa, \tr} \cdot \sup\{ \nu^{\aleph_0} : \nu < \lambda^{\kappa, \tr} \textrm{ cardinal} \}$ (see e.g.~\cite[Lemma 1.6.15(c)]{holz2009introduction}). It follows (using Remark \ref{rem: props of tree exp}(1)) that for every $\nu < \left(\lambda^{\kappa, \tr} \right)^{\aleph_0} $, there is a tree $\mathcal{T}$ of height $\leq \kappa$ and size $\leq \lambda$ with $\geq \nu_0$ branches so that $\nu_0^{\aleph_0} \geq \nu$. 	Then, by the first paragraph, $f_{T}^1(\kappa, \lambda) \geq \nu_0^{\aleph_0} \geq \nu$. Hence  $f_{T}^1(\kappa, \lambda) \geq \nu_0^{\aleph_0} \geq \left( \lambda^{\kappa, \tr} \right)^{\aleph_0}$.

	Otherwise $\cof(\lambda^{\kappa, \tr}) \leq \aleph_0$. In this case we can find a tree $\mathcal{T}$ of height $\leq \kappa$ and size $\leq \lambda$ with $\geq \lambda^{\kappa, \tr}$ branches (i.e.~the supremum in the definition of $\lambda^{\kappa, \tr}$ is achieved). Namely, we can write $\lambda^{\kappa, \tr} = \sup_{i \in \omega} \nu_i$ for some cardinals $\nu_i < \lambda^{\kappa, \tr}$. By definition of $\lambda^{\kappa, \tr}$, for each $i \in \omega$ there is a tree $\mathcal{T}_i$ of height $\leq \kappa$ and size $\leq \lambda$ with $\geq \nu_i$ branches. We define the tree $\mathcal{T}$ by adding a new root $r$ and, for each $i < \omega$, placing the tree $\mathcal{T}_i$ so that its root $r_i$ is one of the $\aleph_0$-many immediate descendants of $r$. Then the size $|\mathcal{T}|$ of $\mathcal{T}$ is $1 + \sum_{i < \omega} |\mathcal{T}_i| \leq 1+ \aleph_0 \cdot \lambda = \lambda$, the height of $\mathcal{T}$ is at most $1 + \kappa  = \kappa$ and the number of branches of $\mathcal{T}$ is $\geq \sum_{i < \omega} \nu_i = \lambda^{\kappa, \tr}$. By the first paragraph again, we thus get $f^1_{T}(\kappa, \lambda) \geq \left( \lambda^{\kappa, \tr} \right)^{\aleph_0}$.
\end{proof}

Using these observations, we can significantly narrow down the possibilities for $f_{T}(\kappa, 
\lambda)$:
\begin{prop}\label{prop: almost all counting functions} Let $T$ be a countable theory with infinite models. 

\begin{enumerate}
\item $T$ is $\omega$-stable $\Rightarrow$ $f_{T}\left(\kappa,\lambda\right)= f^1_{T}\left(\kappa,\lambda\right) = \lambda$
for all $\lambda\geq\kappa\geq\aleph_{0}$.
\item $T$ is superstable, not $\omega$-stable $\Rightarrow$ $f_{T}\left(\kappa,\lambda\right)= f^1_{T}\left(\kappa,\lambda\right) = \lambda+2^{\aleph_{0}}$
for all $\lambda\geq\kappa\geq\aleph_{0}$.
\item $T$ is stable, not superstable $\Rightarrow$ $f_{T}\left(\kappa,\lambda\right)= f^1_{T}\left(\kappa,\lambda\right) = \lambda^{\aleph_{0}}$
for all $\lambda\geq\kappa\geq\aleph_{0}$.
\item $T$ is supersimple, unstable $\Rightarrow$ $f_{T}\left(\kappa,\lambda\right)= f^1_{T}\left(\kappa,\lambda\right) = \lambda+2^{\kappa}$
for all $\lambda\geq\kappa\geq\aleph_{0}$.
\item $T$ is simple, not supersimple, unstable $\Rightarrow$ $f_{T}\left(\kappa,\lambda\right)= f^1_{T}\left(\kappa,\lambda\right) =  \lambda^{\aleph_{0}}+2^{\kappa}$
for all $\lambda\geq\kappa\geq\aleph_{0}$.
\item $T$ is not simple $\Rightarrow$ $f_{T}^1 \left(\kappa,\lambda\right)\geq\lambda^{\kappa,\tr}$
for all $\lambda \geq \kappa \geq \aleph_0$.
\item $T$ is not simple, not $\NIP$ $\Rightarrow$ $f^1_{T}(\kappa, \lambda) = f_T \left(\kappa,\lambda\right)=\lambda^{\kappa}$
for all $\lambda\geq\kappa\geq\aleph_{0}$.
\item $T$ is NIP,  not stable (= not simple), not multi-order $\Rightarrow$ 
\begin{enumerate}
	\item $f_T^1(\kappa, \lambda) \geq \lambda^{\kappa,\tr}$ for all $\lambda \geq \kappa \geq \aleph_0$;
	\item $f_{T}\left(\kappa,\lambda\right)= f_T^{1}(\kappa, \lambda) = \lambda^{\kappa}$ for all $\lambda^{\kappa}>2^{\kappa}$;
	\item $f_T(\kappa, \lambda) \leq \lambda^{\kappa,\tr} 
	 = \ded \kappa$ for all $\kappa = \lambda$.
\end{enumerate}

\item $T$ is NIP,  multi-order $\Rightarrow$ 
\begin{enumerate}
	\item $f_T^1(\kappa, \lambda) \geq \left( \lambda^{\kappa,\tr} \right)^{\aleph_0}$ for all $\lambda \geq \kappa \geq \aleph_0$;
	\item $f_{T}\left(\kappa,\lambda\right)= f_T^{1}(\kappa, \lambda) = \lambda^{\kappa}$ for all $\lambda^{\kappa}>2^{\kappa}$;
	\item $f_T(\kappa, \lambda) \leq \left( \lambda^{\kappa,\tr}  \right)^{\aleph_0}
	 = \left( \ded \kappa \right)^{\aleph_0}$ for all $\kappa = \lambda$.
\end{enumerate}
\end{enumerate}
\end{prop}
\begin{proof}

\begin{enumerate}
\item By  Remark \ref{rem: obvious properties of f_T(k,l)} we have $f_{T}^1(\kappa, \lambda) \geq \lambda$ for all $\aleph_0 \leq \kappa \leq \lambda$. And if $T$ is $\omega$-stable we have $f_{T}(\kappa, \lambda) \leq f_{T}(\lambda, \lambda) \leq \lambda$ by Fact \ref{fac: classification of usual spectra}(1).
\item By Fact \ref{fac: classification of usual spectra}(2) we have $f_{T}\left(\kappa,\lambda\right)\leq f_{T}\left(\lambda\right)\leq\lambda+2^{\aleph_{0}}$, and $f_{T}^{1}(\aleph_0, \aleph_0) = 2^{\aleph_0}$. Hence, by Remark \ref{rem: obvious properties of f_T(k,l)}, $f_{T}^1\left(\kappa,\lambda\right)\geq\lambda+2^{\aleph_{0}}$ for all $\aleph_0 \leq \kappa \leq \lambda$.

\item Again we have $f_{T}\left(\kappa,\lambda\right)\leq f_{T}\left(\lambda\right)\leq\lambda^{\aleph_{0}}$
for all $\lambda\geq\kappa\geq\aleph_{0}$ by Fact \ref{fac: classification of usual spectra}(3).
On the other hand, assume that $T$ is not superstable. Then it is
not supersimple, so by Fact \ref{fac: known results about f_T(kappa,lambda)}(2)(c) 
we have $f_{T}^{1}(\kappa, \lambda) \geq \lambda^{\aleph_0}$ for all $\aleph_0 \leq \kappa \leq \lambda$. See also \cite[Proposition 4.2]{casanovas1999number} in relation to (1), (2) and (3).

\item As $T$ is unstable and (super) simple (so $\NSOP$), by Shelah's theorem it has $\IP$. Then $f_{T}^1\left(\kappa,\kappa\right)\geq2^{\kappa}$ for all
$\kappa\geq\aleph_{0}$, by Fact \ref{fac: classification of usual spectra}(6) (and  $f_{T}^1(\kappa, \lambda) \geq \lambda$ by Remark \ref{rem: obvious properties of f_T(k,l)}).
On the other hand, by Fact \ref{fac: known results about f_T(kappa,lambda)}(2)(b) 
we have that $f_{T}\left(\kappa,\lambda\right)\leq\lambda+2^{\kappa}$ for all $\kappa, \lambda$.  
By Remark \ref{rem: obvious properties of f_T(k,l)}(2),(3) we conclude
that $f_{T}\left(\kappa,\lambda\right)=\lambda+2^{\kappa}$ for all $\aleph_0 \leq \kappa \leq \lambda$.
\item Again $f_{T}^1\left(\kappa,\kappa\right)\geq2^{\kappa}$ for all $\kappa\geq\aleph_{0}$
as $T$ has $\IP$. On the other hand, as $T$ is simple, by Fact
\ref{fac: known results about f_T(kappa,lambda)}(1)(b) we have $f_{T}\left(\kappa,\lambda\right)\leq\lambda^{\aleph_{0}}+2^{\kappa}$ for all $\kappa, \lambda$.
As $T$ is not supersimple then again by Fact \ref{fac: known results about f_T(kappa,lambda)}(2)(c) 
we have $f_{T}^1\left(\kappa,\lambda\right)\geq\lambda^{\aleph_{0}}$ for all $\aleph_0 \leq \kappa \leq \lambda$.

\item 
Let $T$ be not simple, and $\aleph_0 \leq \kappa \leq \lambda$ arbitrary. Assume first that $T$ has $\TP_2$, then there is a formula $\varphi(x,y)$ with $|x|=1$ that has $2$-$\TP_2$ (Fact \ref{fac: 1-var 2-incons for TP_2}). By compactness we can then find an array $A := (a_{i,j}: (i,j) \in \kappa \times \lambda)$ so that $\{\varphi(x,a_{i,j}) : j \in \lambda\}$ is $2$-inconsistent for each $i \in \kappa$, and $\pi_{f}(x) := \{ \varphi(x,a_{i,f(i)}) : i \in \kappa\}$ is consistent for every $f \in \lambda^{\kappa}$. Then $\{\pi_f(x) : f \in \lambda^{\kappa} \}$ is a family a pairwise inconsistent partial types each of size $\kappa$, all over the set of parameters $A$ with $|A| \leq \lambda$, so $f_{T}(\kappa, \lambda) \geq \lambda^{\kappa}$.

Otherwise, by Fact \ref{fac: TP dichotomy}, $T$ has $\TP_1$, so we get $f_{T}^1 \left(\kappa,\lambda\right)\geq\lambda^{\kappa,\tr}$ by Proposition \ref{prop: multi TP1 lower bound}(1).

\item Let $\lambda\geq\kappa\geq\aleph_{0}$ be arbitrary. If $\lambda^{\kappa}>2^{\kappa}$,
then by Proposition \ref{prop:tree exponent lemma} we have $\lambda^{\kappa}=\lambda^{\kappa,\tr}$,
and as $T$ is not simple by (6) we get $f_{T}^1\left(\kappa,\lambda\right)\geq\lambda^{\kappa,\tr}=\lambda^{\kappa}$.
So assume that $\lambda^{\kappa}\leq 2^{\kappa}$. As $T$ has $\IP$, we
have that $f_{T}^1\left(\kappa,\lambda\right)\geq f_{T}^1\left(\kappa,\kappa\right)\geq2^{\kappa}\geq\lambda^{\kappa}$ by Fact \ref{fac: classification of usual spectra}(6).

\item As $T$ is $\NIP$ unstable, it is also not simple, hence  $f_{T}^1\left(\kappa,\kappa\right)\geq \lambda^{\kappa, \tr}$ by (6). Then (b) holds by Proposition  \ref{prop:tree exponent lemma}.  And (c) holds by Fact \ref{fac: classification of usual spectra}(5) (and Remark \ref{rem: props of tree exp}(3)).

\item As $T$ is $\NIP$ and multi-order, by  Proposition \ref{prop: NIP multi-order implies multi TP1},  $T$ is multi-$\TP_1$ witnessed by some formulas $(\varphi_i(x,y_i))_{i \in \omega}$ with $|x|=1$. Then, by Proposition \ref{prop: multi TP1 lower bound}(2),  $f^1_T(\kappa, \lambda) \geq \left( \lambda^{\kappa,\tr} \right)^{\aleph_0}$ for all $\aleph_0 \leq \kappa \leq \lambda$. Again,  (b) holds by Proposition  \ref{prop:tree exponent lemma}. And (c) holds by Fact \ref{fac: classification of usual spectra}(4).
\end{enumerate}
\end{proof}

By Proposition \ref{prop: almost all counting functions}, Conjecture \ref{conj: Adler}(1) follows from the following (in which case the explicit list of possible functions $f_{T}(\kappa, \lambda)$ is given in Proposition \ref{prop: almost all counting functions}):
\begin{conjecture}\label{conj: numb types NIP}
Let $T$ be a countable theory and $\aleph_0 \leq \kappa \leq \lambda$.
\begin{enumerate}
	\item  If $T$ is NIP, then $f_{T}(\kappa, \lambda) \leq \left( \lambda^{\kappa, \tr} \right)^{\aleph_0}$.
	\item If $T$ is not multi-order, then $f_{T}(\kappa, \lambda) \leq  \lambda^{\kappa, \tr}$.
\end{enumerate}
	(Without loss of generality we may restrict to the case $\lambda^{\kappa} \leq 2^{\kappa}$ --- as for $\lambda^{\kappa} > 2^{\kappa}$ this holds trivially by Proposition \ref{prop:tree exponent lemma}.)
\end{conjecture}


\begin{cor}\label{cor: Adler under GCH}
Conjecture \ref{conj: numb types NIP}, and hence Conjecture \ref{conj: Adler}(1),  holds assuming GCH.
\end{cor}
\begin{proof}
	For any infinite $\kappa \leq \lambda$ with $\lambda^{\kappa} \leq 2^{\kappa}$  we  have $\kappa < \ded \kappa = \kappa^{\kappa, \tr} \leq \lambda^{\kappa, \tr} \leq \lambda ^{\kappa} \leq  2^{\kappa}$. Hence, assuming GCH, $\lambda ^{\kappa} = \lambda^{\kappa, \tr}$, so the bound in Conjecture \ref{conj: numb types NIP} holds trivially.
\end{proof}

\begin{cor}\label{cor: counterex to Adler conj}
Conjecture \ref{conj: Adler}(2) is false.
\end{cor}
\begin{proof}
	By Proposition \ref{prop: almost all counting functions}, if a countable theory $T$ is neither simple, nor NIP, then the function $f_{T}(\kappa, \lambda) = \lambda^{\kappa}$ is maximal for all infinite $\kappa \leq \lambda$. But such a $T$ can be taken to have $\NTP_2$ or not.
\end{proof}
We also ask if $f_{T}(\kappa, \lambda) = f^1_{T}(\kappa, \lambda)$ for all $T, \kappa, \lambda$ (unlike that case of complete types and $f_{T}(\kappa)$, this is not automatic for counting partial types; note that Conjecture \ref{conj: numb types NIP} combined with Proposition \ref{prop: almost all counting functions} implies that this holds for countable $T$). We also mention that a different two-cardinal invariant of first-order theories generalizing $f_T(\kappa)$, namely a function of two
cardinals $\kappa$ and $\lambda$ giving the supremum of the possible number of types over a model of size $\lambda$ that do not fork over a sub-model of size $\kappa$, is considered in \cite{chernikov2016non}.
\bibliographystyle{alpha}
\bibliography{pqNTP2}

\end{document}